\documentclass[12pt]{amsart}
\usepackage{amscd,amssymb,amsmath,amsthm}
\usepackage[dvips]{graphicx,color}
\usepackage{amssymb,amscd,latexsym,epsfig,xypic,hyperref}
\usepackage[mathscr]{euscript}
\usepackage{eufrak}
\newtheorem{Proposition}{Proposition}

\newtheorem{Lemma}{Lemma}
\newtheorem{Theorem}{Theorem}
\newtheorem{Corollary}{Corollary}
\newtheorem{Conjecture}{Conjecture}

\newcommand{\XX}{{\mathsf{X}}}
\newcommand{\proj}{\mathbb{P}}
\newcommand{\Z}{\mathbb{Z}}
\newcommand{\ZZ}{{\mathsf{Z}}}

\newcommand{\rarr}{\rightarrow}

\newcommand{\OO}{\mathcal{O}}
\newcommand{\com}{\mathbb{C}}

\newcommand{\dd}{{\mathbf{d}}}

\newcommand{\ch}{\text{ch}}

\DeclareFontFamily{OT1}{rsfs}{}
\DeclareFontShape{OT1}{rsfs}{n}{it}{<-> rsfs10}{}
\DeclareMathAlphabet{\curly}{OT1}{rsfs}{n}{it}

\newcommand\II{\mathbb I^{\scriptscriptstyle\bullet}}
\newcommand\B{\curly B}

\renewcommand\O{\mathcal O}
\newcommand\F{\mathcal F}
\newcommand\FF{\mathbb F}

\newcommand\LL{\mathbb L}
\newcommand\PP{\mathbb P}
\newcommand\E{\mathcal E}

\renewcommand\S{\mathcal S}
\newcommand\X{\mathcal X}

\newcommand\C{\mathbb C}

\renewcommand\t{\mathfrak t}

\newcommand{\rt}[1]{\stackrel{#1\,}{\rightarrow}}
\newcommand{\Rt}[1]{\stackrel{#1\,}{\longrightarrow}}
\newcommand\To{\longrightarrow}

\newcommand\ito{\ar@{^{ (}->}[r]}
\newcommand{\Into}{\ensuremath{\lhook\joinrel\relbar\joinrel\rightarrow}}

\newcommand\ip{\lrcorner\,}

\renewcommand\_{^{}_}
\newcommand\take{\backslash}
\newcommand\udot{^{\scriptscriptstyle\bullet}}

\newfont{\bigtimesfont}{cmsy10 scaled \magstep5}
\newcommand{\bigtimes}{\mathop{\lower0.9ex\hbox{\bigtimesfont\symbol2}}}

\newcommand\Bl{\operatorname{Bl}}
\newcommand\At{\operatorname{At}}
\newcommand\Ob{\operatorname{Ob}}
\newcommand\fix{\operatorname{fix}}
\newcommand\red{\operatorname{red}}
\newcommand\vir{\operatorname{vir}}
\newcommand\rk{\operatorname{rank}}
\newcommand\tr{\operatorname{tr}}

\newcommand\id{\operatorname{id}}

\newcommand\Hom{\operatorname{Hom}}
\renewcommand\hom{\curly H\!om}
\newcommand\Ext{\operatorname{Ext}}
\newcommand\ext{\curly Ext}

\newcommand\Pic{\operatorname{Pic}}
\newcommand\Proj{\operatorname{Proj}}
\newcommand\Spec{\operatorname{Spec}}
\newcommand\Hilb{\operatorname{Hilb}}

\newcommand\beq[1]{\begin{equation}\label{#1}}
\newcommand\eeq{\end{equation}}
\newcommand\beqa{\begin{eqnarray*}}
\newcommand\eeqa{\end{eqnarray*}}

\makeatletter \@addtoreset{equation}{section} \makeatother

\begin{document}

\title{The Katz-Klemm-Vafa conjecture for $K3$ surfaces}
\author{R. Pandharipande and R.P. Thomas}
\date{December 2015}
\maketitle

\begin{abstract}
We prove the KKV conjecture expressing Gromov-Witten invariants of $K3$
 surfaces in terms of modular forms. 
Our results apply in every genus and for every curve class.
The proof uses the Gromov-Witten/Pairs correspondence for  
$K3$-fibered hypersurfaces of dimension 3
to reduce the KKV conjecture to statements about stable pairs on (thickenings of) $K3$ surfaces.
Using degeneration arguments and new multiple cover results for stable pairs,
 we reduce the KKV conjecture further to the known primitive cases.

Our results yield a new proof of the full Yau-Zaslow formula, establish new Gromov-Witten multiple cover
formulas, and
express the fiberwise
Gromov-Witten partition functions of $K3$-fibered 3-folds in terms of explicit modular forms.

\end{abstract}

\setcounter{tocdepth}{1} 
\tableofcontents

\setcounter{section}{-1}

\baselineskip=18pt

\pagebreak
\section{Introduction}

\subsection{Reduced Gromov-Witten theory} \label{yzc}
Let $S$ be a nonsingular projective $K3$ surface, and let 
 $$\beta \in \text{Pic}(S) = H^2(S,\mathbb{Z}) \cap H^{1,1}(S,\com)$$
be a nonzero effective curve class.
The moduli space $\overline{M}_{g}(S,\beta)$
of genus $g$ stable maps (with no marked points)
has expected dimension
$$\text{dim}^{\vir}_\com\ \overline{M}_{g}(S,\beta) 
= \int_\beta c_1(S) + (\text{dim}_\com(S) -3)(1-g) = g-1\,.$$
However, as the obstruction theory admits a 1-dimensional trivial quotient,
the virtual class $[\overline{M}_{g}(S,\beta)]^{\vir}$  vanishes.
The standard Gromov-Witten theory is trivial.

Curve counting on $K3$ surfaces
 is captured instead by the {\em reduced} Gromov-Witten
theory  constructed first via the twistor family in \cite{brl}.
An algebraic construction following \cite{Beh,BehFan} is given in 
\cite{gwnl}. The reduced class 
$$[\overline{M}_{g}(S,\beta)]^{\red} \in A_g(\overline{M}_{g}(S,\beta), \mathbb{Q})$$
has dimension $g$. Let 
$\lambda_g$ be the top Chern class of the rank $g$ Hodge bundle{\footnote{The
Hodge bundle is pulled-back from $\overline{M}_g$ if $g\geq 2$.
See \cite{FP,GP} for a discussion of Hodge classes in
Gromov-Witten theory.}}
$$\mathbb{E}_g \rightarrow \overline{M}_g(S,\beta)$$
with fiber $H^0(C,\omega_C)$ over the moduli point 
$$[f:C\rightarrow S]\in \overline{M}_g(S,\beta)\,.$$
The reduced Gromov-Witten integrals of $S$,
\begin{equation}\label{veq}
R_{g,\beta}(S) = \int_{[\overline{M}_{g}(S,\beta)]^{\red}} (-1)^g \lambda_g   
\  \in \mathbb{Q}\,,
\end{equation}
 are
well-defined. 
Under deformations of $S$ for which $\beta$ remains a $(1,1)$-class,
the integrals \eqref{veq} are invariant. 

Let $\epsilon:\mathcal{X}\rightarrow (B,b)$ be 
a fibration of $K3$ surfaces over a
base $B$ with special fiber 
$$\mathcal{X}_b\cong S\ \ \text{over}\ \ b\in B\,.$$
Let $U\subset B$ be an open set containing $b\in B$
over which the local system of second cohomology 
$R^2\epsilon_*(\mathbb{Z})$ is trivial. 
The class $\beta\in \text{Pic}(S)$ determines a
{\em local Noether-Lefschetz locus} 
$$\text{NL}(\beta) \subset U$$
defined as the subscheme where $\beta$ remains a 
$(1,1)$-class.{\footnote{While $\text{NL}(\beta)$ 
is locally defined on $U$ by a single equation, 
the locus may be degenerate (equal to all of $U$).}}

Let $(\Delta,0)$ be a nonsingular
quasi-projective curve with special point $0\in\Delta$.
The integral
$R_{g,\beta}(S)$ computes the local contribution of 
$S$ to the standard Gromov-Witten theory of
 every $K3$-fibered 3-fold 
\begin{equation} \label{veqq}
\epsilon:T\to(\Delta,0)
\end{equation} with special fiber $S$ 
and
local Noether-Lefschetz locus $\text{NL}(\beta)\subset \Delta$
equal to the reduced point $0\in\Delta$, see \cite{gwnl}.

\subsection{Curve classes}

 The second cohomology of $S$ is a rank 22 lattice
with intersection form 
\begin{equation}\label{ccet}
H^2(S,\mathbb{Z}) \stackrel{\sim}{=} U\oplus U \oplus U \oplus E_8(-1) \oplus E_8(-1)\,,
\end{equation}
where
$$U
= \left( \begin{array}{cc}
0 & 1 \\
1 & 0 \end{array} \right)$$
and 
$$
E_8(-1)=  \left( \begin{array}{cccccccc}
 -2&    0 &  1 &   0 &   0 &   0 &   0 & 0\\
    0 &   -2 &   0 &  1 &   0 &   0 &   0 & 0\\
     1 &   0 &   -2 &  1 &   0 &   0 & 0 &  0\\
      0  & 1 &  1 &   -2 &  1 &   0 & 0 & 0\\
      0 &   0 &   0 &  1 &   -2 &  1 & 0&  0\\
      0 &   0&    0 &   0 &  1 &  -2 &  1 & 0\\ 
      0 &   0&    0 &   0 &   0 &  1 &  -2 & 1\\
      0 & 0  & 0 &  0 & 0 & 0 & 1& -2\end{array}\right)$$
is the (negative) Cartan matrix. The intersection form \eqref{ccet}
is even.

The {\em divisibility} $m(\beta)$ is
the maximal positive integer dividing the lattice
element $\beta\in H^2(S,\mathbb{Z})$.
If the divisibility is 1,
$\beta$ is {\em primitive}.
Elements with
equal divisibility and norm square are equivalent up to orthogonal transformation 
 of $H^2(S,\mathbb{Z})$, see \cite{CTC}.
By straightforward deformation arguments using the Torelli theorem for
$K3$ surfaces, $R_{g,\beta}(S)$ depends, for effective classes, 
 {\em only} on the divisibility $m(\beta)$ and the norm square
$$\langle \beta,\beta \rangle = \int_S \beta^2\,.$$
 We will omit the argument $S$ in the
notation,
$$R_{g,\beta}=R_{g,\beta}(S)\,.$$

\subsection{BPS counts} \label{bpsc}
The KKV conjecture concerns BPS counts associated to the Hodge integrals
\eqref{veq}. Throughout this paper we let
$$\alpha\in \text{Pic}(S)$$ denote a nonzero class which is both \emph{effective} and
\emph{primitive}.
The Gromov-Witten potential $F_{{\alpha}}(\lambda,v)$ 
for classes proportional
to ${\alpha}$
is 
\begin{equation} \label{hbb}
{F}_{{\alpha}}=
\sum_{g\geq 0}\   \sum_{m>0}\   R_{g,m\alpha} \ \lambda^{2g-2} 
v^{m{\alpha}}.
\end{equation}
The
BPS counts $r_{g,m\alpha}$ are uniquely defined 
by the following equation:
\begin{equation} \label{kxxz}
F_\alpha =   \ \ \ \sum_{g\geq 0}  \ \sum_{m>0} \
 r_{g,m\alpha} \ \lambda^{2g-2} \sum_{d>0}
\frac{1}{d}\left( \frac{\sin
({d\lambda/2})}{\lambda/2}\right)^{2g-2} v^{dm\alpha}. 
\end{equation}
Equation \ref{kxxz} defines BPS counts  
for both primitive and
divisible classes.

The string theoretic calculations of Katz, Klemm and Vafa \cite{kkv}
via heterotic duality yield two conjectures.

\begin{Conjecture} \label{xxx1}
The BPS count $r_{g,\beta}$ depends upon $\beta$ only through the norm
square
 $\langle \beta, \beta \rangle$.
\end{Conjecture}

Conjecture \ref{xxx1} is rather surprising from the point
of view of Gromov-Witten theory. From the definition,
 the invariants $R_{g,\beta}$ and $r_{g,\beta}$
depend
upon both the divisibility $m$ of $\beta$ and the norm
square  $\langle \beta,\beta
\rangle $.
Assuming the validity of
Conjecture \ref{xxx1}, let $r_{g,h}$ denote the BPS count associated
to a class $\beta$ of arithmetic genus $h$,
$$\langle \beta, \beta \rangle = 2h-2\,.$$

\begin{Conjecture} \label{xxx2}
The BPS counts $r_{g,h}$ are uniquely determined by the
following equation:
$$\sum_{g\geq 0} \sum_{h\geq 0} (-1)^g r_{g,h}(y^{\frac{1}{2}} - y^{-\frac{1}{2}})^{2g}q^h =
\prod_{n\geq 1} \frac{1}{(1-q^n)^{20} (1-yq^n)^2 (1-y^{-1}q^n)^2}.$$
\end{Conjecture}

As a consequences of Conjecture \ref{xxx2}, $r_{g,h}\in \mathbb{Z}$,
 $r_{g,h}$ vanishes if $g>h$, and
$$r_{g,g}=(-1)^g (g+1)\,.$$
The integrality of $r_{g,h}$ and the vanishing for high $g$ (when $h$ is fixed)
fit in the framework of the Gopakumar-Vafa conjectures. 
The first values are tabulated below:

\vspace{18pt}

\begin{tabular}{|c||ccccc|}
        \hline
\textbf{}
$r_{g,h}$&    $h= 0$ & 1  & 2 & 3 & 4 \\
        \hline \hline
$g=0$ & $1$ & $24$ & $324$ & 
$3200$ &$25650$  \\
1      &  & $-2$ & 
$-54$ & $-800$  & $-8550$      \\
2      & & & $3$ & 
$88$ & $1401$       \\
3      & &  & 
 & $-4$  & $-126$       \\
4      &  &  & 
 &   & 5       \\
       \hline
\end{tabular}

\vspace{18pt}

The right side of Conjecture \ref{xxx2} is related to the 
generating series of Hodge numbers of the Hilbert schemes of points
$\text{Hilb}^n(S)$.
The genus 0 specialization of Conjecture \ref{xxx2} recovers the
Yau-Zaslow formula
$$\sum_{h\geq 0} r_{0,h} q^h = \prod_{n\geq 1} \frac{1}{(1-q^n)^{24}}$$
related to the Euler characteristics of $\text{Hilb}^n(S)$.

The main result of the present paper is a proof of the KKV conjecture
for all genera $g$ and all $\beta \in H_2(S,\mathbb{Z})$.

\begin{Theorem} \label{kkbb}
 The BPS count $r_{g,\beta}$ depends upon $\beta$
only through \\
$\langle \beta, \beta \rangle=2h-2$, and the 
Katz-Klemm-Vafa formula holds:
$$\sum_{g\geq 0} \sum_{h\geq 0} (-1)^g r_{g,h}(y^{\frac{1}{2}} - y^{-\frac{1}{2}})^{2g}q^h =
\prod_{n\geq 1} \frac{1}{(1-q^n)^{20} (1-yq^n)^2 (1-y^{-1}q^n)^2}.$$
\end{Theorem}

\vspace{8pt}

\subsection{Past work}
The enumerative geometry of curves on
$K3$ surfaces has been studied since the 1995 paper
of Yau and Zaslow \cite{yauz}.
A mathematical approach to the genus 0 Yau-Zaslow formula 
can be found in \cite{beu,xc,fgd}.
The Yau-Zaslow formula was proven for
primitive classes $\beta$ by Bryan and Leung \cite{brl}. 
The divisibility 2 case was settled
by Lee and Leung 
 in \cite{ll}. A complete proof of the
Yau-Zaslow formula for all divisibilities was given in 
\cite{KMPS}. 
Our approach to 
Theorem 1 
provides a completely new proof of the  Yau-Zaslow
formula for all divisibilities (which avoids the mirror calculation of the STU model
and the Harvey-Moore identity used in \cite{KMPS}).

Conjecture 2 for primitive classes $\beta$ is connected
to Euler characteristics of the moduli spaces of
stable pairs on $K3$ surfaces by the GW/P correspondence of \cite{PT1,PT2}.
A proof of Conjecture 2 for primitive classes is given in \cite{MPT}
relying upon the Euler characteristic calculations of
Kawai and Yoshioka \cite{ky}.
For cases where  $g>0$ and $\beta$ is not primitive, 
Theorem 1 is a new result.

The cases understood before are very special.
If the genus is 0,
the calculation can be moved via Noether-Lefschetz theory to the genus 0
Gromov-Witten theory of
toric varieties using the hyperplane principle for $K3$-fibrations \cite{KMPS}.
If the class $\beta$ is irreducible, the moduli space of stable pairs is 
nonsingular \cite{ky}, and the calculation can be moved to  stable pairs \cite{MPT}.
The difficulty for positive genus imprimitive curves -- which are essentially {\em all} curves -- lies in
the complexity of the moduli spaces.
There is no effective hyperplane
principle in higher genus,  and the moduli spaces of stable maps and stable pairs are both
highly singular. 

Y. Toda has undertaken a parallel study of the  Euler characteristic
(following Joyce)
of
the moduli spaces of stable pairs on $K3$ surfaces \cite{toda}.
His results -- together with further multiple cover conjectures 
which are still open -- are
connected to an Euler characteristic version of the KKV formula.  
Our methods and results essentially concern the
virtual class and thus do not imply (nor are implied by) Toda's paper \cite{toda}.
In fact, the motivation of \cite{toda} was the original KKV conjecture
proven here.

\subsection{GW/P correspondence}\label{gwp4}
The Katz-Klemm-Vafa formula concerns integrals over the moduli space of stable maps. Our strategy is to transform the calculation to the theory of stable pairs.
Let $\widetilde{\proj^2 \times \proj^1}$ be the blow up
of $\proj^2\times \proj^1$ in a point.
Consider a nonsingular anticanonical Calabi-Yau 3-fold hypersurface,
$$\XX \subset \widetilde{\proj^2 \times \proj^1} \times \proj^1\,.$$
The projection onto the last factor,
\begin{equation}\label{pxxp1}
\pi_3: \XX \rightarrow \proj^1 \,, 
\end{equation}
 determines a 1-parameter family of anticanonical
$K3$ surfaces in $\widetilde{\proj^2 \times \proj^1}$.
The interplay between the Gromov-Witten, stable pairs, and
Noether-Lefschetz theories for the family $\pi_3$ will be used
to transform Theorem 1 to nontrivial claims about of the moduli of sheaves
on $K3$-fibrations.

The KKV formula (conjecturally) evaluates the integrals $R_{g,\beta}$ occurring
in the reduced Gromov-Witten theory of a $K3$ surface $S$.
If we view $S$ as a fiber of $\pi_3$, then 
$$\beta\in\text{Pic}(S) \subset H^2(S,\mathbb{Z}) \cong H_2(S,\mathbb{Z})$$
determines a fiber class in $H_2(\XX,\mathbb{Z})$ by push-forward.
We consider both the Gromov-Witten and stable pairs invariants of
$\XX$ in $\pi_3$-fiber curve classes.
The GW/NL correspondence of \cite{gwnl}
precisely relates the Gromov-Witten theory of $\XX$ in 
fiber classes with the Noether-Lefschetz numbers of the
family and the integrals $R_{g,\beta}$.
We prove a  P/NL correspondence 
which establishes a parallel relationship between the 
stable pairs theory of $\XX$ in 
fiber classes with the same Noether-Lefschetz numbers 
and the invariants $\widetilde{R}_{n,\beta}$ defined as follows.

\subsection{Stable pairs and $K3$ surfaces} \label{spaks}
Let $S$ be a nonsingular projective $K3$ surface with
a nonzero effective curve class $\beta\in \text{Pic}(S)$.
We define here the
stable pairs analogue $\widetilde{R}_{n,\beta}$ of 
the reduced Gromov-Witten invariants $R_{g,\beta}$ of $S$.

For Gromov-Witten invariants, we defined 
$R_{g,\beta}$ directly \eqref{veq} in terms of the moduli of stable maps
to $S$ and observed the result calculated the contributions of the special
fiber
$S$ to the Gromov-Witten theories of all
families \eqref{veqq} appropriately transverse to the local Noether-Lefschetz
locus corresponding to  $\beta$.
The geometry of stable pairs is
more subtle. While the support of a stable pair may probe thickenings of
the special fiber
$S\subset T$ of \eqref{veqq}, the image of a
 stable map does not. As a result, we
will {\em define}
$\widetilde{R}_{n,\beta}$ via the geometry of appropriately
transverse families of $K3$ surfaces. Later in Section \ref{ped},
we will see how to
define $\widetilde{R}_{n,\beta}$ via the intrinsic geometry of $S$.

Let $\alpha\in \text{Pic}(S)$ be a nonzero class
which is both effective and primitive.
Let $T$ be a nonsingular 3-dimensional quasi-projective
 variety,
$$\epsilon: T \rightarrow (\Delta,0)\,,$$
fibered in $K3$ surfaces over a pointed
curve $(\Delta,0)$
satisfying:
\begin{enumerate}
\item[(i)] $\Delta$ is a nonsingular quasi-projective curve,
\item[(ii)] $\epsilon$ is smooth, projective, and $\epsilon^{-1}(0) \stackrel{\sim}{=} S$,
\item[(iii)] the local 
Noether-Lefschetz locus $\text{NL}(\alpha)\subset \Delta$ 
corresponding to
the class $\alpha \in \text{Pic}(S)$ is the
{\em reduced} point $0\in \Delta$.
\end{enumerate}

The class $\alpha \in \text{Pic}(S)$ is {\em $m$-rigid}
with respect to the family $\epsilon$ if the
following further condition is satisfied:
\begin{enumerate}
\item[$(\star)$] for every
 {\em effective} decomposition{\footnote{An effective
decomposition requires all parts $\gamma_i$ to be effective
divisors.}} 
$$m\alpha=\sum_{i=1}^l \gamma_i
\in \text{Pic}(S)\,,$$
the  local Noether-Lefschetz locus $\text{NL}(\gamma_i)
\subset \Delta$ corresponding to
each  class $\gamma_i \in \text{Pic}(S)$ is the
{\em reduced} point $0\in \Delta$.
\end{enumerate}
Let $\text{Eff}({m}\alpha) \subset \text{Pic}(S)$
denote the subset of effective summands of 
$m\alpha$.  Condition $(\star)$ implies (iii).

Assume $\alpha$ is ${m}$-rigid
with respect to the family $\epsilon$.
By property $(\star)$,
there is a compact, open, and closed component  
$$P_n^\star(T,\gamma) \subset P_n(T,\gamma)$$ 
parameterizing 
 stable pairs{\footnote{For any class $\gamma\in \text{Pic}(S)$,
we denote the push-forward to $H_2(T,\mathbb{Z})$ also
by $\gamma$.
Let $P_n(T,\gamma)$ be the moduli space of stable pairs
of Euler characteristic $n$ and class $\gamma \in H_2(T,\mathbb{Z})$.}}
supported set-theoretically over the point 
$0\in \Delta$ for {\em every} effective summand $\gamma\in 
\text{Eff}({m}\alpha)$.

\vspace{7pt}
\noindent{\bf Definition.} 
{\em Let $\alpha\in \text{\em Pic}(S)$ be a nonzero class
which is both effective and primitive.
Given a family
$\epsilon: T \rightarrow (\Delta,0)$ 
satisfying conditions (i), (ii), and $(\star)$ for $m\alpha$, let
 \begin{multline}
\label{fbbb}
\sum_{n\in \mathbb{Z}}
\widetilde{R}_{n,m\alpha}(S)\ q^n  = \\
\text{\em Coeff}_{v^{m\alpha}} \left[
\log\left(1 + \sum_{n\in \mathbb{Z}} \sum_{\gamma\in \text{\em Eff}({m}\alpha)} 
q^n v^{\gamma}
\int_{[P_n^\star(T,\gamma)]^{\vir}} 1\right)\right]\,.
\end{multline}}
\vspace{7pt}

In Section \ref{depped}, we will prove $\widetilde{R}_{n,m\alpha}$
depends {\em only} upon $n$, $m$ and 
$\langle \alpha,\alpha\rangle$, and {\em not}
upon $S$ nor the family $\epsilon$. 
The dependence result is nontrivial and requires 
new techniques to establish. 
The existence{\footnote{Constructions are
given in Section \ref{c123}.}} of ${m}$-rigid families $\epsilon$
for suitable $S$ and $\alpha$ (primitive with fixed
$\langle \alpha,\alpha \rangle$) then
defines
$\widetilde{R}_{n,m\alpha}$ for all $m$.

The appearance of the logarithm in 
\eqref{fbbb} has a simple explanation.
The Gromov-Witten invariants $R_{g,m\alpha}$ are defined
via moduli spaces of stable maps with {\em connected} domains. 
Stable pairs invariants count sheaves with 
possibly {\em disconnected} support curves. The logarithm
accounts for the difference.

The stable pairs potential $\widetilde{F}_{{\alpha}}(q,v)$ 
for classes proportional
to the primitive class ${\alpha}$
is 
\begin{equation} \label{hbbb}
\widetilde{F}_{{\alpha}}=
\sum_{n\in \mathbb{Z}}\   \sum_{m>0}\   \widetilde{R}_{n,m\alpha} \ 
q^{n} 
v^{m{\alpha}}\,.
\end{equation}
By the properties of $\widetilde{R}_{n,m\alpha}$, the potential
$\widetilde{F}_\alpha$
depends only upon the norm square $\langle \alpha,\alpha\rangle$.

Via the correspondences with Noether-Lefschetz theory, we prove that the
GW/P correspondence \cite{PP,PT1} for suitable 3-folds
fibered in $K3$ surfaces implies the
following basic result for the 
potentials \eqref{hbb} and \eqref{hbbb}.

\begin{Theorem} \label{vvee}
After the variable change $-q=e^{i\lambda}$, the
potentials are equal:
$$F_{\alpha}(\lambda, v) =  \widetilde{F}_{\alpha}(q,v)$$
\end{Theorem}

In order to show the variable change of Theorem \ref{vvee}
is well defined, a rationality result is required. 
In Section \ref{mulcov}, we prove
for all $m>0$, 
$$\left[ \widetilde{F}_{\alpha}\right]_{v^{m\alpha}} = 
\sum_{n\in \mathbb{Z}}\     \widetilde{R}_{n,m\alpha} \ 
q^{n}$$
is the 
Laurent expansion of a rational function in $q$.

\subsection{Multiple covers}
While Theorem \ref{vvee} transforms
Theorem \ref{kkbb} to a statement about stable pairs,
the evaluation must still be carried out.

The logarithm in definition \eqref{fbbb}
plays no role for the $v^\alpha$ coefficient,
$$\left[ \widetilde{F}_{\alpha}\right]_{v^{\alpha}} =
\sum_{n\in \mathbb{Z}}  q^n 
\int_{[P_n^\star(T,\alpha)]^{\vir}} 1\,.$$
If $\alpha$ is irreducible (which can be
assumed by deformation invariance), $P_n^\star(T,\alpha)$ 
is a nonsingular variety of
dimension $\langle\alpha,\alpha\rangle+n +1$. 
If $T$ is taken to be Calabi-Yau, the obstruction theory on 
$P_n^\star(T,\alpha)$
is self-dual and
$$
\sum_{n\in \mathbb{Z}}  q^n 
\int_{[P_n(T,\alpha)]^{\vir}} 1
= \sum_{n\in \mathbb{Z}} q^n (-1)^{\langle \alpha,\alpha \rangle+n+1}
\ \chi_{\text{top}} \left( P^\star_n(T,\alpha)\right)
\,.$$
The
Euler characteristic calculations of
Kawai and Yoshioka \cite{ky} then imply
the stable pairs KKV prediction for primitive
$\alpha\in \text{Pic(S)}$. A detailed
discussion can be found in Appendix C of
\cite{PT2}.

In order to prove the KKV conjecture for 
$\left[ \widetilde{F}_{\alpha}\right]_{v^{m\alpha}}$ for all $m>1$,
we find
new multiple cover formulas for stable pairs on $K3$ surfaces.
In fact, the multiple cover structure implicit in the KKV
formula is much more natural on the stable pairs side.

By degeneration arguments and deformation to the normal cone, we reduce 
 the stable pairs multiple cover formula to a
 calculation on the trivial $K3$-fibration $S\times\C$,
 where $\C^*$-localization applies. 
A crucial point here is a vanishing result:
 for each $k$ only the simplest $k$-fold multiple covers contribute --
those stable pairs which are a trivial $k$-times thickening in the $\C$-direction of a stable pair on $S$. 
The moduli space of such trivial thickenings 
is isomorphic to 
 the moduli space of stable pairs supported on $S$. 
This simple geometric relationship provides the key to 
the stable pairs multiple cover formula.

\subsection{Guide to the proof}
The main
steps in our proof of the Katz-Klemm-Vafa formula are summarized as follows:

\begin{enumerate}

\item[(i)] We express the Gromov-Witten invariants of the 
anticanonical hypersurface, 
$$\XX \subset \widetilde{\proj^2\times\proj^1}\times \proj^1\,,$$ 
in terms of the Noether-Lefschetz numbers of $\pi_3$ and the reduced invariants $R_{g,\beta}$ via the GW/NL correspondence.

\item[(ii)] We express the stable pairs invariants of $\XX$ in terms of the Noether-Lefschetz numbers of $\pi_3$ and the stable pairs invariants $\widetilde{R}_{n,\beta}$ via the P/NL correspondence.

\item[(iii)] The GW/P conjecture, proved for the complete intersection $\XX$ in 
\cite{PP}, relates the Gromov-Witten and stable pairs invariants of the 3-fold
$\XX$.

\item[(iv)] 
By inverting the relations (i) and (ii) and using the correspondence (iii),
we establish the equivalence between the sets of numbers $R_{g,\beta}$ and $\widetilde{R}_{n,\beta}$  stated in Theorem \ref{vvee}.

\item[(v)] The invariant $\widetilde{R}_{n,\beta}(S)$ is  defined
via an appropriately transverse family 
$$\epsilon: T \rightarrow (\Delta,0)\,, \ \ \ 
\epsilon^{-1}(0)\stackrel{\sim}{=} S\,. $$
Degenerating the total space $T$ to the normal cone of $S\subset T$, we 
reduce $\widetilde{R}_{n,\beta}(S)$
to a calculation of stable pairs integrals over 
a rubber target. After further geometric arguments,
 the calculation is expressed in terms of
 the reduced  stable pairs invariants 
of the trivial $K3$-fibration $S\times\PP^1$.
A careful analysis of several different obstruction theories is required here.

\item[(vi)] By $\C^*$-localization on $S\times\PP^1$,
we reduce further to a calculation on the moduli space of $\C^*$-fixed stable pairs on $S\times\C$.

\item[(vii)] We prove a vanishing result: 
for each $k$ only the simplest $k$-fold multiple covers contribute. 
We only need calculate the contributions of stable pairs which are a trivial $k$-times thickening (in the $\C$-direction) of a stable pair scheme-theoretically
supported
on $S$.

\item[(viii)] The resulting moduli spaces are isomorphic to $P_n(S,\beta)$, the moduli space of stable pairs on $S$.

\item[(ix)] The resulting integral is calculated in \cite{KT1, KT2} in 
terms of universal formulae in topological constants. In particular, the result does not depend on the divisibility of $\beta$.

\item[(x)] We may therefore assume  $\beta$ to be primitive, and moreover,
by deformation invariance, to be irreducible.  The moduli space
$P_n(S,\beta)$ is then nonsingular. The integrals 
$\widetilde{R}_{n,\beta}(S)$
can be expressed in terms of those evaluated by Kawai-Yoshioka, as explained in \cite{MPT,PT2}.
\end{enumerate}

The paper starts with a discussion of Noether-Lefschetz theory for Gromov-Witten invariants
of $K3$-fibrations. The GW/NL correspondence of \cite{gwnl} and Borcherds' results are
reviewed in Section \ref{nnll}.  A crucial property of the family \eqref{pxxp1}
is established in Proposition \ref{bbh} 
of   Section \ref{tttt}: the BPS states and the Noether-Lefschetz numbers for the
family \eqref{pxxp1} uniquely determine {\em all} the  integrals $R_{g,\beta}$
in the reduced Gromov-Witten theory of $K3$ surfaces.
The result follows by finding a triangularity in the GW/NL correspondence.

Theorem \ref{vvee} constitutes half of our proof of the KKV conjecture. 
In Section \ref{t22}, we prove Theorem \ref{vvee} {\em assuming} the P/NL correspondence.
In fact, Theorem \ref{vvee} is an easy consequence of the GW/NL correspondence,
the P/NL correspondence, and the invertibility established in Proposition \ref{bbh}.
The precise statement of the P/NL correspondence is given in Section \ref{spnl}, but the
proof is presented later in Section \ref{pnlc}.

Sections \ref{k3tcl}--\ref{pnlc} mainly concern the geometry of the 
moduli of stable pairs on $K3$ surfaces and
$K3$-fibrations. The first topic is a detailed
study of the trivial fibration $S\times \C$. In Sections 
\ref{k3tcl} and \ref{van},
an analysis of the perfect obstruction theory of the
$\C^*$-fixed loci of the moduli space of stable pairs on 
$S\times \C$ is presented. We find only the simplest
$\C^*$-fixed stable pairs have nonvanishing contributions.
Moreover, these contributions
directly yield multiple cover formulas. The move from Gromov-Witten
theory to stable pairs was made precisely to exploit
the much simpler multiple cover structure on  the sheaf theory side.

The main results
 of Sections \ref{relt} and \ref{mulcov} concern the expression of
$\widetilde{R}_{n,\beta}$ in terms of the stable pair theory
of $S\times \C$. A careful study of the obstruction theory
is needed. The outcome is a multiple cover formula for
$\widetilde{R}_{n,\beta}$. 

After we establish the P/NL correspondence for the family $\XX$ 
in Section \ref{pnlc}, the proof of the Katz-Klemm-Vafa conjecture
is completed in Section \ref{kkvcon} by transforming the
multiple cover formula to the Gromov-Witten invariants $R_{g,\beta}$.
As a consequence of the KKV formula, the  Gromov-Witten theory of
$K3$-fibrations in vertical classes can be effectively computed.
As an example, the classical pencil of quartic $K3$ surfaces
is treated in Section \ref{qk3}.

A summary of our notation for the various Gromov-Witten and
stable pairs invariants for $K3$ surfaces and $K3$-fibrations
is given in Appendix A. Appendix B contains a discussion of
degenerations of $\XX$ needed for the Gromov-Witten/Pairs
correspondence of \cite{PP}. 
Appendix C
contains  results about
cones, the Fulton total Chern class, and virtual cycles.

\subsection{Acknowledgements}
We thank B. Bakker, J. Bryan, I. Dolgachev, 
B. Fantechi, G. Farkas, T. Graber, S. Katz, A. Klemm, M. Kool,
R. Laza, A. Marian, D. Maulik, G. Oberdieck, A. Oblomkov, 
A. Okounkov, A. Pixton, E. Scheidegger, and V. Shende
 for many conversations over the years related to
the reduced Gromov-Witten theory of $K3$ surfaces. Thanks also to
two anonymous referees for their thorough reading and useful comments.
We are very grateful to A. Kresch for his help with the cone results
in the Appendix.

R.P. was partially supported by grants SNF-200021-143274 and ERC-2012-AdG-320368-MCSK.
R.T. was partially supported by EPSRC programme grant EP/G06170X/1.
We thank the {\em Forschungsinstitut 
f\"ur Mathematik} at ETH Z\"urich and the {\em Centro Stefano Franscini} in Ascona
for support. Crucial progress was made during discussions at Imperial College in the
fall of 2012.

The paper was completed in April 2014. Further developments in 
 the Gromov-Witten and stable pairs theories of $K3$ geometries  
can be found in \cite{KKP}, where the motivic invariants are considered, and \cite{ObPan},
where the 3-fold $K3\times E$ is studied.

\section{Noether-Lefschetz theory}
\label{nnll}
\subsection{Lattice polarization} \label{lpol}
Let $S$ be a nonsingular $K3$ surface.
A primitive class 
$L\in \text{Pic}(S)$ is a {\em quasi-polarization}
if
$$\langle L,L \rangle >0  \ \ \ \text{and} \ \  \ \langle L,[C]\rangle
 \geq 0 $$
for every curve $C\subset S$.
A sufficiently high tensor power $L^n$
of a quasi-polarization is base point free and determines
a birational morphism
$$S\rightarrow \widetilde{S}$$
contracting A-D-E configurations of $(-2)$-curves on $S$.
Hence, every quasi-polarized $K3$ surface is algebraic.

Let $\Lambda$ be a fixed rank $r$  
primitive{\footnote{A sublattice
is primitive if the quotient is torsion free.}}
sublattice
\begin{equation*} 
\Lambda \subset U\oplus U \oplus U \oplus E_8(-1) \oplus E_8(-1)
\end{equation*}
with signature $(1,r-1)$, and
let 
$v_1,\ldots, v_r \in \Lambda$ be an integral basis.
The discriminant is
$$\Delta(\Lambda) = (-1)^{r-1} \det
\begin{pmatrix}
\langle v_{1},v_{1}\rangle & \cdots & \langle v_{1},v_{r}\rangle  \\
\vdots & \ddots & \vdots \\
\langle v_{r},v_{1}\rangle & \cdots & \langle v_{r},v_{r}\rangle 
\end{pmatrix}\,.$$
The sign is chosen so $\Delta(\Lambda)>0$.

A {\em $\Lambda$-polarization} of a $K3$ surface $S$   
is a primitive embedding
$$j: \Lambda \rightarrow \mathrm{Pic}(S)$$ 
satisfying two properties:
\begin{enumerate}
\item[(i)] the lattice pairs 
$\Lambda \subset U^3\oplus E_8(-1)^2$ and
$\Lambda\subset 
H^2(S,\mathbb{Z})$ are isomorphic
 via an isometry which restricts to the identity on $\Lambda$,
 \item[(ii)]
$\text{Im}(j)$ contains
a {quasi-polarization}. 
\end{enumerate}
By (ii), every $\Lambda$-polarized $K3$ surface is algebraic.

The period domain $M$ of Hodge structures of type $(1,20,1)$ on the lattice 
$U^3 \oplus E_8(-1)^2$   is
an analytic open set
 of the 20-dimensional  nonsingular isotropic 
quadric $Q$,
$$M\subset Q\subset \proj\big(    (U^3 \oplus E_8(-1)^2 )    
\otimes_\Z \com\big)\,.$$
Let $M_\Lambda\subset M$ be the locus of vectors orthogonal to 
the entire sublattice $\Lambda \subset U^3 \oplus E_8(-1)^2$.

Let $\Gamma$ be the isometry group of the lattice 
$U^3 \oplus E_8(-1)^2$, and let
 $$\Gamma_\Lambda \subset \Gamma$$ be the
subgroup  restricting to the identity on $\Lambda$.
By global Torelli,
the moduli space $\mathcal{M}_{\Lambda}$ 
of $\Lambda$-polarized $K3$ surfaces 
is the quotient
$$\mathcal{M}_\Lambda = M_\Lambda/\Gamma_\Lambda\,.$$
We refer the reader to \cite{dolga} for a detailed
discussion.

Let $\widetilde{S}$ be a $K3$ surface with A-D-E singularities, and let
$$\tilde{j}: \Lambda \rightarrow \mathrm{Pic}(\widetilde{S})$$ 
be a primitive embedding. Via pull-back along
the desingularization,
$$S\rightarrow \widetilde{S},$$
we obtain a composition $j: \Lambda \rightarrow  \mathrm{Pic}({S})$.
If $(S,j)$ satisfies (i) and (ii), 
we define  $(\widetilde{S},\tilde{j})$ to be
 a $\Lambda$-polarized {\em singular} $K3$ surface. Then
$(S,j)$ is 
a $\Lambda$-polarized nonsingular
$K3$ surface canonically associated to $(\widetilde{S},\tilde{j})$.

\subsection{Families} \label{fams}

Let $X$ be a nonsingular 
projective 3-fold  equipped with
line bundles
$$L_1, \ldots, L_r  \ \rightarrow X$$
 and a map
 $$\pi: X \rarr C$$
to a nonsingular complete curve.

The tuple $(X,L_1,\ldots, L_r, \pi)$ is a
{\em 1-parameter family of $\Lambda$-polarized
$K3$ surfaces}
if 
\begin{enumerate}
\item[(i)] the fibers $(X_\xi, L_{1,\xi}, \ldots, L_{r,\xi})$
are $\Lambda$-polarized $K3$ surfaces with at worst {\em a single
nodal singularity}
 via
$$v_i \mapsto L_{i,\xi}$$
for every $\xi\in C$,
\item[(ii)] there exists a $\lambda^\pi\in \Lambda$
which is a quasi-polarization of all fibers of $\pi$ 
simultaneously.
\end{enumerate}
The family $\pi$ yields a morphism,
$$\iota_\pi: C \rarr \mathcal{M}_{\Lambda}\,,$$
to the moduli space of $\Lambda$-polarized $K3$ surfaces.

Let $\lambda^{\pi}= \lambda^\pi_1 v_1+\dots +\lambda^\pi_r v_r$.
A vector $(d_1,\ldots,d_r)$ of integers is {\em positive} if
$$\sum_{i=1}^r \lambda^\pi_i d_i >0\,.$$
If $\beta \in \text{Pic}(X_\xi)$ has intersection numbers
$$d_i = \langle L_{i,\xi},\beta \rangle\,,$$
then $\beta$ has positive degree with respect to the
quasi-polarization if and only if  $(d_1,\dots,d_r)$
is positive.

\subsubsection{Noether-Lefschetz divisors} \label{nld}
Noether-Lefschetz numbers are defined in \cite{gwnl}
by the intersection of $\iota_\pi(C)$ with Noether-Lefschetz 
divisors in $\mathcal{M}_\Lambda$.
We briefly review the definition of the 
Noether-Lefschetz divisors.

Let $(\mathbb{L}, \iota)$ be a rank $r+1$
lattice  $\mathbb{L}$
with an even symmetric bilinear form $\langle\, ,\rangle$ and a primitive embedding
$$\iota: \Lambda \rightarrow \mathbb{L}\,.$$
Two data sets 
$(\mathbb{L},\iota)$ and $(\mathbb{L}',  \iota')$
are isomorphic if and only if there exist an isometry relating $\mathbb{L}$
and $\mathbb{L}'$ 
which takes $\iota$ to $\iota'$.
The first invariant of the data $(\mathbb{L}, \iota)$ is
the discriminant
 $\Delta \in \mathbb{Z}$ of 
$\mathbb{L}$.

An additional invariant of $(\mathbb{L}, \iota)$ can be 
obtained by considering 
any vector $v\in \mathbb{L}$ for which{\footnote{Here, $\oplus$
is used just for the additive structure (not orthogonal
direct sum).}}
\begin{equation}\label{ccff} 
\mathbb{L} = \iota(\Lambda) \oplus \mathbb{Z}v\,.
\end{equation}
The pairing
$$\langle v,\cdot\rangle : \Lambda \rightarrow \mathbb{Z}$$
determines an element of $\delta_v\in \Lambda^*$.
Let 
$G = \Lambda^{*}/\Lambda$
be the quotient defined via the injection
$\Lambda \rightarrow \Lambda^*$
 obtained from the pairing $\langle\, ,\rangle$ on $\Lambda$.
The group $G$ is abelian of order given by  the 
discriminant $|\Delta(\Lambda)|$.
The image 
$$\delta \in G/\pm$$
of $\delta_v$ is easily seen to be independent of $v$ satisfying 
\eqref{ccff}. The invariant $\delta$ is the {\em coset} of $(\mathbb{L},\iota)$

By elementary arguments, two data sets $(\mathbb{L},\iota)$ and $(\mathbb{L}',\iota')$
of  rank $r+1$ are isomorphic if and only if the discriminants and cosets are
equal.

Let $v_1,\ldots, v_r$ be an integral basis of $\Lambda$ as before.
The pairing of $\mathbb{L}$ 
with respect to an extended basis $v_{1}, \dots, v_{r},v$
is encoded in the matrix
$$\mathbb{L}_{h,d_{1},\dots,d_{r}} = 
\begin{pmatrix}
\langle v_{1},v_{1}\rangle & \cdots & \langle v_{1},v_{r}\rangle & d_{1} \\
\vdots & \ddots & \vdots & \vdots\\
\langle v_{r},v_{1}\rangle & \cdots & \langle v_{r},v_{r}\rangle & d_{r}\\
d_{1} & \cdots & d_{r} & 2h-2
\end{pmatrix}.$$
The discriminant is
$$\Delta(h,d_{1},\dots,d_{r}) 
= (-1)^r\mathrm{det}(\mathbb{L}_{h,d_{1},\dots,d_{r}})\,.$$
The coset $\delta(h, d_{1},\dots,d_{r})$ is represented by the functional
$$v_i \mapsto d_i\,.$$

The Noether-Lefschetz divisor $P_{\Delta,\delta} \subset \mathcal{M}_{\Lambda}$
is the closure of the locus
 of $\Lambda$-polarized $K3$ surfaces $S$ for which
$(\mathrm{Pic}(S),j)$ has rank $r+1$, discriminant $\Delta$, and coset $\delta$.
By the Hodge index theorem{\footnote{The intersection
form on $\mathrm{Pic}(S)$ is
nondegenerate for an algebraic $K3$ surface.
Hence, a rank $r+1$ sublattice of $\mathrm{Pic}(S)$ which
contains a quasi-polarization must have signature $(1,r)$
by the Hodge index theorem.}}, $P_{\Delta,\delta}$ is empty unless $\Delta > 0.$  By definition, $P_{\Delta,\delta}$ is a reduced subscheme.

Let $h, d_{1}, \dots, d_{r}$ determine a positive discriminant
$$\Delta(h,d_{1},\dots,d_{r}) > 0\,.$$  The Noether-Lefschetz divisor
$D_{h, (d_{1},\dots,d_{r})}\subset \mathcal{M}_{\Lambda}$ is defined by 
the weighted sum
$$D_{h,(d_{1},\dots,d_{r})} 
= \sum_{\Delta,\delta} m(h,d_1,\dots,d_r|\Delta,\delta)\cdot[P_{\Delta,\delta}]$$
where the multiplicity $m(h,d_1,\dots,d_r|\Delta,\delta)$ is the number
of elements $\beta$ of the lattice $(\mathbb{L},\iota)$ of 
type $(\Delta,\delta)$ satisfying
\begin{equation}\label{34f}
\langle \beta, \beta \rangle = 2h-2\,,\ \  \langle \beta, v_{i}\rangle = d_{i}\,.
\end{equation}
If the multiplicity is nonzero, then $\Delta | \Delta(h,d_{1},\dots,d_{r})$ so only 
finitely many divisors appear in the 
above sum.

If $\Delta(h,d_{1},\dots,d_{r}) = 0$, the divisor $D_{h,(d_1,\dots,d_r)}$
has a different definition.
The tautological
line bundle $\mathcal{O}(-1)$ is $\Gamma$-equivariant on the period domain
$M_\Lambda$ and descends to the {\em Hodge line bundle} 
$$\mathcal{K} \rightarrow \mathcal{M}_{\Lambda}\,.$$
We define
$D_{h,(d_{1},\dots,d_{r})} = \mathcal{K}^{*}$
if there exists $v\in \Lambda$ 
satisfying{\footnote{If $\Delta(h,d_1,\dots,d_r)=0$ and \eqref{gmm12} holds, then
$\langle v,v \rangle=2h-2$ is forced.
Since the $d_i$ do not simultaneously vanish, $v\neq 0$.}}
\begin{equation} \label{gmm12}
\langle v_1,v\rangle=d_1, \  \langle v_2,v \rangle=d_2, \ \ldots \,,
\langle v_r,v\rangle=d_r\,.
\end{equation}
If $v$ satisfies \eqref{gmm12}, $v$ is unique.
If no such $v\in \Lambda$ exists, then 
$$D_{h,(d_1,\ldots,d_r)}=0\,.$$
In case $\Lambda$ is a unimodular lattice, such a $v$
always exists.
See \cite{gwnl} for an alternate view of degenerate intersection.

If $\Delta(h,d_{1},\dots,d_{r}) < 0$, the divisor 
$D_{h,(d_1,\dots,d_r)}$ on $\mathcal{M}_{\Lambda}$ is defined to vanish
by the Hodge index theorem.

\subsubsection{Noether-Lefschetz numbers}
Let $\Lambda$ be a lattice of discriminant $l=\Delta(\Lambda)$, and
let $(X,L_1,\ldots,L_r,\pi)$ be 
a 1-parameter family of $\Lambda$-polarized $K3$ surfaces.
The Noether-Lefschetz number $NL^\pi_{h,d_1,\dots,d_r}$ is
the  classical intersection
product
\begin{equation}\label{def11}
NL^\pi_{h,(d_1,\dots,d_r)} =\int_C \iota_\pi^*[D_{h,(d_1,\dots,d_r)}]\,.
\end{equation}

Let $\mathrm{Mp}_{2}(\mathbb{Z})$ be
the metaplectic double cover of $SL_{2}(\mathbb{Z})$. 
There is a canonical representation \cite{borch}
associated to $\Lambda$,
$$\rho_{\Lambda}^{*}: 
\mathrm{Mp}_{2}(\mathbb{Z}) \rightarrow \mathrm{End}(\mathbb{C}[G])\,,$$
where $G=\Lambda^*/\Lambda$.
The full set of Noether-Lefschetz numbers
$NL^\pi_{h,d_1,\dots,d_r}$ defines a vector valued
modular form
 $$\Phi^{\pi}(q) = \sum_{\gamma\in G} \Phi^{\pi}_{\gamma}(q)v_{\gamma} \in \com[[q^{\frac{1}{2l}}]]
\otimes \com[G]\,,$$
of weight $\frac{22-r}{2}$ and type $\rho_\Lambda^*$
by results{\footnote{While the results of the papers \cite{borch, kudmil}
have considerable overlap, we will follow the point of view of Borcherds.}}
 of Borcherds and Kudla-Millson \cite{borch,kudmil}.
The Noether-Lefschetz numbers are the coefficients{\footnote{If $f$ is a series in $q$,
 $f[k]$ denotes the coefficient of $q^k$. 
}}
 of the components of 
$\Phi^\pi$,
$$NL^{\pi}_{h,(d_1,\dots,d_r)} = \Phi^{\pi}_{\gamma}\left[ \frac{\Delta(h,d_1,\dots,d_r)}{2l}\right]$$
where $\delta(h,d_1,\dots,d_r) = \pm\gamma$.
The modular form results significantly constrain the Noether-Lefschetz numbers.

\subsubsection{Refinements} \label{reff12}
If  $d_1,\ldots,d_r$ do not simultaneously
vanish, refined Noether-Lefschetz divisors
are defined.
If $\Delta(h,d_1,\dots,d_r)>0$, 
$$D_{m,h,(d_1,\dots,d_r)}
\subset D_{h,(d_1,\dots,d_r)}$$ is defined
by requiring the class $\beta \in \text{Pic}(S)$ to satisfy \eqref{34f} and
have divisibility $m>0$. If $\Delta(h,d_1,\dots,d_r)=0$, then
$$D_{m,h,(d_1,\dots,d_r)}=D_{h,(d_1,\dots,d_r)}$$
if there exists $v\in \Lambda$
of divisibility
$m>0$ satisfying
\begin{equation*} 
\langle v_1,v\rangle=d_1, \  \langle v_2,v \rangle=d_2, \ \ldots \,,
\langle v_r,v\rangle=d_r\,.
\end{equation*}
If $v$ satisfies the above degree conditions, $v$ is unique.
If no such $v\in \Lambda$ exists, then 
$$D_{m,h,(d_1,\ldots,d_r)}=0\,.$$

A necessary condition for the existence of $v$ is the divisibility of
each $d_i$ by $m$. In case $\Lambda$ is a unimodular lattice,
$v$ exists if and only if $m$
is the greatest common divisor of $d_1,\ldots,d_r$.

Refined
Noether-Lefschetz numbers are defined by
\begin{equation}\label{def112}
NL^\pi_{m,h,(d_1,\dots,d_r)} =\int_C \iota_\pi^*[D_{m,h,(d_1,\dots,d_r)}]\,.
\end{equation}
The full 
set of Noether-Lefschetz numbers $NL^\pi_{h,(d_1,\dots,d_r)}$ is
easily shown to determine the refined numbers $NL^\pi_{m,h,(d_1,\dots,d_r)}$, 
see \cite{KMPS}.

\subsection{GW/NL correspondence} \label{gwnlcc}
The GW/NL correspondence intertwines
three theories associated to   
a 1-parameter family $$\pi:X \rightarrow C$$
 of $\Lambda$-polarized $K3$ surfaces:
\begin{enumerate}
\item[(i)] the Noether-Lefschetz numbers  of $\pi$,
\item[(ii)] the genus $g$ Gromov-Witten invariants of $X$,
\item[(iii)] the genus $g$ reduced Gromov-Witten invariants of the
$K3$ fibers. 
\end{enumerate}
The Noether-Lefschetz numbers (i) are classical intersection
products while the Gromov-Witten invariants (ii)-(iii) 
are quantum in origin. 
For (ii),
we view the theory in terms  the
Gopakumar-Vafa invariants{\footnote{A review of the definitions
will be given in Section \ref{bpss}.} \cite{GV1,GV2}.

Let $n_{g,(d_1,\dots,d_r)}^X$ denote the Gopakumar-Vafa invariant of $X$
in genus $g$ for $\pi$-vertical curve classes of degrees{\footnote{The
invariant $n^X_{g,(d_1,\dots,d_r)}$ may be a (finite) sum of $n_{g,\gamma}^X$
for $\pi$-vertical curve classes $\gamma \in H_2(X,\mathbb{Z})$.}} 
$d_1,\dots,d_r$
with respect to the line bundles $L_1,\dots, L_r$. Let
$r_{g,\beta}$ denote the reduced $K3$ invariant defined in Section \ref{bpsc}
for an effective curve class $\beta$. Since $r_{g,\beta}$
depends only upon the divisibility $m$ and the norm square 
$$\langle \beta, \beta \rangle = 2h-2\,,$$
we will use the more efficient notation
$$r_{g,m,h}= r_{g,\beta}\,.$$

The following result is proven{\footnote{The result of
the \cite{gwnl} is stated in the rank $r=1$ case, but the
argument is identical for arbitrary $r$.}} in \cite{gwnl}
by a comparison of
the reduced and usual deformation theories of maps of curves
to
the $K3$ fibers of $\pi$.

\begin{Theorem} \label{ffc}
 For degrees $(d_1,\dots,d_r )$ positive with respect to the
quasi-polarization $\lambda^\pi$,
$$n_{g,(d_1,\dots,d_r)}^X= \sum_{h=0}^\infty \sum_{m=1}^{\infty}
r_{g,m,h}\cdot  NL_{m,h,(d_1,\dots,d_r)\,}^\pi.$$
\end{Theorem}

For fixed $g$ and $(d_1,\ldots,d_r)$, the sum over 
$m$ is clearly finite since $m$ must divide each $d_i$.
The sum over $h$ is also finite since, for fixed
 $(d_1,\ldots,d_r)$,
$NL_{m,h,(d_1,\dots,d_r)}^\pi$  
vanishes for sufficiently high $h$ by Proposition
3 of \cite{gwnl}. By Lemma 2 of \cite{gwnl},
$r_{g,m,h}$ vanishes for $h<0$ (and is therefore omitted from
the sum in Theorem \ref{ffc}).

\section{Anticanonical $K3$ surfaces in $\widetilde{\proj^2\times \proj^1}$}

\label{tttt}
\subsection{Polarization}\label{fccv}
Let $\widetilde{\proj^2 \times \proj^1}$ be the blow-up of
$\proj^2\times \proj^1$ at a point,
 $$\widetilde{\proj^2 \times \proj^1}\rightarrow \proj^2\times \proj^1\,.$$
The Picard group is of rank 3:
$$\text{Pic}( \widetilde{\proj^2 \times \proj^1})\cong
\mathbb{Z} L_1 \oplus \mathbb{Z} L_2 \oplus \mathbb{Z}{E}\,,$$
where $L_1$ and $L_2$ are the pull-backs of $\mathcal{O}(1)$ from the
factors $\proj^2$ and $\proj^1$ respectively and $E$ is the
exceptional divisor. The anticanonical class $3L_1+2L_2-2E$
is base point free.

A nonsingular anticanonical
 $K3$ hypersurface $S \subset \widetilde{\proj^2 \times\proj^1}$
is naturally lattice polarized by $L_1$, $L_2$, and $E$.
The lattice is
$$\Lambda =  \left( \begin{array}{ccc}
2 & 3 & 0  \\
3 & 0 & 0\\
0 & 0 & -2  \end{array} \right).$$
A general anticanonical Calabi-Yau 3-fold hypersurface,
$$\XX \subset \widetilde{\proj^2 \times \proj^1} \times \proj^1\,,$$
 determines a 1-parameter family of anticanonical $K3$ surfaces 
in $\widetilde{\proj^2\times \proj^1}$, 
\begin{equation}\label{pxxp}
\pi_3: \XX \rightarrow \proj^1 \,, 
\end{equation}
via projection $\pi_3$ onto the last $\proj^1$. The fibers of
$\pi_3$ have at worst nodal singularities.{\footnote{There are 192
nodal fibers. We leave the elementary 
classical geometry here to the reader.}}
The Noether-Lefschetz theory of the $\Lambda$-polarized family
$$(\XX, {L}_1,{L}_2,E,\pi_3)$$
plays a central role in our proof of Theorem 1.
The quasi-polarization $\lambda^{\pi_3}$ (condition (ii) of Section \ref{lpol}) can
be taken to be any very ample line bundle on 
$\widetilde{\proj^2\times \proj^1}$.

\subsection{BPS states}\label{bpss}
Let 
$(\XX,{L}_1,{L}_2,E,\pi_3)$ be 
the $\Lambda$-polarized family of anticanonical $K3$ surfaces of 
$\widetilde{\proj^2\times\proj^1}$
defined in Section \ref{fccv}.
The vertical classes are the kernel of the push-forward map
by ${\pi}_3$,
$$0 \rightarrow H_2({\XX},\mathbb{Z})^{\pi_3} \rightarrow 
H_2({\XX},\mathbb{Z})
\rightarrow H_2(\PP^1,\mathbb{Z}) \rightarrow 0\,.$$
 
Let $\overline{M}_{g}({\XX},\gamma)$ be the moduli space of
stable maps from connected genus $g$ curves to ${\XX}$ of class
$\gamma$.
Gromov-Witten theory is defined
by
integration against the virtual class,
\begin{equation}
\label{klk}
N_{g,\gamma}^{{\XX}}
 = \int_{[\overline{M}_{g}({X},\gamma)]^{\vir}} 1\,.
\end{equation}
The expected dimension of the moduli space is 0.

The full genus 
Gromov-Witten potential $F^{{\XX}}$ for nonzero vertical classes
is the series
$${F}^{{\XX}}=
\sum_{g\geq 0}\  \sum_{0\neq \gamma\in H_2({\XX},\mathbb{Z})^{\pi_3}}  
 N^{{\XX}}_{g,\gamma} \ \lambda^{2g-2} v^\gamma\,,$$
where $v$ is the curve class variable.
The
BPS counts $n_{g,\gamma}^{{\XX}}$
of Gopakumar and Vafa are uniquely defined 
by the following equation:
\begin{equation*}
F^{{\XX}}  =    \sum_{g\geq 0} \ \sum_{0\neq \gamma\in
H_2({\XX},\mathbb{Z})^{\pi_3}} 
 n_{g,\gamma}^{{\XX}} \lambda^{2g-2} \ \sum_{d>0}
\frac{1}{d} \left(\frac{\sin(d\lambda/2)}{\lambda/2}\right)^{2g-2} v^{d\gamma}. 
\end{equation*}
Conjecturally, the invariants $n_{g,\gamma}^{{\XX}}$ are integral and
obtained from the cohomology of an as yet unspecified moduli
space of sheaves on ${X}$. 
We do {\em not} assume the conjectural properties
hold.

Using the  $\Lambda$-polarization, we label the classes
$\gamma \in H_2({\XX},\mathbb{Z})^{\pi_3}$ by their pairings with
$L_i$ and $E$,
$$\gamma \mapsto \left( \int_\gamma [L_1], \int_\gamma [L_2],
\int_\gamma [E]\right).$$
We write the BPS counts as
$
n_{g,(d_1,d_2,d_3)}^{{\XX}}$. Since $\gamma\neq 0$,
not all the $d_i$ can vanish.


\subsection{Invertibility of constraints}
Let $\mathcal{P}\subset \mathbb{Z}^3$ be the set of triples
$(d_1,d_2,d_3)\neq (0,0,0)$ 
which are positive
with respect to the 
quasi-polarization $\lambda^{\pi_3}$ 
of the $\Lambda$-polarized family
$$\pi_3 : \XX \rightarrow \proj^1\,.$$

Theorem \ref{ffc}
applied to $(\XX,L_1,L_2,E,\pi_3)$ yields the equation
\begin{equation}\label{maintheoremforstu}
n^{{\XX}}_{g, (d_1, d_2,d_3)} = 
\sum_{h=0}^\infty \sum_{m=1}^{\infty} 
r_{g,m,h}\cdot NL^{{\pi_3}}_{m,h,(d_1,d_2,d_3)}
\end{equation}
for $(d_1,d_2,d_3)\in \mathcal{P}$.
We view \eqref{maintheoremforstu} as linear constraints
for the unknowns $r_{g,m,h}$ in terms of the 
BPS states on the left and
the refined Noether-Lefschetz degrees.

The integrals $r_{g,m,h}$ are very simple in case $h\leq 0$.
By Lemma 2 of \cite{gwnl},
$r_{g,m,h}=0$ for $h<0$,
$$r_{0,1,0}= 1\,,$$ 
and $r_{g,m,0}=0$ otherwise.

\begin{Proposition}\label{bbh}
The set of invariants $\{r_{g,m,h}\}_{g\geq 0,m\geq 1, h> 0}$ 
is uniquely determined by the set of  constraints \eqref{maintheoremforstu}
 for $(d_1,d_2,d_3) \in \mathcal{P}$ and the integrals $r_{g,m,h\leq 0}$.
\end{Proposition}

\begin{proof}
A certain subset of the 
linear equations will be shown to be
upper triangular in the variables $r_{g,m,h}$.  

Let us fix in advance the values of 
$g\geq 0$,
$m\geq 1$, and  $h>0$.
We proceed by induction on $h$  assuming the reduced invariants $r_{g,m',h'}$
have already been determined for all $h' < h$.  
If $2h-2$ is not divisible by $2m^{2}$, then
we have $r_{g,m,h}= 0$ by definition, so we can further assume
$$2h-2 = m^{2}(2s-2)$$
for an integer $s> 0$.

Consider the fiber class $\gamma\in H_2(\XX,\mathbb{Z})^{\pi_3}$ 
given by the lattice
element $msL_1+mL_2+m(s+1)E$,
$$\gamma =  \Big(ms[L_1]+m[L_2]+m(s+1)[E]\Big) \  \cap\ [S]\,,$$
where $S$ is a $K3$-fiber of $\pi_3$.
Since $L_1$, $L_2$ and $E$ are effective on $S$, the
class $\gamma$
is effective and hence positive with respect to the
quasi-polarization of $\Lambda$. The degrees of $\gamma$
are
\begin{equation}\label{chgchg}
(d_1,d_2,d_3)=(2ms+3m,3ms,-2m(s+1))\,,
\end{equation}
and $\gamma$ is of divisibility exactly $m$ in the
lattice $\Lambda$.

Consider equation \eqref{maintheoremforstu} for 
$(d_1,d_2,d_3)$ given by \eqref{chgchg}.
By the Hodge index theorem,
we must have
\begin{eqnarray}\label{fred}
0& \leq& \Delta(h',2ms+3m,3ms,-2m(s+1))\\ & = & \nonumber
18 (2 - 2h' + m^2(2s-2))\\
& =& 36(h-h') \nonumber
\end{eqnarray}
if $NL^{{\pi_3}}_{m',h',(2ms+3m,3ms,-2m(s+1))
} \neq 0$.
Inequality \eqref{fred} implies
$h' \leq h$. If $h'=h$, then 
$$\Delta(h'=h,2ms+3m,3ms,-2m(s+1))=0\,.$$ By the definition
of Section \ref{reff12},
$$NL^{{\pi_3}}_{m',h'=h,(2ms+3m,3ms,-2m(s+1))} = 0$$
unless there exists $v\in\Lambda$ of divisibility $m'$
with degrees $$(2ms+3m,3ms,-2m(s+1))\,.$$ 
But $\gamma\in \Lambda$
is the unique such lattice element,
and $\gamma$ has divisibility $m$.
Therefore, the constraint \eqref{maintheoremforstu}
takes the form
$$n^{{X}}_{g,(2ms+3m,3ms,-2m(s+1))} =  
r_{g,m,h} NL^{{\pi_3}}_{m, h, (2ms+3m,3ms,-2m(s+1))}+ \dots,$$
where the dots represent terms involving $r_{g,m',h'}$ with  
$h' < h$.
The leading coefficient is given by
$$NL^{{\pi}_3}_{m, h, (2ms+3m,3ms,-2m(s+1))} = 
NL^{{\pi}_3}_{h, (2ms+3m,3ms,-2m(s+1))} = -2\,.$$
As the system is upper-triangular, we can invert to solve for $r_{g,m,h}$.

The calculation of $NL^{{\pi}_3}_{h, (2ms+3m,3ms,-2m(s+1)}$ is elementary.
In the discriminant $\Delta=0$ case, we must determine the 
degree of the dual of the Hodge line bundle $\mathcal{K}$ on the 
base $\proj^1$. The relative dualizing sheaf $\omega_{\pi_3}$
is the pull-back of $\mathcal{O}_{\proj^1}(2)$ from the base.
Hence, the dual of the Hodge line has degree $-2$.
See Section 6.3 of \cite{gwnl} for many such calculations.
\end{proof}

The proof of Proposition \ref{bbh} does {\em not}
involve induction on the genus. The same argument will
be used later in the theory of stable pairs.

\section{Theorem 2}
\label{t22}
\subsection{Strategy}
We will prove Theorem 2 via the GW/P and Noether-Lefschetz
correspondences for the family 
$(\XX,L_1,L_2,E,\pi_3)$ of
$K3$ surfaces of defined in Section \ref{fccv}.
While all of the necessary Gromov-Witten theory has been
established in Sections \ref{nnll}  and \ref{tttt}, our proof here
depends upon
stable pairs results proven later in Sections \ref{mulcov} 
and \ref{pnlc}.

\subsection{Stable pairs}
Let $V$ be a nonsingular, projective 3-fold, and let
$\beta \in H_2(V,\mathbb{Z})$ be a nonzero class. We consider the
moduli space of stable pairs
$$[\OO_V \stackrel{s}{\rightarrow} F] \in P_n(V,\beta)$$
where $F$ is a pure sheaf supported on a Cohen-Macaulay subcurve of $V$, 
$s$ is a morphism with 0-dimensional cokernel, and
$$\chi(F)=n\,, \  \  \ [F]=\beta\,.$$
The space $P_n(V,\beta)$
carries a virtual fundamental class of dimension
$\int_\beta c_1(T_V)$ obtained from the 
deformation theory of complexes with trivial determinant in
the derived category \cite{PT1}.

We specialize now to the case where $V$ is the total space of
a $K3$-fibration (with at worst nodal fibers),
$$\pi:V\rightarrow C\,, $$ 
over a nonsingular projective curve 
and $\beta \in H_2(V,\mathbb{Z})^\pi$ is a vertical class.
Then the
expected dimension of $P_n(V,\beta)$ is always 0.
For nonzero $\beta\in H_2(V,\Z)^\pi$,
define the stable pairs invariant 
\begin{eqnarray*}
\widetilde{N}^\bullet_{n,\beta}&  = &
\int_{[P_{n}(V,\beta)]^{\vir}}
1\,. 
\end{eqnarray*}
The partition function is 
$$
\ZZ_{\mathsf{P}}\Big(V;q\Big)_\beta
=\sum_{n} \widetilde{N}^\bullet_{n,\beta}\, q^n.
$$

Since $P_n(V,\beta)$ is empty for sufficiently negative
$n$, the partition function 
is a Laurent series in $q$. The following is a special
case of Conjecture 3.26 of
\cite{PT1}.

\begin{Conjecture}
\label{111} 
The partition function
$\ZZ_{\mathsf{P}}\big(V;q)_\beta$ is the 
Laurent expansion of a rational function in $q$.
\end{Conjecture}

If the total space $V$ is a Calabi-Yau 3-fold, 
then Conjecture \ref{111} has been
proven in \cite{Br,Toda}.
In particular, Conjecture \ref{111} holds for the
anticanonical 3-fold
$$\XX \subset \widetilde{\proj^2 \times \proj^1} \times \proj^1$$
of Section \ref{fccv}.

In fact, if $V$ is any  complete intersection
Calabi-Yau 3-fold 
in a toric variety which admits sufficient degenerations, Conjecture
\ref{111} has been proven in \cite{PP}. 
By factoring equations, there
is no difficulty in constructing the degenerations of $\XX$
into toric $3$-folds required for \cite{PP}. Just as in the
case of the quintic in $\proj^4$, the geometries which occur
are toric $3$-folds, projective bundles over $K3$ and toric
surfaces, and fibrations over curves. A complete discussion
of the degeneration scheme for $\XX$ is given in Appendix B.

\subsection{GW/P correspondence for $\XX$}
Following the notation of Section \ref{bpss}, 
let $H_2(\XX,\mathbb{Z})^{\pi_3}$ denote the vertical classes
of $\XX$ and let
$${F}^{{\XX}}=
\sum_{g\geq 0}\  \sum_{0\neq \gamma\in H_2({X},\mathbb{Z})^{\pi_3}}  
 N^{{\XX}}_{g,\gamma} \ \lambda^{2g-2} v^\gamma\ $$
be the potential of connected Gromov-Witten invariants.
The partition function (of possibly disconnected) Gromov-Witten
invariants is defined via the exponential,
$$
\ZZ_{\mathsf{GW}}\Big(\XX;\lambda\Big)
=\exp\left(F^\XX\right).
$$
Let $\ZZ_{\mathsf{GW}}\Big(\XX;\lambda\Big)_\gamma$
denote the coefficient of $v^\gamma$ in
$\ZZ_{\mathsf{GW}}\Big(\XX;\lambda\Big)$.
The main result of \cite{PP} applied to $\XX$ is the following GW/P
correspondence for complete intersection Calabi-Yau
3-folds in products of projective spaces. 

\vspace{10pt}
\noindent {\bf GW/P correspondence}. {\em After the
change of variable $-q= e^{i \lambda}$, we have
$$\ZZ_{\mathsf{GW}}\Big(\XX;\lambda\Big)_\gamma =
\ZZ_{\mathsf{P}}\Big(\XX;q\Big)_\gamma\,. $$}
\vspace{10pt}

The change of variables is well-defined by the rationality of
$\ZZ_{\mathsf{P}}\Big(\XX;q\Big)_\gamma$ of Conjecture 3. The GW/P correspondence
is proven in \cite{PP} for {\em every} non-zero class in
$H_2(\XX,\mathbb{Z})$, but we only will require here the statement for
fiber classes $\gamma$.

\subsection{$K3$ integrals} \label{k3intp}
Let $S$ be a nonsingular projective $K3$ surface with
a nonzero class $\alpha\in \text{Pic}(S)$
 which is both effective and primitive.
By the definitions of Sections \ref{bpsc} in Gromov-Witten theory,
\begin{equation*} \label{hbb9}
{F}_{{\alpha}}=
\sum_{g\geq 0}\   \sum_{m>0}\   R_{g,m\alpha} \ \lambda^{2g-2} 
v^{m{\alpha}},
\end{equation*}
\begin{equation*} \label{kxxz9}
F_\alpha =   \ \ \ \sum_{g\geq 0}  \ \sum_{m>0} \
 r_{g,m\alpha} \ \lambda^{2g-2} \sum_{d>0}
\frac{1}{d}\left( \frac{\sin
({d\lambda/2})}{\lambda/2}\right)^{2g-2} v^{dm\alpha}. 
\end{equation*}

Via $K3$-fibrations over a pointed curve $$\epsilon: T \rightarrow (\Delta,0)$$
satisfying the conditions (i), (ii), and $(\star)$ of Section \ref{gwp4},
we have defined in \eqref{hbbb} the series
\begin{equation*}
\widetilde{F}_{{\alpha}}=
\sum_{n\in \mathbb{Z}}\   \sum_{m>0}\   \widetilde{R}_{n,m\alpha} \ 
q^{n} 
v^{m{\alpha}}\ 
\end{equation*}
in the theory of stable pairs.
Using the identity
\begin{eqnarray*} 2^{2g-2}{\sin(d\lambda/2)}^{2g-2} &=& \left(
\frac{e^{id\lambda/2}-e^{-id\lambda/2}}{i}\right)^{2g-2}\\
& =&
(-1)^{g-1} \left((-q)^d-2+(-q)^{-d} \right)^{g-1} 
\end{eqnarray*}
under the change of variables $-q=e^{i\lambda}$, we define
the stable pairs BPS invariants
$\widetilde{r}_{g,m\alpha}$ by the relation
$$\widetilde{F}_{{\alpha}}= 
\ \ \ \sum_{g\in \mathbb{Z}}  \ \sum_{m>0} \
 \widetilde{r}_{g,m\alpha} \  \sum_{d>0}
\frac{(-1)^{g-1}}{d} 
 \left((-q)^d-2+(-q)^{-d} \right)^{g-1}
 v^{dm\alpha}\,.$$
See Section 3.4 of \cite{PT1} for a discussion of such
BPS expansions for stable pairs. The invariants $\widetilde{r}_{g,m\alpha}$
are integers.

Since $\widetilde{r}_{g,\beta}$
depends only upon the divisibility $m$ and the norm  square
$$\langle \beta, \beta \rangle = 2h-2\,,$$
we will use, as before, the notation
\begin{eqnarray*}
\widetilde{r}_{g,m,h}= \widetilde{r}_{g,\beta}\,.
\end{eqnarray*}
By definition in Gromov-Witten theory, $r_{g,m,h}=0$ for $g<0$. However
for fixed $m$ and $h$, the definitions allow
$r_{g,m,h}$ to
be nonzero for all positive $g$.
On the stable pairs side for fixed $m$ and $h$,
$\widetilde{r}_{g,m,h}=0$ for sufficiently large $g$, but
$\widetilde{r}_{g,m,h}$ may be nonzero
for all negative $g$.

We will prove Theorem \ref{vvee} by showing the BPS counts
for $K3$ surfaces in
Gromov-Witten theory and stable pairs theory exactly match:
\begin{equation} \label{rtrtrt}
r_{g,m,h} = \widetilde{r}_{g,m,h}\ 
\end{equation}
for all $g\in \mathbb{Z}$, $m\geq 1$, and $h\in \mathbb{Z}$.

\subsection{Noether-Lefschetz theory for stable pairs} \label{spnl}
The stable pairs potential $\widetilde{F}^{{\XX}}$ for nonzero vertical classes
is the series
$$\widetilde{F}^{{\XX}}= \log \left(1 + \sum_{0\neq \gamma\in H_2({\XX},\mathbb{Z})^{\pi_3}}  
 \ZZ_{\mathsf{P}}\Big(\XX;q\Big)_\gamma v^\gamma\right)\,,$$
where $v$ is the curve class variable.
The stable pairs
BPS counts $\widetilde{n}_{g,\gamma}^{{\XX}}$
are uniquely defined 
by 
\begin{equation*}
\widetilde{F}^{{\XX}}  =    \sum_{g\in \mathbb{Z}} \ \sum_{0\neq \gamma\in
H_2({\XX},\mathbb{Z})^{\pi_3}} 
\widetilde{n}_{g,\gamma}^{{\XX}} \ \sum_{d>0}
\frac{(-1)^{g-1}}{d}  \left((-q)^d-2+(-q)^{-d} \right)^{g-1}
v^{d\gamma}, 
\end{equation*}
following Section 3.4 of \cite{PT1}.

The following stable pairs result is proven in Section \ref{pnlc}.
A central issue in the proof is the translation of the
Noether-Lefschetz geometry of stable pairs to 
a precise relation constraining the {\em logarithm} $\widetilde{F}^{{\XX}}$
of
the stable pairs series.

\begin{Theorem} \label{ffc2}
 For degrees $(d_1,d_2,d_3 )$ positive with respect to the
quasi-polariza\-tion of the family $\pi_3:\XX \rightarrow \proj^1$,
$$\widetilde{n}_{g,(d_1,d_2,d_3)}^\XX= \sum_{h=0}^\infty \sum_{m=1}^{\infty}
\widetilde{r}_{g,m,h}\cdot  NL_{m,h,(d_1,d_2,d_3)}^{\pi_3}\,.$$
\end{Theorem}

\subsection{Proof of Theorem 2}
We first match the BPS counts of $\XX$ by using the GW/P correspondence.
Then, the uniqueness statement of Proposition \ref{bbh} implies
\eqref{rtrtrt}.

\begin{Proposition} \label{gg11} For all $g\in \mathbb{Z}$ and all
 $\gamma\in
H_2({\XX},\mathbb{Z})^{\pi_3}$, we have
$$n_{g,\gamma}^\XX = \widetilde{n}_{g,\gamma}^\XX\,.$$
\end{Proposition}

\begin{proof} By Corollary \ref{ww3} of Section \ref{mulcov},
$\widetilde{r}_{g,m,h}=0$ if $g<0$. Theorem \ref{ffc2} then implies
$\widetilde{n}_{g,\gamma}^\XX=0$ if $g<0$.
Hence, there are only finitely many nonzero BPS states{\footnote{The
GW/P correspondence yields an equality of
partition functions after the variable change $-q=e^{i\lambda}$
whether or not $\widetilde{n}_{g<0,\gamma}$ vanishes. Proposition
\ref{gg11} asserts
a the stronger result: the Gromov-Witten BPS expansion
equals the stable pairs BPS expansion. Since these
expansions are in opposite directions, finiteness is needed.}} 
for fixed $\gamma$ since $\widetilde{n}^\XX_{g,\gamma}$ vanishes
for sufficiently large $g$ by construction \cite{PT1}.
By the 
GW/P correspondence, the $\widetilde{n}^\XX_{g,\gamma}$ then yield
the Gromov-Witten BPS expansion.
\end{proof}

\begin{Proposition}\label{gg12} For all $g\in \mathbb{Z}$, $m\geq 1$, and 
$h\in \mathbb{Z}$, we have
$$r_{g,m,h} = \widetilde{r}_{g,m,h}\,.$$
\end{Proposition}

\begin{proof}
The equality $r_{g,m,h} = \widetilde{r}_{g,m,h}$
holds in case $h\leq 0$ by following the
argument of Lemma 2 of \cite{gwnl} for stable pairs.{\footnote{A 
different argument is given in Corollary \ref{ww22} of Section
\ref{mulcov}.}}
For $h<0$, the
vanishing of 
$\widetilde{r}_{g,m,h}$
holds by the same geometric 
argument given in Lemma 2 of \cite{gwnl}.
The $h=0$ case is the conifold for which
the equality is well known (and a consequence of
GW/P correspondence).

We view
relation \eqref{maintheoremforstu} and Theorem 4 as systems of linear
equations for the unknowns
$r_{g,m,h}$  and $\widetilde{r}_{g,m,h}$
respectively. 
By Proposition \ref{gg11}, we have
$$n_{g,\gamma}^X = \widetilde{n}_{g,\gamma}^X\ $$
for all $g$.
Hence, the two systems of linear equations are
the {\em same}.

We now apply the uniqueness established
in Proposition \ref{bbh}. The initial conditions and
the linear equations are identical. Therefore
the solutions must also agree. 
 \end{proof}

Theorem 2 follows immediately from Proposition \ref{gg12} for the
$K3$ invariants in Gromov-Witten theory and stable pairs. \qed


\section{$K3 \times \C$: Localization}
\label{k3tcl}
\subsection{Overview}
We begin now our analysis of the moduli spaces of stable pairs related to $K3$ surfaces and $K3$-fibrations.
Let $S$ be a nonsingular projective $K3$ surface.
We first study the trivial fibration 
$$Y=S \times \C \To \C$$
 by $\C^*$-localization with respect to
the scaling action on $\C$.  Let $\t$ denote
 the weight 1 representation
of $\C^*$ on the tangent space to $\C$ at $0\in \C$.

 We compute here the $\C^*$-residue contribution to the 
reduced stable pairs theory of $S \times \C$ 
of 
the $\C^*$-fixed component\footnote{Throughout we use the term  {\em component}  to denote any open and closed subset. More formally, a component for us is a union of connected components
in the standard sense.} parameterizing stable pairs supported on $S$ and thickened uniformly $k$ times about 
$0\in\C$. In Section \ref{van}, all other $\C^*$-fixed components will be shown to have vanishing 
contributions to the virtual localization formula.

\subsection{Uniformly thickened pairs} \label{uniform}
Define the following Artinian rings and schemes:
\beq{Ak}
A_k=\C[x]/(x^k)\, ,\qquad B_k=\Spec A_k\, .
\eeq
We have the obvious maps
$$
\xymatrix{\Spec\C & B_k\ \ar[l]_(.38){\pi_k}\ar@{^(->}[r]^(.3){\iota\_k} & \C=\Spec\C[x]\, .}
$$
For any variety $Z$,  we define
 $$Z_k=Z\times B_k\, ,$$ and use the same symbols $\pi_k,\iota_k$ to denote the corresponding projections and inclusions,
\beq{Y}
\xymatrix{Z & Z_k\ \ar[l]_(.46){\pi_k}\ar@{^(->}[r]^(.4){\iota\_k} & Z\times\C\, .}
\eeq
We will often abbreviate $\iota_k$ to $\iota$.

Let $\beta \in H_2(S,\Z)$ be a curve class.
Let $P_S=P_n(S,\beta)$ denote the moduli space of stable pairs on $S$ 
with universal stable pair $(\FF,s)$ and universal complex
$$
\II_S=\big\{\O_{S\times P_S}\Rt{s}\FF\big\}\,.
$$
Using the maps \eqref{Y} for $Z=S\times P_S$ (so $Z\times\C=Y\times P_S$) we define
\beq{kk}
\FF_k=\pi_k^*\FF, \qquad \II_{S_k}=\big\{\O_{S_k\times P_S}\Rt{s_k}\FF_k\big\}
\eeq
on $S_k\times P_S$, where $s_k=\pi_k^*s$. Pushing $s_k$ forward to $Y\times P_S$ we obtain
\beq{IXk}
\II_Y=\big\{\O_{Y\times P_S}\Rt{s_k}\iota_*\FF_k\big\}\,.
\eeq
Since we have constructed a 
flat family over $P_S$ of stable pairs on $Y$ of class{\footnote{For any class $\gamma\in \text{Pic}(S)$,
we denote the push-forward to $H_2(Y,\mathbb{Z})$ also
by $\gamma$.}} 
$k\beta$ and holomorphic Euler characteristic $kn$, we obtain a classifying map from $P_S$ to the moduli space of stable pairs on $Y$:
\beq{family}
\xymatrix{f\colon P_S=P_n(S,\beta)\ \ar[r] & P_{kn}(Y,k\beta)=P_Y\,.}
\eeq


\begin{Lemma}
The map \eqref{family} is an isomorphism onto
an open and closed component of
 the $\C^*$-fixed locus of $P_Y$. 
\end{Lemma}

\begin{proof}
Let $P_f$ denote the open and closed component of $(P_Y)^{\C^*}$ containing the image of $f$. Certainly, $f$ is a bijection on closed points onto $P_f$.
There is a $\C^*$-fixed universal stable pair on $Y\times P_f$.
We push down the universal stable pair to $S\times P_f$ and 
then take $\C^*$-invariant sections. The result is flat over $P_f$
and hence classified by a map $P_f\to P_S$ which is
easily seen to be the inverse map to $f$.
\end{proof}

\subsection{Deformation theory of pairs}

Let $P_Y=P_m(Y,\gamma)$  be the moduli space of stable pairs on $Y$
of class $\gamma\in H_2(Y,\Z)$ with holomorphic Euler characteristic $m$. There is a universal 
complex $\II_Y$ over $Y\times P_Y$.  We will soon take 
$m=kn$ and $\gamma=k\beta$, in which case $\II_Y$ pulls back 
via the classifying map 
$f\colon P_S\to P_Y$ to \eqref{IXk}.

We review here the basics of
 the deformation theory of stable pairs on the 3-fold $Y$ \cite{PT1}.
Let $$\pi_P\colon Y\times P_Y\to P_Y$$ be the projection, and define
\beq{pot}
E_Y\udot=(R\hom_{\pi_P}(\II_Y,\II_Y)\_0)^\vee[-1]\cong
R\hom_{\pi_P}(\II_Y,\II_Y\otimes\omega_{\pi_P})\_0[2]\,.
\eeq
Here $R\hom_{\pi_P}=R\pi_{P*}R\hom$, the subscript $0$ denotes trace-free homomorphisms, and the isomorphism is Serre duality\footnote{Although $\pi_P$ is not proper, the compact support of $R\hom(\II_Y,\II_Y)_0$ ensures that Serre duality holds. This is proved in \cite[Equations 15,16]{MPT}, for instance, by compactifying $Y=S\times\C$ to $Y=\C\times\PP^1$.} 
down $\pi_P$. 
Let $\LL$ denote the truncated cotangent complex $\tau^{\ge-1}L\udot$. Using the Atiyah class of $\II_Y$, we obtain a map \cite[Section 2.3]{PT1},
\begin{equation}\label{xcc12}
E_Y\udot\To\LL_{P_Y}\,,
\end{equation}
exhibiting $E_Y\udot$ as a perfect obstruction theory for $P_Y$ \cite[Theorem 4.1]{HT}.

In fact, \eqref{xcc12} is the natural obstruction theory of trivial-determinant objects $I\udot=\{\O_Y\rt{s}F\}$ of the derived category
$D(Y)$. The more natural obstruction theory of pairs $(F,s)$ 
is given by the complex\footnote{This is essentially proved in \cite{Ill} once combined with \cite[Theorem 4.5]{BehFan}, see \cite[Sections 12.3-12.5]{JS} for a full account.}
\beq{wrongone}
(R\hom_{\pi_P}(\II_Y,{\mathbb{F}}_Y))^\vee
\eeq
where ${\mathbb{F}}_Y$ is the universal sheaf.
However \eqref{wrongone} is \emph{not} perfect in general. To define stable pair invariants, we must use \eqref{pot}. The two theories give the same tangent
spaces, but different obstructions.  On surfaces, however, the analogous obstruction theory
\beq{Spot}
E\udot_S=(R\hom_{\pi_P}(\II_S,\mathbb F))^\vee
\eeq
is indeed perfect and is  used to define invariants \cite{KT1}. 
Here $\pi_P$ denotes the projection $S\times P_S\to P_S$. 

The following result
describes the relationship between the above obstruction theories 
when pulled-back via the map
$f:P_S\to P_Y$  of \eqref{family}.

\begin{Proposition} \label{obsthys}
We have an isomorphism
$$
f^*E_Y\udot\ \cong\ E_S\udot\otimes\!A_k^*\ \oplus\
(E_S\udot)^\vee\!\otimes\!\t^{-1\!}A_k[1]\,,
$$
where\footnote{We ignore here the ring structure \eqref{Ak} on $A_k$ and considering $A_k$ as just a vector space with $\C^*$-action. As such, $A_k^*\cong
\t^{k-1\!}A_k$, as we use below.} $A_k^*=1+\t+\ldots+\t^{k-1}$ and
$\t^{-1\!}A_k=\t^{-1}+\t^{-2}+\ldots+\t^{-k}$.
\end{Proposition}

\begin{proof}
We will need two preliminaries on pull-backs. First, over $S_k\times P_S$ there is a canonical exact triangle
\beq{FFF}
\mathbb F_k\otimes N_k^*[1]\To\iota^*\iota_*\mathbb F_k\To\mathbb F_{k}\,,
\eeq
where $\iota^*=L\iota_k^*$ is the \emph{derived} pull-back functor, and
 $$N_k\cong\O_{S_k\times P_S}\otimes\t^k$$ denotes the normal bundle of $S_k\times P_S$ in $Y\times P_S$. 
Second, combine the first arrow of \eqref{FFF} with the obvious map $\O_{Y\times P_S}\to\iota_*\O_{S_k\times P_S}$:
\begin{equation*}
\begin{split}
R\curly H&\!om(\iota_*\mathbb F_k,\O_{Y\times P_S})\To R\hom(\iota_*\mathbb F_k,\iota_*\O_{S_k\times P_S}) \\ &\cong\,\iota_*R\hom(\iota^*\iota_*\mathbb F_k,\O_{S_k\times P_S})\To
\iota_*R\hom(\mathbb F_k\otimes N_k^*[1],\O_{S_k\times P_S})\,.
\end{split}
\end{equation*}
A local computation shows the above composition is an isomorphism:
\beq{shriek}
R\hom(\iota_*\mathbb F_k,\O_{Y\times P_S})\ \cong\ \iota_*R\hom(\mathbb F_k\otimes N_k^*[1],\O_{S_k\times P_S})\,.
\eeq

\smallskip

Now combine \eqref{FFF} with $\iota^*$ of the triangle
\beq{stand}
\II_Y\to\O_{Y\times P_S}\to\iota_*\mathbb F_k
\eeq
to give the following diagram of exact triangles on $S_k\times P_S$:
\beq{iFI}\xymatrix@=16pt{
& \iota^*\II_Y \ar[d] \\
& \O_{S_k\times P_S} \ar[d]\ar@{=}[r] & \O_{S_k\times P_S} \ar[d] \\
\mathbb F_k\otimes N_k^*[1] \ar[r] & \iota^*\iota_*\mathbb F_k \ar[r] & \mathbb F_k\,.\!\!\!}
\eeq
The right hand column defines the complex $\II_{S_k}$ \eqref{kk}, so by the octahedral axiom we can fill in the top row with the exact triangle
\beq{kx}
\FF_k\otimes N_k^*\To\iota^*\II_Y\To\II_{S_k}\,.
\eeq
Again letting $\pi_P$ denote both projections 
$$S_k\times P_S\to P_S\ \ \text{and}\ \ Y\times P_S\to P_S\,,$$
we apply $R\hom_{\pi_P}(\II_Y,\ \cdot\ )$ to $\II_Y\to\O_{Y\times P_S}\to\iota_*\mathbb F_k$ to give the triangle
\beq{starti}
R\hom_{\pi_P}(\II_Y,\iota_*\mathbb F_k)\To
R\hom_{\pi_P}(\II_Y,\II_Y)[1]\To R\hom_{\pi_P}(\II_Y,\O)[1]
\eeq
relating the obstruction theory \eqref{wrongone} to the obstruction theory \eqref{pot} (without its trace-part removed: we will deal with this presently).

Now use the following obvious diagram of exact triangles on $Y\times P_S$:
$$\xymatrix@=16pt{
\II_Y \ar[r]\ar[d] & \O_{Y\times P_S} \ar[r]\ar[d] & \iota_*\mathbb F_k \ar@{=}[d]
\\ \iota_*\II_{S_k} \ar[r] & \iota_*\O_{S_k\times P_S} \ar[r] & \iota_*\mathbb F_k\,.\!\!}
$$
This maps the triangle \eqref{starti} to the triangle
$$
R\hom_{\pi_P}(\II_Y,\iota_*\mathbb F_k)\!\To\!R\hom_{\pi_P}(\II_Y,\iota_*\II_{S_k})[1] \!\To\!R\hom_{\pi_P}(\II_Y,\iota_*\O)[1].
$$
By adjunction this is
$$
R\hom_{\pi_P}(\iota^*\II_Y,\mathbb F_k)\!\To\!R\hom_{\pi_P}(\iota^*\II_Y,\II_{S_k})[1] \!\To\!R\hom_{\pi_P}(\iota^*\II_Y,\O)[1],
$$
which in turn maps to
\beq{larst}
R\hom_{\pi_P}(\mathbb F_k,\mathbb F_k)\t^k\!\To\!R\hom_{\pi_P}(\mathbb F_k,\II_{S_k})\t^k[1] \!\To\!R\hom_{\pi_P}(\mathbb F_k,\O)\t^k[1]
\eeq
by the first arrow $\FF_k\otimes N_k^*\to\iota^*\II_Y$ of \eqref{kx}. Notice that since $\II_{S_k}$ and $\FF_k$ are the pull-backs of $\II_S$ and $\FF$ by $\pi_k\colon S_k\to S$, the central term simplifies to
$$
R\hom_{\pi_P}(\mathbb F,\II_S)\otimes\t^kA_k[1]\ \cong\ E\udot_S\otimes\t^k A_k[-1]\,,
$$
where $E\udot_S$ is the obstruction theory \eqref{Spot}.
\medskip

We next remove the trace component of \eqref{starti} using the diagram
\vspace{3mm}\beq{biggy}
\raisebox{15pt}[0pt]{\xymatrix@=14pt{
& R\pi_{P*}\O_{Y\times P_S}[1] \ar@{=}[r]\ar[d]^\id & R\pi_{P*}\O_{Y\times P_S}[1] \ar[d] \\
R\hom_{\pi_P}(\II_Y,\iota_*\mathbb F_k) \ar[r]&
R\hom_{\pi_P}(\II_Y,\II_Y)[1] \ar[r]\ar[d]&
R\hom_{\pi_P}(\II_Y,\O)[1] \ar[d] \\
& R\hom_{\pi_P}(\II_Y,\II_Y)\_0[1] \ar[r]&
R\hom_{\pi_P}(\iota_*\mathbb F_k,\O)[2]\,.}}
\eeq
Here the right hand column is given by applying $R\hom_{\pi_P}(\ \cdot\ ,\O)$ to \eqref{stand}. The top right hand corner commutes because the composition $\O\Rt{\id}\hom(\II_Y,\II_Y)
\to\hom(\II_Y,\O)$  takes $1$ to the canonical map $\II_Y\to\O$. Therefore the whole diagram commutes.

The central row of \eqref{biggy} is \eqref{starti}, and
our map from \eqref{starti} to \eqref{larst} kills the top row of \eqref{biggy} by $\C^*$-equivariance: $R\pi_{P*}\O_{Y\times P_S}[1]$ has $\C^*$-weights in $(-\infty,0]$ while \eqref{larst} has weights in $[1,k]$. Therefore it descends to a map from the bottom row of \eqref{biggy} (completed using the octahedral axiom) to \eqref{larst}. The upshot is the following map of triangles
\beq{6am}
\xymatrix@=16pt{
R\hom_{\pi_P}(\II_Y,\iota_*\mathbb F_k) \ar[r]\ar[d] & f^*(E\udot_Y)^\vee \ar[r]\ar[d] & R\hom_{\pi_P}(\iota_*\mathbb F_k,\O)[2] \ar[d] \\
R\hom_{\pi_P}(\mathbb F_k,\mathbb F_k)\t^k \ar[r] & E\udot_S\otimes\t^k A_k[-1] \ar[r] & R\hom_{\pi_P}(\mathbb F_k,\O)\t^k[1]\,.}
\eeq
Recall that the first column was induced from the triangle \eqref{kx} so sits inside a triangle
$$
R\hom_{\pi_P}(\II_{S_k},\FF_k)\To R\hom_{\pi_P}(\II_Y,\iota_*\FF_k)\To
R\hom_{\pi_P}(\FF_k,\FF_k)\t^k\,.
$$
Again we can simplify because $\II_{S_k}$ and $\FF_k$ are the pull-backs of $\II_S$ and $\FF$ by $\pi_k\colon S_k\to S$. That is,
$$
R\hom_{\pi_P}(\II_Y,\iota_*\FF_k)\ \cong\
R\hom_{\pi_P}(\II_S,\FF)\otimes A_k\ \oplus\ R\hom_{\pi_P}(\FF,\FF)\otimes\t^{k}A_k\,,
$$
where the splitting follows from the $\C^*$-invariance of the connecting homomorphism: it must vanish because $A_k$ has weights in $[-(k-1),0]$ while $\t^kA_k$ has weights in $[1,k]$.

So this splits the first vertical arrow of \eqref{6am}; we claim the last vertical arrow is the isomorphism induced by \eqref{shriek}. Altogether this gives the splitting
$$
f^*(E_Y\udot)^\vee\ \cong\ R\hom_{\pi_P}(\II_S,\FF)\otimes A_k\ \oplus\ E\udot_S[-1]\otimes\t^kA_k\,.
$$
Dualizing gives
$$
f^*E_Y\udot\ \cong\ E_S\udot\otimes A_k^*\ \oplus\ (E\udot_S)^\vee[1]\otimes\t^{-1\!}A_k\,,
$$
as required.

It remains to prove the claim that the third vertical arrow of \eqref{6am} is induced by \eqref{shriek}. By the construction of these maps, it is sufficient to prove the commutativity of the diagram
$$\xymatrix@=16pt{
R\hom_{\pi_P}(\II_Y,\O_{Y\times P_S}) \ar[r]\ar[d]_{\partial^*} & R\hom_{\pi_P}(\mathbb F_k\otimes N_k^*,\O_{S_k\times P_S}) \\
R\hom_{\pi_P}(\iota_*\mathbb F_k[-1],\O_{Y\times P_S})\,. \ar[ur]_\sim}
$$
Here the vertical arrow is induced by the connecting homomorphism $\partial$ of the standard triangle $\II_Y\to\O_{Y\times P_S}\to\iota_*\mathbb F_k$, and the diagonal arrow is $R\pi_{P*}$ applied to \eqref{shriek}.

The horizontal arrow is our map from \eqref{starti} to \eqref{larst} (restricted to the right hand term in each triangle). It is therefore the composition
\begin{equation*}
\begin{split}
R\hom_{\pi_P}&(\II_Y,\O_{Y\times P_S})\To R\hom_{\pi_P}(\II_Y,\iota_*\O_{S_k\times P_S}) \\ \cong&\
R\hom_{\pi_P}(\iota^*\II_Y,\O_{S_k\times P_S})\Rt{\eqref{kx}} R\hom_{\pi_P}(\mathbb F_k\otimes N_k^*,\O_{S_k\times P_S})\,.
\end{split}
\end{equation*}
Via $\partial\colon\iota_*\mathbb F_k[-1]\to\II_Y$, the above composition maps to the composition
\begin{equation*}
\begin{split}
R&\hom_{\pi_P}(\iota_*\mathbb F_k[-1],\O_{Y\times P_S})\To R\hom_{\pi_P}(\iota_*\mathbb F_k[-1],\iota_*\O_{S_k\times P_S}) \\ &\ \cong
R\hom_{\pi_P}(\iota^*\iota_*\mathbb F_k[-1],\O_{S_k\times P_S})\Rt{\eqref{FFF}}
R\hom_{\pi_P}(\mathbb F_k\otimes N_k^*,\O_{S_k\times P_S})\,.
\end{split}
\end{equation*}
Therefore the first two resulting squares commute. The last square  is:
$$
\xymatrix@R=15pt{
R\hom_{\pi_P}(\iota^*\II_Y,\O_{S_k\times P_S}) \ar[r]^(.46){\eqref{kx}}\ar[d]^{\iota^*\partial^*}
& R\hom_{\pi_P}(\mathbb F_k\otimes N_k^*,\O_{S_k\times P_S}) \ar@{=}[d] \\
R\hom_{\pi_P}(\iota^*\iota_*\mathbb F_k[-1],\O_{S_k\times P_S}) \ar[r]^{\eqref{FFF}}
& R\hom_{\pi_P}(\mathbb F_k\otimes N_k^*,\O_{S_k\times P_S})\,.}
$$
By the construction of \eqref{kx} from \eqref{FFF}, the square commutes.
\end{proof}

A virtual class on $P_{kn}(Y,k\beta)^{\C^*}$ 
induced by $(E\udot_Y)^{\fix}$ is defined in \cite{GP}.
The moduli space  
$P_n(S,\beta)$ hence carries several virtual classes:
\begin{enumerate}
\item[(i)] via the intrinsic obstruction theory $E\udot_S$,
\item[(ii)] via $(E\udot_Y)^{\fix}$ and the local isomorphism 
$$f:P_n(S,\beta) \Into P_{kn}(Y,k\beta)^{\com^*}$$
for
 every $k\geq 1$.
\end{enumerate}

\begin{Proposition}\label{obmat} The virtual classes on $P_n(S,\beta)$
obtained from (i) and (ii) are all equal.
\end{Proposition}

\begin{proof}
By Proposition \ref{obsthys}, there is an isomorphism,
\begin{equation}
\label{n234}
{f^*(E\udot_Y)^{\fix} \cong E\udot_S}\,,
\end{equation}
in the derived category.
Since the virtual class
is expressed in terms of the
Fulton total Chern class of $P_n(S,\beta)$ and the
Segre class of the dual of the obstruction theory{\footnote{The relationship of the
virtual class with the Fulton total Chern class
and the normal cone is reviewed in Appendix \ref{ap11}.}},
the isomorphism \eqref{n234} implies the equality of
the virtual classes.
\end{proof}

In fact Proposition \ref{obmat} is trivial:
the virtual classes of
$P_n(S,\beta)$ obtained from the obstruction theories $E\udot_S$
and $(E\udot_Y)^{\fix}$ both vanish
by the existence of the
{\em reduced} theory.

The reduced obstruction theory for $Y$ is constructed in, for instance, \cite[Section 3]{MPT}. We review the construction in a slightly more general setting in Section \ref{rott}. For the reduced theory of $S$,
 we can either $\C^*$-localize the 3-fold reduced class, or equivalently, use the construction in \cite{KT1}. In particular, Proposition 3.4 of \cite{KT1} shows  the two constructions are compatible under the isomorphism \eqref{family}. They both remove a trivial piece $\O_P[-1]$ (of $\C^*$-weight 0!) from the obstruction theory. The nontrivial version of Proposition \ref{obmat} is the following.

\begin{Proposition}\label{obmat2} The reduced virtual classes on $P_n(S,\beta)$
obtained from (i) and (ii) are all equal.
\end{Proposition}

The proof of Proposition \ref{obmat2} is exactly the same as
the proof of Proposition \ref{obmat} given above.

\subsection{Localization calculation I} \label{calc}
We can now evaluate
 the residue contribution of the locus of
$k$-times uniformly thickened stable pairs 
$$P_n(S,\beta)\subset P_{kn}(Y,k\beta)^{\C^*}$$ of \eqref{family}
to the $\com^*$-equivariant
integral $$\int_{[P_{kn}(Y,k\beta)]^{\red}_{\com^*}} 1\,. $$
We will see in Section \ref{van}  that the
contributions of all other $\C^*$-fixed loci 
to the virtual localization formula vanish.

By Proposition \ref{obmat2}, the reduced
virtual class on $P_{kn}(Y,k\beta)^{\C^*}$ 
obtained from $(E\udot_Y)^{\fix}$ matches the reduced virtual class of 
$P_n(S,\beta)$ obtained from the obstruction theory $E\udot_S$.
The virtual normal bundles 
are the same for the reduced and standard
obstruction theories (since the semiregularity map is
$\C^*$-invariant here).

Writing $A_k$ as $\C\oplus\t^{-1\!} A_{k-1}$, we can read off the virtual normal bundle to $P_n(S,\beta)\subset P_{kn}(Y,k\beta)$ 
from Proposition \ref{obsthys}:
$$
N^{\vir}\ =\ (E_S\udot)^\vee\!\otimes\t^{-1\!}A_{k-1}\ \oplus\
E_S\udot\!\otimes\t A_k^*[-1]\,.
$$
After writing $\t A_k^*$ as $\t A_{k-1}^*\oplus\t^k$, the residue
contribution of $P_n(S,\beta)$ to the $\com^*$-equivariant
integral $\int_{[P_{kn}(Y,k\beta)]^{\red}_{\com^*}} 1$ is
$$
\int_{[P_n(S,\beta)]^{\red}}\frac1{e(N^{\vir})}\ =\,\int_{[P_n(S,\beta)]^{\red}}
\frac{e\big(E\udot_S\otimes\t A_{k-1}^*\big)}{e\big((E\udot_S\otimes\t A_{k-1}^*)^\vee\big)}
\,e\big(E\udot_S\otimes\t^k\big)\,.
$$
The rank of $E\udot_S$ is the virtual dimension $$\langle\beta,\beta 
\rangle+n=2h_\beta-2+n$$ of $P_n(S,\beta)$ before reduction. Therefore, the rank of tensor product $E\udot_S\otimes\t A_{k-1}^*$ is
 $(k-1)(2h_\beta-2+n)$, and the quotient in the integrand is 
$$(-1)^{(k-1)(2h_\beta-2+n)}=(-1)^{(k-1)n}\, .$$
Let $t$ denote the $\C^*$-equivariant first Chern class of the
representation $\t$. We have proven the following result.

\begin{Proposition} \label{seven7} The residue contribution of 
$$P_n(S,\beta)\subset P_{kn}(Y,k\beta)^{\C^*}$$ 
to the 
integral $\int_{[P_{kn}(Y,k\beta)]^{\red}_{\com^*}} 1$ is:
\begin{eqnarray*}
\int_{[P_n(S,\beta)]^{\red}}\frac1{e(N^{\vir})} &=&
(-1)^{(k-1)n}\int_{[P_n(S,\beta)]^{\red}}e\big(E\udot_S\otimes\t^k\big)
\\ &=&\frac{(-1)^{(k-1)n}}{kt}\,\int_{[P_n(S,\beta)]^{\red}}c\_{\langle
\beta,\beta\rangle+n+1}(E\udot_S)\,.
\end{eqnarray*}
\end{Proposition}

\subsection{Dependence}
Let $S$ be a $K3$ surface equipped with an ample primitive 
polarization $L$.
Let $\beta \in \text{Pic}(S)$ be a positive class
with respect to $L$,
$$\langle L,\beta\rangle >0\,.$$
If $\beta$ is nonzero and effective, $\beta$ must be positive.
The integral 
\begin{equation}\label{fcc99}
\int_{[P_n(S,\beta)]^{\red}}c\_{\langle
\beta,\beta\rangle+n+1}(E\udot_S)
\end{equation}
 is deformation invariant
as $(S,\beta)$ varies
so long as $\beta$ remains an algebraic class.
Hence, the integral depends {\em only} upon $n$, the divisibility
of $\beta$, and $\langle \beta, \beta \rangle$.

 If $\beta$ is effective, then $H^2(S,\beta)=0$
otherwise $-\beta$ would also be effective by Serre duality.
Hence,
by the results of \cite{KT2}, the integral
\eqref{fcc99} depends {\em only} upon $n$ and 
$$\langle \beta,\beta\rangle = 2h_\beta-2$$
in the effective case.

If $\beta$ is effective and  $h_\beta<0$, then
the integral vanishes
since the virtual number of sections of $\beta$ is negative.
A proof is given below in Section \ref{ktvan} following
\cite{KT1,KT2}.
Finally, if $h_\beta \geq 0$, then $\beta$ must be effective
(since $\beta$ is positive) by Riemann-Roch.

If $\beta$ is not effective, then the integral \eqref{fcc99}
vanishes. In the ineffective case, $h_\beta<0$ must hold.
The discussion of cases is summarized by the following
result, whose final statement will be proved in Proposition \ref{hxxh} below.

\begin{Proposition} \label{nnn333}
For a positive class $\beta \in \text{\em Pic}(S)$, the integral
\begin{equation}
\int_{[P_n(S,\beta)]^{\red}}c\_{\langle
\beta,\beta\rangle+n+1}(E\udot_S)
\end{equation}
depends {\em only} upon
$n$ and $h_\beta$. Moreover, if $h_\beta<0$, the integral
vanishes.
\end{Proposition}

\subsection{Vanishing}\label{ktvan}
Let $S$ be an $K3$ surface with an effective curve class 
$\beta\in \text{Pic}(S)$ satisfying
$$
\beta^2=2h-2 \quad\mathrm{with}\quad h<0\,.
$$
Let $\proj_\beta$ be the linear system of all curves of class $\beta$.
Since $h$ is the reduced virtual dimension of the moduli space
$$
P_{1-h}(S,\beta)=\proj_\beta\,,
$$
the corresponding virtual cycle vanishes,
\beq{train0}
[P_{1-h}(S,\beta)]^{\red}=0\,.
\eeq
We would like to conclude
\beq{train02}
[P_{1-h+k}(S,\beta)]^{\red}=0
\eeq
for all $k$.

If $k<0$, then $P_{1-h+k}(S,\beta)$ is empty, so \eqref{train02}
certainly holds.
If $k>0$,
the moduli space $P_{1-h+k}(S,\beta)$ fibers over $P_{1-h}(S,\beta)$:
$$
P_{1-h+k}(S,\beta)\ \cong\ \Hilb^k(\mathcal C/\proj_\beta) \Rt{\pi} \proj_\beta
\ \cong\ P_{1-h}(S,\beta)\,,
$$
where $\mathcal C\to\proj_\beta$ is the universal curve. 
By \cite{AIK} and \cite[Footnote 22]{KT1},
 the projection $\pi$ is flat of relative dimension $n$. Therefore,
 the vanishing \eqref{train02} follows from \eqref{train0} and 
 Proposition \ref{hxxh} below.{\footnote{The result  was implicit in \cite{KT1} but never actually stated there, so we provide a proof. The result holds more generally for any surface $S$ and class $\beta\in H^2(S,\Z)$ for which $H^2(L)=0$ whenever $c_1(L)=\beta$.}} 
 Since $\pi$ is flat, pull-back is well defined on algebraic cycles. 
Also, as we have noted, $\beta$ effective implies $H^2(S,\beta)=0$.

\begin{Proposition} \label{hxxh} For $H^2(S,\beta)=0$ and $k\geq 0$, we have
$$[P_{1-h+k}(S,\beta)]^{\red}\ =\ \pi^*
[P_{1-h}(S,\beta)]^{\red}\,. $$
\end{Proposition}

\begin{proof} In \cite[Appendix A]{KT1},
 the moduli space $P_{1-h+k}(S,\beta)$ is described by equations as 
follows.{\footnote{Since $H^1(S,\O_S)=0$ the description here is
simpler.}}
Let $A$ be a sufficiently ample divisor on $S$. The inclusion,
 $$\proj_\beta\subset \proj_{\beta+A}\,,$$ 
is described as the zero locus of a section of a vector bundle $E$. Next,
\beq{cutF}
\Hilb^k(\mathcal C/\proj_\beta)\ \subset\ \proj_\beta\times\Hilb^kS
\eeq
is described as the zero locus of a section of a bundle $F$ which extends 
to $\proj_{\beta+A}\times\Hilb^kS$.

Let $\mathcal A$ denote the nonsingular ambient space 
$\proj_{\beta+A}\times\Hilb^kS$ which contains our moduli space 
$$P=P_{1-h+k}(S,\beta)\cong \Hilb^k(\mathcal C/\proj_\beta)\,.$$
The above
description then defines a natural virtual cycle on $P$ via
\beq{redA}
[P]^{\red}\ =\ \big[s(P\subset\mathcal A)c(\pi^*E)c(F)\ \cap\ [P]\big]_{h+k}\, ,
\eeq
the refined top Chern class of $\pi^*E \oplus F$ on $P$.
Here,
$$h+k\ =\ \dim\mathcal A-\rk E-\rk F$$
is the virtual dimension of the construction. 
By the main result of \cite[Appendix A]{KT1}, the
class \eqref{redA} is, as the notation suggests, equal to the reduced virtual cycle of $P$.

By \cite{AIK} and \cite[Footnote 22]{KT1}, the section of $F$ cutting out \eqref{cutF} is in fact regular. Hence,  the resulting normal cone
$$
C_{P\subset\proj_\beta \times\Hilb^kS}\cong F
$$
is locally free and isomorphic to $F$. 
We have the following exact sequence of cones{\footnote{Since $H^1(S,\O_S)=0$ and the Hilbert scheme of curves is just the nonsingular
 linear system $\proj_\beta$, all three terms are locally free. In general only the first is.}}
:
$$
0\To C_{P\subset\proj_\beta\times\Hilb^kS}\To C_{P\subset\mathcal A}\To
\pi^*C_{\proj_\beta\subset \proj_{\beta+A}}\To0\,.
$$
After substitution in \eqref{redA}, we obtain
\begin{align*}
[P]^{\red}\ =&\ \Big[s(F)\pi^*s\big(\proj_\beta\subset\proj_{\beta+A}\big)\pi^*c(E)c(F)\ \cap\ \pi^*\big[\proj_\beta\big]\Big]_{h+n} \\
=&\ \pi^*\Big(\Big[s\big(\proj_\beta\subset\proj_{\beta+A}\big)c(E)\ \cap\ 
\big[\proj_\beta\big]\Big]_h\,\Big)
\\ =&\ \pi^*\big[\proj_\beta\big]^{\red}\,. \hfill\qedhere
\end{align*}

\end{proof}

\section{$K3\times \C$: Vanishing} \label{van}
\subsection{Overview}
 We will show the components of $(P_Y)^{\C^*}$ which do not correspond
to the thickenings studied in Section \ref{k3tcl}
do {\em not} contribute to the localization formula.

Recall first the proof  
of the vanishing of the ordinary (non-reduced) $\C^*$-localized invariants of 
$$Y=S\times\C\,.$$ Translation along the $\C$-direction in $Y$ induces a vector field on $P_Y$ which has $\C^*$-weight 1. By the symmetry of the obstruction theory, such translation  induces a $\C^*$-weight 0 cosection: a surjection from the obstruction sheaf $\Omega_{P_Y}$ to $\O_{P_Y}$. Since 
the cosection is $\C^*$-fixed, it descends to a cosection for the $\C^*$-fixed obstruction theory on $(P_Y)^{\C^*}$, forcing the virtual cycle to vanish \cite{KiemLi}.

To apply the above strategy to the \emph{reduced} obstruction theory,
 we need to find \emph{another} weight 1 vector field on the 
moduli space. We will describe such a vector field which is proportional to the original translational vector field along $(P_Y)^{\C^*}\subset P_Y$ precisely on the components of uniformly thickened stable pairs of Section \ref{k3tcl}. 
On the other components of $(P_Y)^{\C^*}$,
the  linear independence of the two weight 1 vector fields
forces the reduced localized invariants to vanish.

\subsection{Basic model} \label{lm}
Our new vector field will again arise from an action in the $\C$-direction on 
$Y$, pulling apart a stable pair supported on a thickening of the central fiber $S\times\{0\}$. 
The $S$-direction plays little role, so we start by explaining the basic
 model on $\C$ itself. For clarity, we will here use the notation
 $$\C_x=\Spec\C[x]\,,$$  where the subscript denotes the corresponding parameter. The space $\C_x$ carries its usual $\C^*$-action, with the coordinate function
$x$ having weight $-1$.

Consider the $k$-times thickened origin $B_k\subset\C_x$. We wish to fix $B_{k-1}\subset B_k$ and move the remaining point away through $\C_x$ at unit speed. In other words, we consider the flat family of subschemes of $\C_x$ parameterized by $t\in\C_t$ given by 
\beq{ffam}
Z=\{x^{k-1}(t-x)=0\}\subset\C_x\times\C_t\,.
\eeq
Specializing to $t=0$ indeed gives the subscheme $\{x^k=0\}=B_k$, while for $t\ne0$ we have $\{x^{k-1}=0\}\sqcup\{x=t\}$.


\subsection{Extension class}
Consider $\O_Z$ as a flat family 
of sheaves over $\C_t$ defining a deformation of fiber $\O_{B_k}$ over $t=0$.
After restriction to $\Spec\C[t]/(t^2)$,
we obtain a first order deformation $\O_Z/(t^2)$ of the sheaf 
$$\O_Z/(t)=\O_{B_k} \,.$$
Such deformations are  classified by an element $e$ of the group
\beq{ext1}
\Ext^1_{\C_x}(\O_{B_k},\O_{B_k})
\eeq
described as follows. The exact sequence
$$
0\To\O_Z/(t)\Rt{t}\O_Z/(t^2)\To\O_Z/(t)\To0\,,
$$
is isomorphic to
\beq{nightgarden}
0\To\O_{B_k}\Rt{t}\frac{\C[x,t]}{(t^2,x^{k-1}(x-t))}\To\O_{B_k}\To0\,.
\eeq
Considering \eqref{nightgarden} as a sequence of $\C[x]$-modules (pushing it down by $\pi_2:\C_x\times\Spec\C[t]/(t^2)\to\C_x\,$)
an extension class $e$ in \eqref{ext1} is determined.
\medskip

Using the resolution
$$
0\To\O_{\C_x}(-B_k)\Rt{x^k}\O_{\C_x}\To\O_{B_k}\To0
$$
of $\O_{B_k}$ we compute \eqref{ext1} as
\begin{eqnarray} \nonumber
\Ext^1_{\C_x}(\O_{B_k},\O_{B_k}) &\cong& \Hom(\O_{\C_x}(-B_k),\O_{B_k}) \\ \label{eext}
&\cong& H^0(\O_{B_k}(B_k))\ \cong\ H^0(\O_{B_k})\otimes\t^k.
\end{eqnarray}

\begin{Proposition} \label{local}
The class $e\in\Ext^1_{\C_x}(\O_{B_k},\O_{B_k})$ of the first order deformation $\O_Z/(t^2)$ of $\O_{B_k}$ is
$$
x^{k-1}\otimes \t^k\,\in\,H^0(\O_{B_k})\otimes\t^k.
$$
It coincides (to first order) with the deformation given by moving $\O_{B_k}$ by the translational vector field $\frac1k\partial_x$. In particular, $e$ has $\C^*$-weight $1$.
\end{Proposition}

\begin{proof}
Consider the generator of $\Hom(\O_{\C_x}(-B_k),\O_{B_k})$ multiplied by $x^{k-1}$. From the description \eqref{eext}, it
 corresponds to the extension $E$ coming from the pushout diagram
\beq{extn}
\xymatrix@R=18pt{
0 \ar[r]& \O_{\C_x}(-B_k) \ar[d]_{x^{k-1}}\ar[r]^(.65){x^k}& \O_{\C_x} \ar[r]\ar[d]& \O_{B_k} \ar[r]\ar@{=}[d]& 0 \\
0 \ar[r]& \O_{B_k} \ar[r]& E \ar[r]& \O_{B_k} \ar[r]& 0\,.\!\!}
\eeq
On the other hand, $e$ is defined by the push-down to $\C_x$ of the extension \eqref{nightgarden}. The latter sits inside the diagram
$$
\xymatrix@R=18pt{
0 \ar[r]& \C[x] \ar[d]_{x^{k-1}}\ar[r]^(.53){x^k}& \C[x] \ar[r]\ar[d]^{\pi^*}& \C[x]/(x^k) \ar[r]\ar@{=}[d]& 0 \\
0 \ar[r]& \C[x]/(x^k) \ar[r]^(.35)t& \C[x,t]/(t^2,x^{k-1}(x-t)) \ar[r]& \C[x]/(x^k) \ar[r]& 0\,.\!\!}
$$
Here the central vertical arrow $\pi^*$ takes a polynomial in $x$ to the same polynomial in $x$ (with no $t$-dependence). This is indeed a map of $\C[x]$-modules (though not $\C[x,t]$-modules) and makes the left hand square commute since $x^k=x^{k-1}t$ in the ring 
$$\O_Z/(t^2)=\C[x,t]/(t^2,x^{k-1}(x-t))\,.$$
Since the second diagram is isomorphic to the first diagram \eqref{extn}, we find 
 $e$ is indeed $x^{k-1}\otimes \t^k$.
%

\medskip
Next we observe that moving $\O_{B_k}$ by the translation vector field $\frac1k\partial_x$ yields the structure sheaf of the different family
$$
\left\{(x-t/k)^k=0\right\}\subset\C_x\times\C_t\,.
$$
Restricting to $\Spec\C[t]/(t^2)$ gives the first order deformation
$$
\frac{\C[x,t]}{\Big(t^2,\big(x-t/k\big)^k\Big)}\ =\
\frac{\C[x,t]}{(t^2,x^k-tx^{k-1})}\,,
$$
the same as $\O_Z/(t^2)$. \medskip

Finally, the $\C^*$-weight of $\frac 1k\partial_x$ is clearly 1. More directly, $x^{k-1}\otimes \t^k$ has weight $-(k-1)+k=1$.
\end{proof}

The conceptual reason for the surprising result of Proposition \ref{local}
 is that, \emph{to first order}, weight $1$ deformations only see the corresponding 
deformation of the center of mass of the subscheme. 
The two deformations in Proposition \ref{local} clearly deform the center of mass in the same way.

We will next apply a version of the above deformation to $\C^*$-fixed stable pairs. 
The first order part of the deformation will describe a weight $1$ vector field on $P_Y$
along $(P_Y)^{\C^*}\subset P_Y$ and thus a $\C^*$-invariant cosection of the obstruction theory.

On the stable pairs which are \emph{uniformly} thickened as in Section \ref{uniform}, 
Proposition \ref{local} will show the new vector field to be proportional to the standard vector field given by the translation $\partial_x$. Thus our cosection is proportional to the cosection
 we have already reduced by, and provides us nothing new. Hence, the nonzero contributions of Section \ref{calc} are permitted.

For non uniformly thickened  stable pairs, however, the new
vector field will be seen to be 
linearly independent of the translational vector field. 

\subsection{Full model} \label{fullmod}
The basic model of Section \ref{lm} gives a deformation of the $\C[x]$-modules $A_k=\C[x]/(x^k)$. We now extend this to describe a deformation of any $\C^*$-equivariant torsion $\C[x]$-module $M$ which is a (possibly infinite) direct sum of finite dimensional $\C^*$-equivariant $\C[x]$-modules.
By the classification of modules over PIDs, $M$ is a direct sum of $\t^j$-twists of the standard modules 
$A_k=\C[x]/(x^k)$. Since these were treated in Section \ref{lm}, the extension
is a simple matter. 
However, by describing our deformations intrinsically, we will be able 
to apply the construction
to $\C^*$-fixed stable pairs $(F,s)$ on $S\times\C_x$. 
Let $$U\subset S$$ be an affine open set. Then, $F|_{U\times\C_x}$ 
is equivalent to a $\C^*$-equivariant torsion $\C[x]$-module 
carrying an action of the ring $\O(U)$. 
The model developed here will sheafify over $S$ and
determines a deformation of $(F,s)$.

Since the sheaf $F|_{U\times\C_x}$ has only finitely many weights (all
nonpositive), 
we restrict attention to torsion $\C^*$-equivariant $\C[x]$-modules $M$ with
weights lying in the interval $[-(k-1),0]$ for some $k\ge1$. Examples include $A_k$ and $A_j\t^{-(k-j)}$ for $j\le k$. 
Write
$$
M=\bigoplus_{i=0}^{k-1} M_i
$$
as a sum of weight spaces, where $M_i$ has weight $-i$. 
Multiplication by $x$ is encoded in the weight $(-1)$ operators
$$
X\colon M_i\To M_{i+1}\,.
$$
Since $X$ annihilates the most negative weight space, $M_{k-1}\subset M$ is 
an equivariant $\C[x]$-submodule. We will define a deformation which moves $M_{k-1}$ away at unit speed while leaving the remaining $M/M_{k-1}$ fixed.

To do so, notice the basic
 model \eqref{ffam} of Section \ref{lm} can be described as follows. 
Take the direct sum of the structure sheaves of the two irreducible components of $Z$ (or, equivalently, the structure sheaf $\O_{\,\overline{\! Z}}$ of the normalization of $Z$), then $\O_Z$ is the subsheaf of sections which agree over the intersection $\Delta(B_{k-1})$ of the two components:
\beq{localmodel}
\xymatrix@=35pt{
\O_Z\ =\ \ker\Big(\pi^*\O_{B_{k-1}}\oplus\Delta_*\O \ar[r]^(.65){(r,-r)} & \Delta_*\O_{B_{k-1}}\Big)\,.}
\eeq
Here, $\Delta_*\O_{B_{k-1}}=\Delta_*A_{k-1}$ and $r$ denotes restriction to $\Delta(B_{k-1})$. Finally
\begin{align*}
\pi\colon&\C_x\times\C_t\To\C_x\,, \\
\Delta\colon&\C\Into\C_x\times\C_t
\end{align*}
are the projection and the inclusion of the diagonal respectively.

We define the $\C^*$-equivariant $\C[x,t]$-module $\widetilde M$ to be
the kernel of the map
\begin{align} \nonumber
\pi^*(M/M_{k-1})\,\oplus\, & \Delta_*(M_{k-1}\t^{k-1}\otimes_\C \C[x]) \\
&\xymatrix@=50pt{&\ar[r]^(.23){(\psi\,\circ\,r,\ -r)} & 
\Delta_*(M_{k-1}\t^{k-1}\otimes_\C A_{k-1})\,,} \label{model}
\end{align}
where $\psi$ is the map
$$
\xymatrix@=90pt{M/M_{k-1}\,=\,\bigoplus_{i=0}^{k-2}M_i
\ar[r]^{\bigoplus_i X^{k-1-i}\t^{k-1}\otimes\,x^i} &
M_{k-1}\t^{k-1}\otimes_\C A_{k-1}\,.}
$$
By construction this is a weight 0 map of equivariant $\C[x]$-modules.

By splitting $M$ into direct sums of irreducible modules $A_n\t^m$, comparing with \eqref{localmodel}, and using Proposition \ref{local},
we obtain the following result.

\begin{Proposition} \label{tildeM}
The sheaf $\widetilde M$ defined by \eqref{model} is flat over $\C_t$ and
 specializes to $\widetilde M/t\widetilde M=M$ over $t=0$. The first order deformation
$$
e\in\Ext^1_{\C_x}(M,M) \quad\mathrm{classifying}\quad \widetilde M/t^2\widetilde M
$$
is proportional to the first order translation deformation $\partial_x$ on any irreducible module $M$ with weights in $[-(k-1),0]$ as follows:
\begin{itemize}
\item For $M=A_k=\O_{B_k}$ we have $e=\partial_x/k$.
\item For $M=A_j\t^{-(k-j)}=\O_{B_j}\t^{-(k-j)}$ with $j\le k$ we have $e=\partial_x/j$.
\item For $M$ with $M_{k-1}=0$ we have $e=0$. \hfill$\square$
\end{itemize}
\end{Proposition}

Proposition \ref{tildeM}
is the foundation of our localization to \emph{uniformly thickened} stable pairs. The above deformation applied to $\C^*$-fixed stable pairs will describe a weight $1$ vector field on $P_Y$ along $(P_Y)^{\C^*}\subset P_Y$, and thus a $\C^*$-invariant cosection of the obstruction theory.


For any stable pair which is not uniformly thickened, 
the new vector field acts as a translation which operates at different speeds 
along different parts of the stable pair.
The corresponding cosection is therefore linearly independent of the 
pure translation $\partial_x$  and descends to give a nowhere vanishing
 cosection of the reduced obstruction theory. 

\subsection{Second cosection} \label{vf}
Fix a component of $(P_Y)^{\C^*}$ over which the $\C^*$-fixed stable pairs are $k$-times thickened, supported on 
$$(P_Y)^{\C^*}\!\times S\times B_n\ \subset\ (P_Y)^{\C^*}\!\times S\times\C_x$$ 
for $n=k$ but not for any $n<k$.

We now apply the results of Section \ref{fullmod} to the universal sheaf
$$
\FF \quad\mathrm{on}\quad (P_Y)^{\C^*}\!\times S\times\C_x
$$
to produce a flat deformation over $(P_Y)^{\C^*}\!\times S\times\C_x\times\C_t$ by the formula \eqref{model}. The universal section $s$ of $\FF$ also deforms to the $\C^*$-invariant section
$$
\big([s],\,X^{k-1}s.\t^{k-1}\otimes1\big)
$$
of \eqref{model}. Restricting to $\Spec\C[t]/(t^2)$ defines a $P_Y$ tangent vector field along $(P_Y)^{\C^*}$ of weight 1:
$$
v\,\in\,\t^*\otimes H^0\Big((P_Y)^{\C^*}\!,\,\ext^1_{\pi_P}(\II_Y,\II_Y)\_0\Big).
$$
From $v$ and the isomorphism $\omega_Y\cong\t^{-1}$,
 we construct a weight 0 cosection over $(P_Y)^{\C^*}$:
$$
\xymatrix{
\ext^2_{\pi_P}(\II_Y,\II_Y)\_0
\ar[r]^(.42){\cup v}& \ext^3_{\pi_P}(\II_Y,\II_Y)\otimes\t^{-1} 
\ar[r]^(.52) {\ \ \ \ \ \ \tr}
&\O_{(P_Y)^{\C^*}}\,,}
$$
where the final arrow is dual to the identity
$$\O_{(P_Y)^{\C^*}} \rightarrow  \hom_{\pi_P}(\II_Y,\II_Y)$$
under the Serre duality of \eqref{pot}.
Combined with our standard cosection, we obtain a map
\beq{2cosection}
\xymatrix@=125pt{
\ext^2_{\pi_P}(\II_Y,\II_Y)\_0
\ar[r]^(.53){\tiny{\tr(\ \cdot\ \cup v)\,\oplus\,
\tr\!\big(\,\cdot\ \cup(\partial_x\ip \At(\II_Y))\!\big)}}
& \O_{(P_Y)^{\C^*}}^{\oplus2}\,.}
\eeq

\begin{Proposition} \label{dubred} The map \eqref{2cosection} is surjective over $(P_Y)^{\C^*}$ away from stable pairs which are uniformly thickened as in Section \ref{uniform}.
\end{Proposition} 

\begin{proof} By the Nakayama Lemma, we can check surjectivity at  closed points $(F,s)$. By Serre duality, we need only show that if
our two elements of the Zariski tangent space to $P_Y$ at $(F,s)$,
$$
v,\ \partial_x\ip\At(I\udot_Y)\,\in\Ext^1(I\udot_Y,I\udot_Y)_0\,,
$$
are linearly dependent then $(F,s)$ is uniformly thickened.

Pick an affine open set $U\subset S$ and consider the restriction of $(F,s)$ to $U\times\C_x$ as a $\C^*$-equivariant $\C[x]$-module. Since $F$ is $k$-times thickened and pure, 
the support $C$  of $F$ is also $k$-times thickened over the open set where $\O_C\rt{s}F$ is an isomorphism. 
In particular the equivariant $\C[x]$-module $F|_{U\times\C_x}$ contains copies of $A_k$ as summands. By Proposition \ref{local}, the deformation $v$ is the same as $\partial_x/k$ on $A_k$.

Therefore if $v$ and $\partial_x$ are linearly dependent at $(F,s)$ then in fact $v$ must equal $\partial_x/k$ at $(F,s)$. In particular, by Proposition \ref{tildeM} all of the irreducible equivariant $\C[x]$-submodule summands of $F|_{U\times\C_x}$ are isomorphic to $A_k$ (for any $U$). Therefore we get an isomorphism
\beq{uth}
F\Rt{\sim}F_{k-1}\t^{k-1}\otimes_\C A_k
\eeq
by the map defined in terms of the weight space decomposition of $F$ as
$$
\xymatrix@=75pt{\bigoplus_{i=0}^{k-1}F_i \ar[r]^(.42){X^{k-1-i}\t^{k-1}\otimes \,x^i} & F_{k-1}\t^{k-1}\otimes_\C A_k\,.}
$$
The isomorphism \eqref{uth} implies $F$ is uniformly thickened.

Finally, the $\C^*$-invariant section maps 1 to a degree 0 element of the module $F|_{U\times\C_x}\cong F|_{U\times\{0\}}\otimes_\C A_k$. Such elements are of the form $f\otimes1$. Hence, the elements are pulled back from $S$ to $S\times B_k$,
and the pair $(F,s)$ is uniformly thickened.
\end{proof}

\begin{Corollary} \label{tfred}
The invariants calculated in Section \ref{calc} are the only nonzero contributions to the reduced localized invariants of $Y=S\times\C$.
\end{Corollary}

\begin{proof}
We work on a component of $(P_Y)^{\C^*}$ parameterizing pairs which are \emph{not} uniformly thickened. (By $\C^*$-invariance, one pair in a component is uniformly thickened if and only if they all are.)
 
Since \eqref{2cosection} is $\C^*$-invariant it factors through the $\C^*$-fixed part of the obstruction sheaf, which by \cite{GP} is the obstruction sheaf of the induced perfect obstruction theory on $(P_Y)^{\C^*}$. The \emph{reduced} obstruction sheaf is given by taking the kernel of the first factor of this map:
$$
\Ob^{\red}_{(P_Y)^{\C^*}}=\ker\Big(\ext^2_{\pi_P}(\II_Y,\II_Y)^{\C^*}_0
\xymatrix@C=80pt{\ar[r]^{\tr\!\big(\,\cdot\ \cup(\partial_x\ip \At(\II_Y))\!\big)}&}
\O_{(P_Y)^{\C^*}}\Big).
$$
Proposition \ref{dubred} then states that \eqref{2cosection} gives a surjection
$$
\Ob^{\red}_{(P_Y)^{\C^*}}\xymatrix@C=40pt{\ar[r]^{\tr(\ \cdot\ \cup v)}&}
\O_{(P_Y)^{\C^*}}.
$$
Therefore the reduced class vanishes.
\end{proof}

\subsection{Localization calculation II}\label{lc22}
The results of Sections \ref{k3tcl} and \ref{van} together yield
a complete localization calculation.

Let $S$ be a $K3$ surface equipped with a ample primitive 
polarization $L$.
Let $\alpha \in \text{Pic}(S)$ be a primitive and positive class.{\footnote{No further conditions
are placed on $\alpha$: the Picard rank of $S$ may be high and
$\alpha$ may be the sum of effective classes.}
Define
\begin{equation}\label{kzz21}
\big\langle 1 \big\rangle_{Y,m\alpha}^{\red}=
\sum_{n} y^n
\int_{[P_n(Y,m\alpha)]^{\red}_{\com^*}} 1 \,. 
\end{equation}

The integral on the right side of \eqref{kzz21}, denoting
the $\com^*$-residue, is well-defined since the
$\com^*$-fixed loci of $P_n(Y,m\alpha)$ are compact.
Since the {\em reduced} virtual dimension of 
$P_n(Y,m\alpha)$ is 1, the residues are of degree $-1$ in $t$ (cf. Proposition \ref{seven7}),
$$\big\langle 1 \big\rangle_{Y,m\alpha}^{\red}\in 
\frac{1}{t}\mathbb{Q}(\!(y)\!)\,.$$

\begin{Proposition} \label{zzzzz}
Let $\alpha\in \text{Pic}(S)$ be a primitive and positive class. Then
$\big\langle 1 \big\rangle_{Y,\alpha}^{\red}$ depends
{\em only} upon 
$$\langle \alpha,\alpha\rangle=2h-2\,.$$ Moreover, if $h<0$,
then $\big\langle 1 \big\rangle_{Y,\alpha}^{\red}=0$.
\end{Proposition}
\begin{proof}
By the vanishing of Corollary \ref{tfred} and the
residue formula of Section \ref{calc}, we have
\begin{equation}\label{hbb3}
\big\langle 1 \big\rangle_{Y,\alpha}^{\red} =
\frac{1}{t} 
\sum_n y^n
\int_{[P_n(S,\alpha)]^{\red}} c_{\langle \alpha,\alpha \rangle+n+1}(E^\bullet_S) \,.
\end{equation}
By Proposition \ref{nnn333}, the integral over 
$[P_n(S,\alpha)]^{\red}$ occurring in \eqref{hbb3} depends
only upon $n$ and $\langle \alpha,\alpha \rangle$ and vanishes
if $h<0$. 
\end{proof}

To isolate the dependence of Proposition \ref{zzzzz}, we define, 
for primitive and
positive $\alpha$,
\begin{equation}\label{v231v}
I_h = \big\langle 1 \big\rangle^{\red}_{Y,\alpha}\,, \ \ \langle
\alpha,\alpha \rangle = 2h-2\,.
\end{equation}
In case $\alpha$ is {\em also} irreducible (which we may assume), the moduli
space $P_n(S,\alpha)$ is nonsingular \cite{ky,PT2}. 
The evaluation of \eqref{hbb3}
reduces to the Euler characteristic calculation of Kawai and
Yoshioka \cite{ky} as explained in the Appendix C of \cite{PT2}
and reviewed in Section \ref{kye} below.

\begin{Proposition} \label{mc11}
The following multiple cover formula holds:
$$\big\langle 1 \big\rangle_{Y,m\alpha}^{\red} =
 \sum_{k|m} \frac{1}{k} I_{\frac{m^2}{k^2}(h-1)+1}(-(-y)^k)\,.
$$
\end{Proposition}

\begin{proof}
The result follow from
the vanishing of Corollary \ref{tfred}, the
residue formula of Section \ref{calc}, the dependence
result of Proposition \ref{nnn333},
and the definition \eqref{v231v}.
\end{proof}

\subsection{Kawai-Yoshioka evaluation} \label{kye}
  Let $P_n(S,h)$ denote the nonsingular moduli space 
of stable pairs for an irreducible
class $\alpha$ satisfying
$$
2h-2= \langle \alpha, \alpha \rangle .
$$
The cotangent bundle $\Omega_{P}$
of
the moduli space $P_n(S,h)$ is the obstruction
bundle of the reduced theory. Since the
dimension of $P_n(S,h)$ is $2h-2+n+1$,
$$I_h(y) = \frac{1}{t}\sum_{n} (-1)^{2h-1+n} e(P_n(S,h))\  y^n.$$

The topological Euler characteristics of 
$P_n(S,h)$ have been calculated by Kawai-Yoshioka.
By Theorem 5.80 of \cite{ky},
\begin{multline*}
\sum_{h=0}^\infty \sum_{n=1-h}^\infty  
e(P_n(S,h))\  y^n q^h = \\
\left(\sqrt{y}-\frac{1}{\sqrt{y}}\right)^{-2}\
\prod_{n=1}^\infty \frac{1}{(1-q^n)^{20} (1-yq^n)^2(1-y^{-1}q^n)^2} \,.
\end{multline*}
For our pairs invariants, we require the signed Euler characteristics,
$$
\sum_{h=0}^\infty I_h(y) \ q^h = 
\frac{1}{t}\sum_{h=0}^\infty \sum_{n=1-h}^\infty (-1)^{2h-1+n} 
e(P_n(S,h))\  y^n q^h. 
$$
Therefore,
$\sum_{h=0}^\infty tI_h(y) \ q^h$ equals
$$
-\left(\sqrt{-y}-\frac{1}{\sqrt{-y}}\right)^{-2}\
\prod_{n=1}^\infty \frac{1}{(1-q^n)^{20} (1+yq^n)^2(1+y^{-1}q^n)^2} \,.
$$
The above formula implies $tI_h(y)$ is the Laurent expansion
of a rational function of $y$.

\section{Relative theory and the logarithm}
\label{relt}
\subsection{Overview} \label{rtov}
Our goal here is to define and study the
analogue $\widetilde{R}_{n,\beta}$ for stable pairs
of the Gromov-Witten integrals $R_{g,\beta}$ 
associated to $K3$ surfaces. Though the definition is via
the stable pairs theory of $K3$-fibrations, the main idea is
to move the integration to  the rubber of an associated
relative geometry. The interplay with various rubber theories
allows for a geometric interpretation of the logarithm occurring
in the definition of $\widetilde{R}_{n,\beta}$. 

\subsection{Definition}  \label{c123}
Let $S$ be a $K3$ surface equipped with an ample 
primitive polarization
$L$. 
Let $\alpha\in \text{Pic}(S)$ be a primitive, positive{\footnote{Positivity,
$\langle L, \alpha \rangle >0$,
is with respect to the polarization $L$.}} class 
not proportional to $L$ with norm square
 $$\langle \alpha,\alpha \rangle =2h-2\,. $$
Let ${m}>0$ be an integer. By replacing $L$ with $\widehat{L}= xL + \alpha$
for large $x$, we can assume $\widehat{L}$ is ample and primitive,
$\alpha$ is positive with respect to $\widehat{L}$, and the inequality
\begin{equation}\label{vddx34}
\langle \widehat{L}, \widehat{L} \rangle > {m} \langle\widehat{L},\alpha\rangle \ 
\end{equation}
holds.
Condition \eqref{vddx34} forbids effective summands of $m\alpha$ 
to be multiples
of $\widehat{L}$.

Let $(S,\widehat{L})\in \mathcal{M}$
denote the corresponding moduli 
point in the associated moduli space of
polarized $K3$ surfaces.{\footnote{We require $\widehat{L}$ to be ample,
so $\mathcal{M}$ is an open set of the moduli of quasi-polarized
$K3$ surfaces considered in Section \ref{nnll}.}
Since no effective summand  $\gamma$ of $m\alpha$ 
is a multiple of $\widehat{L}$,  
every such summand corresponds to a nondegenerate
local Noether-Lefschetz locus $\text{NL}(\gamma)$ near $(S,\widehat{L})$
of codimension 1.
Let $$\Delta \subset \mathcal{M}$$ be a quasi-projective
curve passing through
$(S,\widehat{L})$ and transverse to the local Noether-Lefschetz loci 
corresponding to all the (finitely many)
effective summands of  $m\alpha$.

Associated to $\Delta$ is a
3-fold $X$ fibered in polarized $K3$ surfaces,
\begin{equation}\label{gqrt}
\epsilon: X \rightarrow (\Delta,0)
\end{equation}
We summarize the conditions
we have as follows: 
\begin{enumerate}
\item[(i)] $\Delta$ is a nonsingular quasi-projective curve,
\item[(ii)] $\epsilon$ is smooth, projective, and $\epsilon^{-1}(0) \stackrel{\sim}{=} S$,
\item[$(\star)$] for every
 {\em effective} decomposition 
$$m\alpha=\sum_{i=1}^l \gamma_i
\in \text{Pic}(S)\,,$$
the  local Noether-Lefschetz locus $\text{NL}(\gamma_i)\subset\Delta$ corresponding to
each  class $\gamma_i \in \text{Pic}(S)$ is the
{\em reduced} point $0\in \Delta$.
\end{enumerate}
The class $\alpha\in \text{Pic}(S)$ is ${m}$-rigid
with respect to the $K3$-fibration
 $\epsilon$.

 In Section \ref{spaks}, ${m}$-rigidity was defined
for effective $\alpha$. The above definition  is for positive
$\alpha$. Since effective implies positive, the definition here
extends the definition of Section \ref{spaks}.

At the special fiber $\epsilon^{-1}(0) \cong S$,
the   Kodaira-Spencer class 
\beq{kappa}
\kappa\in H^1(T_S)
\eeq
associated to $\epsilon$ 
is the extension class of the exact sequence
$$
0\To T_S\To T_X|_S\To\O_S\To0\,.
$$
After fixing a holomorphic symplectic form $\sigma\in H^0(\Omega^2_S)$,
we obtain the (1,1) class
$$
\kappa\ip\sigma\in H^1(\Omega_S)\,.
$$
The transversality of $\Delta$ to  the local 
Noether-Lefschetz locus corresponding
to the class $\gamma \in \text{Pic}(S)$ is equivalent to the condition
\beq{twitor}
\int_\gamma\kappa\ip\sigma\ne0 \,.
\eeq

Let $\text{Eff}({m}\alpha) \subset \text{Pic}(S)$
denote the subset of effective summands of 
$m\alpha$. 
By property $(\star)$,
there is a compact, open, and closed component  
$$P_n^\star(X,\gamma) \subset P_n(X,\gamma)$$ 
parameterizing 
 stable pairs
supported set-theoretically over the point 
$0\in \Delta$ for {every} effective summand $\gamma\in 
\text{Eff}({m}\alpha)$.

\vspace{7pt}
\noindent{\bf Definition.} 
{\em Let $\alpha\in \text{\em Pic}(S)$ be a primitive, positive class.
Given a family $\epsilon: X \rightarrow (\Delta,0)$
satisfying conditions (i),(ii), and $(\star)$ for $m\alpha$, let
\begin{multline}
\label{fbbb4}
\sum_{n\in \mathbb{Z}}
\widetilde{R}_{n,m\alpha}(S)\ q^n  = \\
\text{\em Coeff}_{v^{m\alpha}} \left[
\log\left(1 + \sum_{n\in \mathbb{Z}} \sum_{\gamma\in \text{\em Eff}({m}\alpha)} 
q^n v^{\gamma}
\int_{[P_n^\star(X,\gamma)]^{\vir}} 1\right)\right]\,.
\end{multline}}
\vspace{7pt}

An immediate geometric consequence of the above definition is
the following vanishing statement: {\em if $m\alpha \in \text{\em Pic}(S)$
is not effective, then $\widetilde{R}_{n,m\alpha}(S)=0$ for all $n$}.

The main result of our study here  will be a geometric interpretation of the
logarithm on the right. As a consequence, we will see
that $\widetilde{R}_{n,m\alpha}(S)$
depends {\em only} upon $n$, $m$, and 
$\langle \alpha,\alpha\rangle$ and {\em not}
upon $S$ nor the family $\epsilon$. We therefore drop $S$
from the notation. 
Also, $\widetilde{R}_{n,m\alpha}$ is well-defined
for all $m$ by 
the existence of ${m}$-rigid families $\epsilon$
for suitable $\widehat{L}$ (as we have constructed).

The integrals over $P^\star_n(X,m\alpha)$ appearing
on the right side of \eqref{fbbb4} play a central role,
$$P_{n,\gamma}^\star(X)= \int_{[P^\star_n(X,\gamma)]^{\vir}} 1\,, \ \ \ \ \gamma \in 
\text{Eff}({m}\alpha)\,.
$$

\subsection{Relative moduli spaces}
Let $\alpha\in \text{Pic}(S)$ be a primitive class. 
Let $\epsilon$ be a family of polarized $K3$ surfaces 
$$\epsilon: X \rightarrow (\Delta,0)$$
for which $\alpha$ is positive and ${m}$-rigid.
We will consider the relative geometry associated to
$$X/S = X/\epsilon^{-1}(0)\,.$$ 

Let $\beta \in \text{Eff}({m}\alpha) \subset \text{Pic}(S)$.
We recall the definition \cite{LiWu, PT1} of the moduli space $P_n(X/S,\beta)$ 
parameterizing
stable relative pairs
\begin{equation}\label{vyq}
\O_{X[k]} \stackrel{s}{\longrightarrow} F
\end{equation}
on $k$-step degenerations\footnote{$X[k]$ is the union of $X$ with a chain of $k\ge0$ copies of $S\times\PP^1$, where the $i$th copy of $S\times\PP^1$ is attached along $S\times\{\infty\}$ to the $(i+1)$st along $S\times\{0\}$. Contracting the chain to $S\subset X$ defines a projection $X[k]\to X$ with automorphism group $(\C^*)^k$. There is a distinguished divisor $S_\infty\subset X[k]$ at $S\times\{\infty\}$ in the extremal component. If $k=0$, then
$S_\infty$ is just $S\subset X$.} $X[k]$ of $X$ along $S$
\cite{LiRelative}. Here, $F$ is a sheaf on $X[k]$ with
$$\chi(F)=n$$
and support $[F]$ which pushes down to the class
$\beta\in H_2(X,\Z)$.
The pair $(F,s)$ satisfies the following stability conditions:
\begin{enumerate}
\item[(i)] $F$ is pure with finite locally free resolution,
\item[(ii)] the higher derived functors of the
restriction of $F$ to the singular loci of $X[k]$, and the divisor at infinity, vanish,
\item[(iii)] the section $s$ has 0-dimensional cokernel supported
away from the singular loci of $X[k]$ and away from $S_\infty$.
\item[(iv)] the pair \eqref{vyq} has only finitely many automorphisms covering
the automorphisms of $X[k]/X$.
\end{enumerate}
The moduli space $P_n(X/S,\beta)$ is a Deligne-Mumford stack with a perfect obstruction theory which we describe in Section \ref{defy}.

In our situation, $\beta$ is a fiber class and the nearby fibers $X_{t\ne0}$ contain no curves of class $\beta$ (by the transversality condition 
of $\epsilon$). Hence, there is a compact, open and closed substack
$$P^\star_n(X/S,\beta) \subset P_n(X/S,\beta)$$
parameterizing 
stable pairs $(F,s)$ lying over the central fiber. 
By condition (ii) of stability, the target must be bubbled, 
with $(F,s)$ living on some $X[k]$ with $k\ge1$. Restricted to
 the $i^{th}$ bubble $S\times\PP^1$, $(F,s)$ determines  a stable pair
\beq{bubbled}
(F_i,s_i) \ \text{ disjoint from }\ S\times\{0,\infty\}\,,
\eeq
with invariants{\footnote{For the class of the supports $[F_i]$, we 
always push down to $S$.}}
$$
\chi(F_i)=n_i, \quad [F_i]=\beta_i\in \text{Eff}({m}\alpha) \subset H_2(S,\Z)\,.
$$
Since $F$ is a disjoint union of the $F_i$, 
\beq{sum}
 \sum_{i=1}^kn_i=n\,, \qquad \sum_{i=1}^k\beta_i=\beta\,.
\eeq

Let $\B$ denote the stack of $(n,\beta)$-\emph{marked} expanded 
degenerations{\footnote{To emphasize the marking, we will sometimes denote
 $\B$ by $\B_{n,\beta}$.}}
\cite{LiRelative, LiWu} of $X/S$, with  universal family 
$$\X\to\B\,.$$ 
Over a closed point of $\B$ with stabilizer $(\C^*)^k$, the fiber of $\X$ is the scheme $X[k]$ acted on by $(\C^*)^k$, covering the identity on $X$,
 with markings
$$
(n_0,\beta_0)\in \Z \oplus H_2(X,\Z)\,, \qquad (n_i,\beta_i)\in \Z
\oplus H_2(S,\Z)\,,
$$
satisfying
\beq{2sum}
n_0+\sum_{i=1}^kn_i=n\,, \qquad \beta_0+\sum_{i=1}^k\beta_i=\beta\,. 
\eeq
All the stable pairs parameterized by $P_n^\star(X/S,\beta)$ 
lie over the substack where $(n_0,\beta_0)=(0,0)$. 
By \eqref{2sum}, we see the $(n_i,\beta_i)$ are required to satisfy \eqref{sum}. 

We view 
$
P_n^\star(X/S,\beta)$ as a moduli space of stable pairs on the fibers of $\X\to\B$
 with a map 
$$P_n^\star(X/S,\beta)\to\B$$ taking a pair to the marked support.

\subsection{Rubber} The universal family over the
 divisor of  $\B$ corresponding to a nontrivial degeneration of $X$ over
$S$ with
$$(n_0,\beta_0)=(0,0)$$
is called \emph{rubber}.

Alternatively,  rubber geometry arises from the following
construction.
Consider the stack $\B_{0,\infty}$ of $(n,\beta)$-marked expanded 
degenerations of
\beq{rubb}
S\times\PP^1\ \big/  \ S\times\{0,\infty\}\,.
\eeq
The markings $(n_i,\beta_i)$ on the components are required to satisfy \eqref{sum}. Let
$${\B_\infty}\subset \B_{0,\infty}$$ 
be the open substack where $S\times\{0\}$ has not been bubbled.
The standard $\C^*$-action on $\PP^1$ induces a $\C^*$-action on ${\B_\infty}$ 
and the associated universal family. 
Quotienting by the $\C^*$-action yields the {\em rubber target}: the universal family
\beq{rubdef}
\S\To\B_r
\eeq
over the {\em rubber stack} 
$$\B_r = {\B_\infty}/\C^*\,.$$
 The universal family $\S$ carries the canonical divisors 
$$S_0\times\B_r\,,\ \ S_\infty\times\B_r
\subset \S \,.$$

Gluing $S_0$ of the rubber target to the central fiber $S$ of $X$ embeds 
the rubber stack 
into the stack of $(n,\beta)$-marked expanded degenerations of $X$.
We obtain
the commutative diagram
\beq{durex}
\xymatrix{
\S\,\ \ar@{^(->}[r]\ar[d]<-.5ex> & \X \ar[d] \\
\B_r\ \ar@{^(->}[r] & \B\,.\!\!}
\eeq

Let $R(n, \beta)$ denote the
moduli space of stable pairs on the fibers of \eqref{rubdef}.
Concretely, $R(n,\beta)$ is the moduli space of relative stable pairs on
$$S\times\PP^1\ \big/\ S\times\{0,\infty\}$$
with Euler characteristic $n$, class $\beta$,
and no bubbling over $S\times\{0\}$ -- all modulo the action of $\C^*$. 
The compactness of $R(n, \beta)$ is a consequence of
the $\C^*$-quotient geometry.

We have seen that relative stable pairs on $X/S$ near $0\in \Delta$ are 
in fact supported on the rubber target \eqref{durex}. Pushing forward from rubber to the expanded degenerations of $X/S$ yields a 
morphism 
\beq{iRP}
\iota:R(n,\beta) \rightarrow
P_n^\star(X/S,\beta)
\eeq
which is a closed embedding of Deligne-Mumford stacks {\em and} a bijection on closed points.
The equation which cuts out $R(n,\beta)\subset P_n^\star(X/S,\beta)$ is 
the smoothing parameter of the first bubble.

We will  prove $\iota$ 
is {\em almost} an isomorphism: $\iota$ satisfies the
curvilinear lifting property. 
To prove $\iota$ is an isomorphism, the smoothing parameter of the first bubble
must be shown to vanish in all flat families associated to
the moduli space $P^\star_n(X/S,\beta)$. We leave the isomorphism question open.

\subsection{Deformation theory} \label{defy} 
Following Section \ref{rtov}, let
$$\epsilon: X\rightarrow (\Delta,0)$$
be a $K3$-fibration  for which
$\alpha$ is ${m}$-rigid.
Let $\beta\in \text{Eff}({m}\alpha)$.
We study here the deformation theory of
$$
P=P_n^\star(X/S,\beta)\To\B\,,
$$
the moduli space of stable pairs on the fibers of the right hand side of the diagram \eqref{durex}. 
Identical arguments apply to the left hand side of \eqref{durex}, 
replacing $\B$ by the substack $\B_r$ and $P$ by $R(n,\beta)$ 
to give the deformation theory of stable pairs on the rubber target.

Over $\X\times_\B P$ there is a universal stable pair
$$
\II=\{\O_{\X\times_\B P}\Rt{s}\FF\}\,,
$$
where the complex $\II$ has $\O$ in degree 0. Let $\pi_P$ denote the projection $\X\times_\B P\to P$ and\footnote{Here $R\hom_{\pi_P}
=R\pi_{P*}R\hom$ is the right derived functor of 
$$\hom_{\pi_P}=\pi_{P*}\hom\,.$$}
\beq{edot}
E\udot=(R\hom_{\pi_P}(\II,\II)_0[1])^\vee.
\eeq
By \cite{HT,LiWu}, $P\to\B$ admits a relative perfect obstruction theory
\beq{rpot}
E\udot\To\LL_{P/\B}
\eeq
described as follows. Under the map $\LL_{\X\times_\B P}\to\LL_{(\X\times_\B P)/\B}$, the Atiyah class $\At(\II)$ of $\II$ projects to the \emph{relative} Atiyah class:
\beq{pAt} \xymatrix@R=0pt{
\Ext^1\big(\II,\II\otimes\LL_{\X\times_\B P}\big) \ar[r]&
\Ext^1\Big(\II,\II\otimes\big(\LL_{\X/\B}\oplus\LL_{P/\B}\big)\Big)\,, \\ \hspace{2cm}
\At(\II)\qquad \ar@{|->}[r]& \qquad(\At_{\X/\B}(\II),\At_{P/\B}(\II))\,. \hspace{1cm}}
\eeq
The map \eqref{rpot} is given by the partial Atiyah class $\At_{P/\B}(\II)$ via the following identifications:
\begin{eqnarray}
\Ext^1(\II,\II\otimes\pi_P^*\LL_{P/\B}) &=&
H^1\big(R\hom(\II,\II)\otimes\pi_P^*\LL_{P/\B}\big) \nonumber \\
&=& H^0\big(R\pi_{P*}R\hom(\II,\II)[1]\otimes\LL_{P/\B}\big) \nonumber \\
&\!\!\To\!\!& H^0\big(R\pi_{P*}R\hom(\II,\II)_0[1]\otimes\LL_{P/\B}\big)
\nonumber \\ &=& \Hom\big(E\udot,\LL_{P/\B}\big)\,. \label{ident}
\end{eqnarray}
Defining $\E\udot$ to be the cone on the induced map $E\udot[-1]\to\LL_\B$,
 we obtain a commutative diagram of exact triangles:
\beq{5} \xymatrix@R=18pt{
\E\udot \ar[r]\ar[d] & E\udot \ar[r]\ar[d] & \LL_\B[1] \ar@{=}[d]<-.5ex>
\\ \LL_P \ar[r] & \LL_{P/\B} \ar[r] & \LL_\B[1]\,.\!\!}
\eeq
The Artin stack $\B$ is smooth  with $\LL_\B$ a 2-term complex of bundles supported in degrees $0$ and $1$. The induced map
$$
\Omega_P=h^0(E\udot)\To h^1(\LL_\B)
$$
is onto because the points of $P$ all have finite stabilizers by condition (iv) above. Therefore the long exact sequence of sheaf cohomologies of the top row of \eqref{5} shows $\E\udot$ has cohomology in degrees $-1$ and $0$ only. 
The complex $\E\udot$ 
is perfect because $E\udot$ and $\LL_\B$ are. Since the Deligne-Mumford stack $P$ is projective, $\E\udot$ is quasi-isomorphic to a 2-term complex of locally free sheaves on $P$. Finally, the 5-Lemma applied to the long exact sequences in cohomology of \eqref{5} 
implies
\beq{apot}
\E\udot\To\LL_P
\eeq
is an isomorphism on $h^0$ and onto on $h^{-1}$. Therefore \eqref{apot} is a perfect obstruction theory for $P$.

The virtual dimension of $P_{n}(X/S,\beta)$ is 0. The open and closed 
component
$P_n^\star(X/S,\beta) \subset P_n(X/S,\beta)$
hence carries a virtual class of dimension 0.  
We define
$$P^\star_{n,\beta}(X/S) = \int_{[P_n^\star(X/S,\beta)]^{\vir}} 1\,.$$

\begin{Lemma} We have $P^\star_{n,\beta}(X)= P^\star_{n,\beta}(X/S)$.
\end{Lemma}

\begin{proof}
Consider the degeneration of the total space $X$ to the
normal cone of the special fiber $S=\epsilon^{-1}(0)$.
By the degeneration
formula for stable pairs invariants, $P^\star_{n,\beta}(X)$ is expressed
as a product of integrals over $P^\star_{n_1}(X/S,\beta_1)$ and
$P_{n_2}(S\times\proj^1/S\times\{0\}, \beta_2)$ where
$$n=n_1+n_2\,, \ \ \ \beta=\beta_1+\beta_2\,. $$
Since the virtual class of 
$P_{n_2}(S\times\proj^1/S\times\{0\}, \beta_2)$ vanishes by the
existence of the reduced theory (see Section \ref{rott} below), 
the only surviving term of
the degeneration formula is $n_1=n$ and $\beta_1=\beta$.
\end{proof}

\subsection{Reduced obstruction theory} \label{rott}
Let $R=R(n,\beta)$ be the moduli space of stable pairs on the
rubber target.

The construction of the
obstruction theory in Section \ref{defy}
applies to $\S\times_{\B_r}R$ with the
associated universal complex $\II$ and projection $\pi_R$ to $R$. 
The result is a relative obstruction theory for $R$ given by a similar formula:
\beq{rrpot}
F\udot=(R\hom_{\pi_R}(\II,\II)_0[1])^\vee\To\LL_{R/\B_r}\,,
\eeq
and an absolute obstruction theory
\beq{arpot}
\F\udot=\mathrm{Cone\,}\big(F\udot[-1]\To\LL_{\B_r}\big)\,.
\eeq
The relative obstruction sheaf of \eqref{rrpot} contains $H^{0,2}(S)$ which can be thought of as the topological or Hodge theoretic part of the obstruction to deforming a stable pair.
So long as $S$ remains fixed, $H^{0,2}(S)$ is trivial
and can be removed. After removal, we obtain the {\emph {reduced
obstruction theory}}. By now,
there are many approaches to the reduced theory: see \cite{KT1} for an extensive account and references. We include here a brief treatment.

We fix a holomorphic symplectic form $\sigma$ on $S$.
Let the 2-form $\bar\sigma$ denote the pull-back of $\sigma$ to $\S$. The 
 semiregularity map from the relative obstruction sheaf $\Ob_F=h^1((F\udot)^\vee)$ to $\O_R$ plays a central role:
\begin{align} \label{SRmap}
\nonumber \ext^2_{\pi_R}(\II,\II)_0
\Rt{\cup\At(\II)} \ext^3_{\pi_R}(\II,\II\otimes\LL_{(\S\times_{\B_r}R)/\B_r}) \hspace{2cm} \\ \To\ext^3_{\pi_R}(\II,\II\otimes\Omega_{\S/\B_r})\Rt{\wedge\bar\sigma}
\ext^3_{\pi_R}(\II,\II\otimes\Omega^3_{\S/\B_r}) \nonumber \\
\Rt{\tr}R^3\pi_{R*}(\Omega^3_{\S/\B_r})\To R^3\pi_{R*}(\omega_{\S/\B_r})\,\cong\ \O_R\,. \hspace{-1cm}
\end{align}
In the last line, $\omega_{\S/\B_r}$ is the fiberwise canonical sheaf.
Using the simple structure of the singularities, 
we see $\omega_{\S/\B_r}$ is the sheaf of fiberwise 3-forms with logarithmic poles along the singular divisors in each fiber with opposite
 residues along each branch. The canonical sheaf $\omega_{\S/\B_r}$
inherits a natural map from $\Omega^3_{\S/\B_r}$.

\begin{Proposition} \label{SR}
The semiregularity map \eqref{SRmap} is onto.
\end{Proposition}

\begin{proof}
We work at a closed point $(F,s)$ of $R$ where $F$ is
 a sheaf on  $S\times\PP^1[k]$. 
In \eqref{SRmap}, we replace $\II$ by 
$$I\udot=\{\O\rt{s}F\}$$ and each $\ext_{\pi\_R}$ sheaf by the corresponding $\Ext_{S\times\PP^1[k]}$ group. We will show the result is a surjection
\beq{map2}
\Ext^2(I\udot,I\udot)_0\To\C\,.
\eeq
By the vanishing of the higher trace-free Ext sheaves, base change and the Nakayama Lemma, the surjection \eqref{map2} implies the claimed result.

We use the first order deformation $\kappa\in H^1(T_S)$ of $S$ of \eqref{kappa} and the holomorphic symplectic form $\sigma$. By \eqref{twitor},
 we have
\begin{equation} \label{nonzero}
\int_\beta\kappa\ip\sigma\ne0\,.
\end{equation}
The pull-back of the Kodaira-Spencer class $\kappa$ to $S\times\PP^1[k]$ is
$$
\bar\kappa\in\Ext^1(\LL_{S\times\PP^1[k]},\O_{S\times\PP^1[k]})\,,
$$
the class of the corresponding deformation of $S\times\PP^1[k]$.
We consider
\begin{equation}\label{h555}
\bar\kappa\circ\At(I\udot)\,\in\,\Ext^2(I\udot,I\udot)\,,
\end{equation}
which by \cite{BuF,Ill} is the obstruction to deforming $I\udot$ to first order with the deformation $\kappa$ of $S$. Since $\det(I\udot)$ is trivial, \eqref{h555}
 lies in the subgroup of trace-free Exts. We will show the map \eqref{map2} is nonzero on the element \eqref{h555} of $\Ext^2(I\udot,I\udot)_0$.

The semiregularity map is entirely local to the support
$$
\operatorname{supp}(F)\subset\coprod_i S\times\{p_i\}\,,
$$
where the $p_i$ lie in the interiors of the $\PP^1$ bubbles.
 Using the $(\C^*)^k$ action, we may assume the $p_i$ are all different points of $\C^*=\PP^1\take\{0,\infty\}$.
By moving all of the $p_i$ to a single bubble, we may 
compute the same map on $S\times\PP^1$.

By \cite[Proposition 4.2]{BuF}, 
$$
\tr\big(\bar\kappa\circ\At(I\udot)\circ\At(I\udot)\big)\in H^3(\Omega_{S\times\PP^1})
$$
equals
$
2\bar\kappa\ip \mathrm{ch}_2(I\udot).
$
Therefore, the image of $\bar\kappa\circ\At(I\udot)$ under the map \eqref{map2} is 
\begin{equation}\label{bb5}
2\int_{S\times\PP^1}(\bar\kappa\ip\mathrm{ch}_2(I\udot))\wedge\bar\sigma\ =\ -2\int_{S\times\PP^1}(\bar\kappa\ip\bar\sigma)\wedge\mathrm{ch}_2(I\udot)\,,
\end{equation}
by the homotopy formula 
\begin{equation}\label{htyy}
0\ =\ \bar\kappa\ip(\mathrm{ch}_2(I\udot)\wedge\bar\sigma)\ =\
(\bar\kappa\ip\mathrm{ch}_2(I\udot))\wedge\bar\sigma+
(\bar\kappa\ip\bar\sigma)\wedge\mathrm{ch}_2(I\udot)\,.
\end{equation}
Since $\mathrm{ch}_2(I\udot)$ is Poincar\'e dual to $-\beta$, we conclude \eqref{bb5} equals
$$
2\int_\beta\kappa\ip\sigma\,,
$$
which is nonzero \eqref{nonzero} by the choice of $\kappa$.
\end{proof}

Composing \eqref{SRmap} with the truncation map $(F\udot)^\vee\to h^1((F\udot)^\vee)[-1]$ and dualizing gives a map
\beq{dSR}
\O_R[1]\To F\udot\,.
\eeq

\begin{Proposition} \label{descen} The map \eqref{dSR} lifts uniquely to the absolute obstruction theory $\F\udot$ of \eqref{arpot}.
\end{Proposition}

\begin{proof}
To obtain a lifting, we must show the composition
$$
\O_R[1]\To F\udot\To\LL_{R/\B_r}\To\LL_{\B_r}[1]
$$
is zero.
In fact, the composition of the second and third arrows is already zero on $R$. We will show the vanishing of the dual composition
\beq{zzer}
(\LL_{\B_r})^\vee[-1]\To(\LL_{R/\B_r})^\vee\To(F\udot)^\vee\,.
\eeq

We work with the universal complex $\II$ on $\S\times_{\B_r}R$.
By \eqref{pAt}, we have the diagram (in which we have suppressed some 
pull-back maps):
$$ \xymatrix{
\LL_{\B_r}^\vee[-1] \ar[r]& \LL_{\S/\B_r}^\vee\oplus\LL_{R/\B_r}^\vee
\ar[r]\ar[d]_{\At_{\S/\B_r}(\II)\oplus\!}^{\!\At_{R/\B_r}(\II)}&
\LL_{\S\times_{\B_r}R}^\vee \ar[d]^{\At(\II)} \\
& R\hom(\II,\II)_0[1] \ar@{=}[r] & R\hom(\II,\II)_0[1]\,.\!}
$$
The top row is an exact triangle, so the induced map
$$
\LL_{\B_r}^\vee[-1]\To R\hom(\II,\II)_0[1]
$$
vanishes.\footnote{There is no obstruction to deforming 
 as we move through $R$ over the base $\B_r$: 
there indeed exists a complex $\II$ over all of $\S \times_{\B_r} R$.} 
Therefore the composition
\beq{11}
\xymatrix@C=21pt{
\pi_R^*\LL_{\B_r}^\vee[-1] \ar[r]& \pi_{\S}^*\LL^\vee_{\S/\B_r}
\ar[rr]^(.41){\At_{\S/\B_r}(\II)}&& R\hom(\II,\II)_0[1]}
\eeq
equals minus the composition
\beq{22}
\xymatrix@C=21pt{
\pi_R^*\LL_{\B_r}^\vee[-1] \ar[r]& \pi_R^*\LL^\vee_{R/\B_r}
\ar[rr]^(.4){\At_{R/\B_r}(\II)}&& R\hom(\II,\II)_0[1]\,.}
\eeq
By adjunction the composition \eqref{22} gives the composition
\beq{33}
\xymatrix{
\LL_{\B_r}^\vee[-1] \ar[r]& \LL^\vee_{R/\B_r}
\ar[rr]^(.34){\At_{R/\B_r}(\II)}&& R\pi_{R*}R\hom(\II,\II)_0[1]\,,}
\eeq
which by (\ref{rpot}, \ref{ident}, \ref{rrpot}) is precisely the composition \eqref{zzer} we want to show is zero. So it is sufficient to show  \eqref{11} vanishes.

The first arrow of \eqref{11} is (the pull-back from $\B_r$ to $R$ of) 
the Kodaira-Spencer class of $\S/\B_r$:
 the final arrow in the exact triangle
$$
\LL_{\B_r}\To\LL_{\S}\To\LL_{\S/\B_r}\To\LL_{\B_r}[1]\,.
$$
Away from the singularities $S\times\{0,\infty\}$ in each $S\times\PP^1$-bubble, $\S$ is locally a trivial family over $\B_r$: it is isomorphic to 
$$S\times\C^*\times\B_r$$
 locally\footnote{But not globally due to the nontrivial $\C^*$-action on the $\C$ factor giving a nontrivial line bundle over $\B_r$.} over $\B_r$. 
Therefore this Kodaira-Spencer map vanishes in a neighborhood of the support of the universal sheaf $\FF$.
But the second arrow of \eqref{11} -- the Atiyah class $\At_{\S/\B_r}(\II)$ of $\II$ -- is nonzero only on the support of $\FF$, so the composition is zero. \medskip

Finally, choices of lift are parameterized by $\Hom(\O_R[1],\LL_{\B_r})$. This vanishes because $\LL_{\B_r}$ is concentrated in degrees $0$ and $1$. Therefore the lift is unique.
\end{proof}

The relative and  absolute \emph{reduced} obstruction theories 
are defined respectively by:
\beq{rred}
F\udot_{\red}=\mathrm{Cone}\Big(\O_R[1]\To F\udot\Big), \qquad
\F\udot_{\red}=\mathrm{Cone}\Big(\O_R[1]\To\F\udot\Big)
\eeq
The associated obstruction sheaves
$$
\Ob_F^{\red}=h^1((F\udot_{\red})^\vee)\,, \qquad
\Ob_\F^{\red}=h^1((\F\udot_{\red})^\vee)
$$
are the kernels of the induced semiregularity maps
\beq{aSR}
\Ob_F\To\O_R\,, \qquad
\Ob_\F\To\O_R\,,
\eeq
with the first given by \eqref{SRmap}.

Though not required here, one can show \cite{KT1,Pr} the complexes \eqref{rred} do indeed define perfect obstruction theories for $R$. 
For our purposes of extracting invariants,  the simpler cosection
method of Kiem-Li \cite{KiemLi} is sufficient
to produce the reduced virtual cycle $[R]^{\red}$ as in \cite{MPT}. 

We summarize here the cosection method for the reader. Writing
$$
(\F\udot)^\vee=\{F_0\To F_1\}
$$
as a global two-term complex of locally free sheaves on $R$, Behrend and
Fantechi \cite{BehFan} produce a cone
$$
C\subset F_1 \quad\mathrm{such\ that}\quad [R]^{\vir}=0_{F_1}^!C=
[s(C)c(F_1)]_{\text{virdim}}\,.
$$
Here, $s$ is the Segre class, $c$ is the total Chern class, and we take the piece in degree equal to the virtual dimension 
$$\text{virdim}=\rk F_0-\rk F_1\,.$$ 
Kiem and Li show the  cone $C$ lies cycle theoretically (rather than scheme theoretically) in the kernel $K$ of the composition
$$
F_1\To\Ob\To\O_R.
$$
We define the reduced virtual cycle in the reduced virtual dimension 
$(\text{virdim}+1)$ by
\beq{reddef}
[R]^{\red}=0_K^!C=[s(C)c(K)]_{\text{virdim}+1}\,.
\eeq

The reduced class is much more interesting than
the standard virtual class from the point of view of invariants:
the exact sequence
$$
0\To K\To F_1\To\O_R\To0
$$
implies the vanishing of  the standard virtual class,\footnote{In our particular situation, the vanishing is even more obvious since the standard
virtual dimension 
is $-1$. Since the reduced virtual dimension is $0$, the reduced virtual class is nonetheless nontrivial in general.}
$$
[R]^{\vir}=c_1(\O_R).[R]^{\red}+[s(C)c(K)]_{\text{virdim}}=0\,.
$$
(The second term vanishes because $C$ is a cycle inside $K$, so $s(C)c(K)$ is a sum of cycles in dimension $\ge\rk F_0-\rk K=\text{virdim}+1$.)

The vanishing reflects the fact that we can deform $S$ along $\kappa_S$ \eqref{nonzero}: $\beta$ does not remain  of type $(1,1)$ so there can be no holomorphic curves in class $\beta$.

We define the reduced rubber invariants of $S$ via
integration over the dimension 0 class \eqref{reddef}:
\beq{rrinvt}
R^{\red}_{n,\beta}(S\times\mathcal R)=\int_{[R(n,\beta)]^{\red}}1\,.
\eeq
Here $\mathcal R$ denotes the rubber, the quotient by $\C^*$ of the relative geometry $\PP^1\big/\{0,\infty\}$.

\subsection{Comparison of obstruction theories} \label{cot}
We have constructed three obstruction theories:
\begin{enumerate}
\item[(i)] $\F\udot$ on the rubber moduli space $R$, 
\item[(ii)] $\F\udot_{\red}$ on the rubber moduli space $R$,
\item[(iii)] $\E\udot$ on the moduli space $P^\star$ of stable pairs on $X/S$
over $0\in \Delta$.
\end{enumerate}
Our goal in Sections \ref{cot} - \ref{rigi} is to relate (i), (ii), and (iii).

By pushing stable pairs forward from the rubber to the expanded degenerations of $X/S$, we get a map \eqref{iRP}:
$$
\iota\colon R\To P^\star.
$$
Since $E\udot$ and $F\udot$ were defined by essentially the same formulas 
\eqref{edot} and \eqref{rrpot} respectively, we see
\beq{sameo}
\iota^*E\udot\ \cong\ F\udot.
\eeq
The definitions of $\E\udot$ and $\F\udot$ (\ref{5} and \ref{arpot})
 then yield the following diagram of exact triangles on $R$:
\beq{diog}
\xymatrix@R=16pt{
\ N^* \ar@{=}[r]\ar[d]& N^*\!\!\! \ar[d] \\
\iota^*\LL_\B \ar[r]\ar[d]& \iota^*\E\udot \ar[r]\ar[d]& \iota^*E\udot \ar@{=}[d] \\
\LL_{\B_r} \ar[r]& \F\udot \ar[r]& F\udot\,,\!\!}
\eeq
where $N^*=\LL_{\B_r/\B}[-1]$ is the conormal bundle of the divisor 
$\B_r\subset\B$.
Dualizing the central column and passing to cohomology gives a map
\beq{map1}
N\To\Ob_\F=h^1((\F\udot)^\vee)
\eeq
which describes the obstruction to deforming a pair in the image of $\iota$ 
off the rubber and into the bulk of $X/S$.
Composing with the semiregularity map \eqref{aSR} gives
\beq{map3}
N\To\O_R\,.
\eeq
The rest of Section \ref{cot} will be devoted to proving the following result.

\begin{Proposition} \label{lemm} The maps \emph{(\ref{map1}, \ref{map3})} are injections of sheaves on $R$. Moreover, \eqref{map1} has no zeros.
\end{Proposition}

\subsection*{Connected case}
We first work at a stable pair $(F,s)$ with connected support. 
The sheaf $F$  is therefore supported on $S\times(\PP^1\take\{0,\infty\})$ with no further bubbles in the
rubber.\footnote{If we view the stable pair as lying in $X/S$,
 there is a single bubble and $(F,s)$ is supported in its interior.} We will show that the composition \eqref{map3} is an isomorphism at the point $(F,s)$. By the vanishing of the higher trace-free Exts, base change and the Nakayama Lemma,
 the Proposition will follow in the connected support case.
 
A chart for the stack $\B$ of $(n,\beta)$-marked expanded degenerations of $X/S$ in a neighborhood of the 1-bubble locus is $\C/\C^*$  with universal family $\X\to\B$ given by \cite{LiRelative}
\beq{t}\xymatrix@C=0pt@R=18pt{
\operatorname{Bl}_{S\times\{0\}}(X\times\C) \ar[d] & \curvearrowleft\C^* \\ \C & \curvearrowleft\C^*.\!}
\eeq
Here, the trivial $\C^*$-action on $X$  and  the usual weight $1$ action on 
$\C$ yield a $\C^*$-action on $X\times\C$. The blow up along
the $\C^*$-fixed subvariety $S\times\{0\}$ has a canonically induced
$\C^*$-action.
The exceptional divisor $S\times\PP^1$ inherits a $\C^*$-action.
The central fiber is
$$X[1]=X\cup_S(S\times\PP^1) \,.$$

More explicitly, let $x$ denote the coordinate on $X$ pulled back locally
at $0\in \Delta$ 
from the base of the $K3$-fibration
\begin{equation}\label{cvv5} 
X\to\Delta\,, 
\end{equation}
 and let $t$ denote the coordinate pulled back from the $\C$-base of \eqref{t}. 
By definition,
\beq{blowcone}
\operatorname{Bl}_{S\times\{0\}}(X\times\C_t)\ \ \mathrm{is}\ \ 
\big\{t=\lambda x\big\}\subset X\times\C_t\times\PP^1_\lambda\,,
\eeq
where $x$ has $\C^*$-weight $0$ while $t$ and $\lambda$ have weight $-1$. 
Here, $\lambda$ is the usual coordinate on $\PP^1$ which takes the value $\infty$ on the relative divisor $S\times\{\infty\}$ in the central fiber, and the value $0$ on the proper transform $\overline{X\times\{0\}}$ of the central fiber. Removing these loci, which are disjoint from the support of $(F,s)$, our universal family over $\B$ becomes the quotient by $\C^*$ of
\beq{key}\xymatrix@=30pt{
X\times\C^*_\lambda \ar[r]^(.55){t=\lambda x} & \C_t\,.}
\eeq
The key to Proposition \ref{lemm} is the following observation: 
as we move in the direction $\partial_t$ in the base of \eqref{key}, 
we move away from the central fiber $S\subset X$ in the direction 
$\lambda^{-1}\partial_x$ over the base of the $K3$-fibration \eqref{cvv5}. In other
words, on the central fiber $S\times\C^*_\lambda$, the Kodaira-Spencer class of the family \eqref{key} applied to $\partial_t$ is
\beq{key2}
\lambda^{-1}\kappa\ \in\ \Gamma(\O_{\C^*})\otimes H^1(T_S)\ \cong\ H^1(T_X|_{S\times\C^*})\,.
\eeq
Here, as usual, $\kappa\in H^1(T_S)$ is the Kodaira-Spencer class \eqref{kappa} of \eqref{cvv5} on the central fiber $S$. Since $\lambda\ne \infty$ on the support of $(F,s)$, the $(0,2)$-part of the class $\beta$ of $F$ immediately becomes nonzero along $\lambda^{-1}\partial_x$, just as in the proof of Proposition \ref{SR}. Thus the semiregularity map is nonzero. We now make the argument more precise.

The semiregularity map \eqref{aSR} of Proposition \ref{descen}
was first defined on $F\udot$. After rearranging \eqref{diog},
 we see  the composition \eqref{map3} we require is induced from the composition
\beq{bed}
\iota^*\LL_\B^\vee\To(\iota^*E\udot)^\vee[1]\ \cong\ (F\udot)^\vee[1]\To\O_R\,.
\eeq
The last arrow is \eqref{dSR}. In the proof of Proposition \ref{descen} we showed the vanishing of the composition of the first arrow with
$\LL_{\B_r}^\vee\to \iota^*\LL_\B^\vee$, so the first
arrow factors through the cone $N$ as required. In fact, we even have a splitting
\beq{LBsplit}
\iota^*\LL_\B^\vee\ \cong\ N\oplus\LL_{\B_r}^\vee\,,
\eeq
obtained from expressing $\B$ locally as $\C/\C^*$, with the substack $\B_r$ 
given by $\{0\}/\C^*$. Therefore,
 $$\LL_\B\cong\{\Omega_\C\to\mathfrak g^*\otimes \O_\B\}\,,$$
 where $\mathfrak g$ is the Lie algebra of $\C^*$, and, in the
 standard trivialization, the map takes $dt$ to $t$. 
On applying $\iota^*$ the map therefore vanishes, leaving
$$\iota^*\Omega_\C\oplus (\mathfrak g^*\otimes
\O_{\B_r})[-1]\,=\,N^*\oplus\LL_{\B_r}$$ as claimed.

By the same argument as in (\ref{11}--\ref{33}), the first arrow of \eqref{bed} is (up to a sign) the composition of the following Kodaira-Spencer and Atiyah classes:
$$
\xymatrix@C=20pt{
\iota^*\LL_\B^\vee \ar[r]& R\pi_{R*}(\iota^*\LL^\vee_{\X/\B})[1]
\ar[rr]^(.45){\At_{\X/\B}(\II)}&& R\pi_{R*}R\hom(\II,\II)_0[2]\,.}
$$
Together with the splitting \eqref{LBsplit} and the description \eqref{bed} of our map, we find that at a point $I\udot$ the map \eqref{map3} is the composition
\beq{tub}
N|_{I\udot}\To\Ext^1(\LL_{X[1]},\O_{X[1]})\To\Ext^2(I\udot,I\udot)_0\To\C\,.
\eeq
The first arrow is the Kodaira-Spencer class of the family \eqref{t} on the central fiber $X[1]$. The connected support of our stable pair $(F,s)$ is contained in $S\times\{1\}$, without loss of generality. On restriction to this support, the Kodaira-Spencer class is $\kappa$ (\ref{key}, \ref{key2}).\footnote{If we act by $\lambda\in\C^*$ the relevant statement becomes that for a stable pair supported in $S\times\{\lambda\}$, the value of the Kodaira-Spencer class on $\lambda\partial_t$ is $\lambda.\lambda^{-1}\kappa=\kappa$ \eqref{key2}. The point is that there is no natural trivialisation of $N$, and $\lambda\in\C^*$ takes the trivialization $\partial_t$ to the trivialization $\lambda\partial_t$.}

The second arrow is composition with the Atiyah class of the complex $I\udot$ on $X[1]$. The Atiyah class vanishes on the complement of the support $S\times\{1\}$. We may restrict $I\udot$ to the bubble $S\times\PP^1_\lambda$ and calculate there. The final arrow is the semiregularity map \eqref{SRmap}. So \eqref{tub} simplifies to the composition
\begin{equation}\label{cmpp}
\xymatrix@=30pt{
\C \ar[r]^(.33){\kappa}& H^1(T_{S\times\PP^1}) \ar[r]^(.43){\At(I\udot)}&
\Ext^2_{S\times\PP^1}(I\udot,I\udot)_0 \ar[rr]^(.62){\tr(\,\cdot\,\circ\At(I\udot)\wedge\bar\sigma)}&& \C\,,}
\end{equation}
where again we have trivialized $N$ by the section $\partial_t$.
The composition is therefore
$$
\int_{S\times\PP^1}
\tr(\kappa\circ\At(I\udot)\circ\At(I\udot))\wedge\bar\sigma\,.
$$
By \cite[Proposition 4.2]{BuF}, this is
$$
2\int_{S\times\PP^1}\kappa\ip\mathrm{ch}_2(I\udot)\wedge\bar\sigma=
-2\int_{S\times\PP^1}(\kappa\ip\bar\sigma)\wedge\mathrm{ch}_2(I\udot)=
2\int_{\beta}\kappa\ip\sigma\,,
$$
just as in \eqref{bb5}. Since $\kappa\ip\sigma$ is nonzero on $\beta$ by design \eqref{twitor}, the composition \eqref{cmpp} is nonzero.  Proposition \ref{lemm} is established in the connected  case.

\subsection*{Disconnected case}
To deal with the case of arbitrary support, we write 
a stable pair $(F,s)$ on the rubber target as a 
direct sum of stable pairs with connected supports:
\beq{decom}
(F,s)\,=\,\bigoplus_i(F_i,s_i)\,.
\eeq
By stability we may assume, without loss of generality, that $(F_1,s_1)$ is supported on the interior of the first bubble.

The decomposition \eqref{decom} holds in a neighborhood of $(F,s)$ in the moduli space $R$
(though the $i^{th}$ summand need not have connected
support for pairs not equal to $(F,s)$). The obstruction sheaf $\Ob_\F$ is additive with respect to the decomposition:
 $\ext^2_{\pi_R}(\II,\II)_0$ splits into a corresponding direct sum. 
We will prove \eqref{map3} is an isomorphism on the summand $(F_1,s_1)$. The
isomorphism
 will follow from the connected case after we have set up appropriate notation.
Since the map \eqref{map1} is linear, the result will prove 
\eqref{map1} has no zeros. We will address the injectivity claim
for the map \eqref{map3} in the statement of 
Proposition \ref{lemm} at the end of the proof. \medskip

Suppose $(F,s)$ is supported on $X[k]$. In a neighborhood of $X[k]$, 
a chart for the stack of
$(n,\beta)$-marked expanded degenerations of $X$ is given by 
\beq{chat}
\C^k\big/(\C^*)^k
\eeq
where the group acts diagonally \cite{LiRelative}. We let $t_1,\ldots,t_n$ denote the coordinates on the base $\C^k$.
Let $x$ be the coordinate pulled back from the base $\Delta$ of the $K3$-fibration $X$.

The universal family 
\begin{equation}
\X\to\B\,, \label{unif}
\end{equation}
 restricted to the chart \eqref{chat}, is constructed
 by the following sequence of $(\C^*)^k$-equivariant blow-ups of $X\times\C^k$:
\begin{itemize}
\item Blow-up $X\times\C^k$ along $x=0=t_1$ (the product of the surface $S\subset X$ and the first coordinate hyperplane).
\item Blow-up the result along the \emph{proper transform} of $x=0=t_2$, (the proper transform of $S$ times the second coordinate hyperplane).
\item At the $i$th stage, blow-up the result of the previous step in the proper transform of $x=0=t_i$.
\end{itemize}
After $k$ steps,
 we obtain the universal family \eqref{unif} over the chart \eqref{chat}. 

The fiber of the universal family over the origin of \eqref{chat} is $X[k]$ with marking 
$$(n_0,\beta_0)\in \Z\oplus H_2(X,\Z)$$ on the first component and 
marking  $(n_i,\beta_i)\in \Z\oplus H_2(S,\Z)$ on the $i$th bubble
for  $1\leq i \leq k$. The data 
 satisfy \eqref{2sum}:
$$
n_0+\sum_{i=1}^kn_i=n\,, \qquad \beta_0+\sum_{i=1}^k\beta_i=\beta\,. 
$$
Over a point of $\C^k$ with precisely $j$ of the coordinates $t_i$ vanishing, 
the fiber is $X[j]$. The $j$ vanishing coordinates are in bijective
 ordered correspondence with the $j$ bubbles and the $j$ creases\footnote{The $i$th crease of $X[j]$ is the copy of $S$ at the bottom of the $i$th bubble: the intersection of the $(i-1)$th and $i$th bubbles of $X[j]$.} of $X[j]$. Moving away from this point of the base, a crease smooths if and only if the corresponding coordinate becomes nonzero. If the $i$th and $(i+1)$th vanishing coordinates are $t_a$ and $t_b$, then the marking on the $i$th bubble of $X[j]$ is
$$
\left( \sum_{i=a}^{b-1}n_i\,,\, \sum_{i=a}^{b-1}\beta_i \right)\in 
\Z\oplus H^2(S)\,.
$$
Similarly, if the first vanishing coordinate is $t_c$, then the marking on $X\subset X[j]$ is
\beq{nonbub}
\left(n_0+\sum_{i=1}^{t_c-1}n_i\,,\,\beta_0+\sum_{i=a}^{t_c-1}\beta_i\right)\in \Z\oplus H^2(X)\,.
\eeq
Relative stable pairs -- which cannot lie in $X$ -- all lie over the locus 
$$t_1=0,\ n_0=0, \ \beta_0=0\,,$$ 
where $c=1$ in \eqref{nonbub} and the marking on $X$ vanishes. 
The inclusion of the hyperplane $t_1=0$,
\beq{B0}
\iota\colon\C^{k-1}\big/(\C^*)^k\ =\ \{t_1=0\}\big/(\C^*)^k\ \Into\ \C^k\big/(\C^*)^k
\eeq
describes the inclusion \eqref{durex} of the corresponding chart of the stack $\B_r\subset\B$ over which the rubber target $\S$ lies. 

On the chart \eqref{chat},
$$
\LL_\B\ =\ \big\{\Omega_{\C^k}\To(\mathfrak g^*)^k\otimes \O_\B\big\}\,,
$$
where the map is diag\,$(t_1,\ldots,t_k)$ in the natural trivializations.
Pulling back by \eqref{B0} gives
\beqa
\iota^*\LL_\B &=& \big(N^*\oplus\mathfrak g^*\otimes \O_{\B_r}
[1]\big)\,\oplus
\xymatrix@=55pt{\big\{\Omega_{\C^{k-1}} \ar[r]^(.42){\operatorname{diag}(t_2,\ldots,t_k)}& (\mathfrak g^*)^{k-1}\otimes \O_{\B_r}\big\}} \\ &=&
N^*\ \oplus\ \LL_{\B_r}\,.
\eeqa
Thus $\iota^*\LL_\B^\vee\cong N\oplus\LL^\vee_{\B_r}$ just as in \eqref{LBsplit}. 
The element $\partial_{t_1}$ lies in -- and generates -- the first summand (but, just as before, is not a global trivialisation as it is not $\C^*$-invariant).

We have to work out the composition \eqref{tub} as before, replacing $\partial_t$ by $\partial_{t_1}$.
The stable pair $(F_1,s_1)$ is supported on some $S\times\{1\}$ of the first bubble of $X[k]$ just as before. Restricting to the first bubble, the Kodaira-Spencer class evaluated on the section $\partial_{t_1}$ of $N$ is $\kappa$ just as in the single bubble case: all further blow-ups in the construction of $\X\to\B$ occur at $S\times\{\infty\}$ in the bubble and hence do not affect the interior of the first bubble or its Kodaira-Spencer class. The same calculation then
shows the map \eqref{map3} at the point $(F_1,s_1)$ takes $\partial_{t_1}$ to
\begin{equation}\label{vddx3}
2\int_{\beta_1}\kappa\ip\sigma\,,
\end{equation}
where $\beta_1=[F_1]$. Since \eqref{vddx3} is nonzero by \eqref{twitor}, the map \eqref{map3} is an isomorphism on the first summand $(F_1,s_1)$ as claimed. As explained above, this implies that \eqref{map1} has no zeros. \medskip

Consider now the summand $(F_2,s_2)$. If the support of $(F_2,s_2)$ is in the first bubble, the support lies in $S\times\{\lambda\}$ for some $\lambda\ne1$, and the work we have already done shows that applied to $(F_2,s_2)$ the map \eqref{map3} takes  $\partial_{t_1}$ to
\beq{vary}
2\lambda^{-1}\int_{\beta_2}\kappa\ip\sigma\ne0\,.
\eeq
Here, the nonvanishing is by \eqref{twitor} applied to $\beta_2=[F_2]$.

For summands $(F_i,s_i)$ not in the first bubble we can do a similar calculation, blowing up \eqref{blowcone} once more and using local coordinates again. The result is that the Kodaira-Spencer class in the higher bubbles is 0 (this is effectively the $\lambda\to\infty$ limit of the above calculation).

To prove the  injectivity claim for the  map
\eqref{map3} in the statement of Proposition \ref{lemm}, we 
consider two possibilitites.
\begin{enumerate}
\item[$\bullet$]
 If $(F_1,s_1)$ is the only summand in the first bubble, all others contribute zero to \eqref{map3}, so in total \eqref{map3} is nonzero by \eqref{vddx3}. 

\item[$\bullet$] If there is another summand $(F_2,s_2)$ in the first bubble, then, by linearity, 
the nonzero contribution \eqref{vary} of $(F_2,s_2)$ is added to that of the other summands, and can be varied by perturbing its support 
$$\lambda\in\C^*\subset\PP^1\, .$$ Therefore, even though the map \eqref{map3} might be zero at $(F,s)$, the map \eqref{map3} is nonzero at a nearby perturbation.
Since the perturbation by moving the support  $\lambda$ is along an
\'etale trivial factor in the moduli space $R$, 
%
the map $\eqref{map3}$ must be injective as a morphism of sheaves. 
\end{enumerate}
The proof of Proposition \ref{lemm} is complete. \qed

\subsection{Curvilinear lifting}\label{izz}


Proposition \ref{lemm} does not imply the moduli spaces $R$ and $P^\star$ are isomorphic.
Our analysis of $\widetilde{R}_{n,\beta}$ is crucially dependent upon 
a weaker curvilinear lifting relationship between $R$ and $P^\star$ which does follow
from Proposition \ref{lemm}.

\begin{Lemma} \label{P=RR}
The map $\iota\colon R\to P^\star$ of \eqref{iRP} induces an isomorphism 
$\iota^*\Omega_{P^\star}\cong\Omega_R$ of cotangent sheaves.
\end{Lemma}

\begin{proof}
The obstruction theories $\E\udot,\ \F\udot$ are related by the exact triangle
$$
N^*\To \iota^*\E\udot\To\F\udot
$$
of \eqref{diog}, giving the exact sequence
$$
h^{-1}(\F\udot)\To N^*\To h^0(\iota^*\E\udot)\To h^0(\F\udot)\To0\,.
$$
Since $\E\udot$ vanishes in strictly positive degrees 
and $\iota^*$ is right exact, 
$h^0(\iota^*\E\udot)=\iota^*h^0(\E\udot)$. Therefore, we obtain
\beq{eseq}
h^{-1}(\F\udot)\To N^*\To \iota^*\Omega_{P^\star}\To\Omega_R\To0\,.
\eeq
By Proposition \ref{lemm}, the first map is surjective.
\end{proof}

\begin{Corollary} \label{cor5}
Suppose $A$ is a subscheme of $B$ with ideal $J$ satisfying
\beq{curviy}
d\colon J\to\Omega_B|_A\ \ \mathrm{injective}.
\eeq
(In particular $J^2=0$.) Then any extension $\widetilde f\colon B\to P^\star$ of a map $f\colon A\to R$ factors through $R$.
\end{Corollary}

\begin{proof}
Let $I$ denote the ideal of $R$ inside $P$. To show the factorization of $\tilde{f}$ through $R$,
we must show the image of
$$\tilde{f}^* I \rightarrow J$$
vanishes. Since  $J^2=0$, the above image can be evaluated after restriction to A,
\begin{equation}\label{vqq2}
\tilde{f}^*I|_A =  f^*I \rightarrow J\,.
\end{equation}

By Lemma \ref{P=RR}, the first map $d$ vanishes in the
 exact sequence of K\"ahler differentials
$$
I\Rt{d}\iota^*\Omega_{P^\star}\To\Omega_R\To0\,.
$$
Pulling-back via $f$ gives the top row of the following commutative diagram with exact rows.
$$\xymatrix@=18pt{
& f^*I \dto^{\widetilde f^*}\rto^(.4){0}& f^*\iota^*\Omega_{P^\star} \dto^{\widetilde f^*}\rto^{\sim}& f^*\Omega_R \dto^{f^*}\rto& 0 \\
0 \rto& J \rto^(.4)d& \Omega_B|_A \rto& \Omega_A \rto& 0\,.\!\!}
$$
Here, the central map uses the isomorphism $f^*\iota^*\Omega_{P^\star}\cong(\widetilde f^*\Omega_{P^\star})|_A$. 
As a result, the first vertical arrow (given by \eqref{vqq2}) is zero. Hence,
$\widetilde f$ has image in $R$.
\end{proof}

The basic relationship between $R$ and $P^\star$ which we need is the 
 {\em curvilinear lifting property} proven in the following Corollary.

\begin{Corollary} \label{jeqq}
Every map $\Spec\C[x]/(x^k)\to P^\star$ factors through $R$.
\end{Corollary}

\begin{proof} Since $R\subset P^\star$ is a bijection of sets, 
we have the result for
$k=1$. Since
$$A=\Spec\frac{\C[x]}{(x^k)}\subset\Spec\frac{\C[x]}{(x^{k+1})}=B$$ satisfies \eqref{curviy}, the 
result for higher $k$ follows by Corollary \ref{cor5} and induction.
\end{proof}

 We summarize the above results in the following
Proposition.

\begin{Proposition} \label{P=R}
The map $\iota\colon R\to P^\star$ of \eqref{iRP} is a
closed embedding of Deligne-Mumford stacks which satisfies the
curvilinear lifting property. $\hfill\square$
\end{Proposition}




 The complexes $\iota^*\E\udot$,
$\F\udot$, and $\F_{\red}\udot$ on $R$
are related by the exact triangles
\beq{term4}
\xymatrix@R=14pt{
& (\F\udot_{\red})^\vee \ar@{=}[r]\ar[d] & (\F\udot_{\red})^\vee \ar[d] \\
N[-1] \ar[r]\ar@{=}[d]& (\F\udot)^\vee \ar[r]\ar[d]& (\iota^*\E\udot)^\vee \ar[d] \\
N[-1] \ar[r] & \O_R[-1] \ar[r]& \O_D[-1]\,.\!\!}
\eeq
Here, $D$ is the Cartier divisor on which the injection $N\to\O_R$ \eqref{map3} vanishes.

In particular, in $K$-theory the classes of $(\iota^*\E\udot)^\vee$ and $(\F\udot_{\red})^\vee$ differ by $\mathcal O_D[-1]$.
The $K$-theory classes determine the corresponding virtual cycles via the formula of \cite{Fu,Si}. For the virtual class $[P^\star]^{\vir}$ associated to
$\E\udot$, the formula is
\beq{virtform}
[P^\star]^{\vir}=\big[s\big((\E\udot)^\vee\big) \cap c_F(P^\star)\big]_{\text{virdim}}\,,
\eeq
where $c_F(P^\star)$ is the Fulton total Chern class \cite[4.2.6.(a)]{Fu} of the 
scheme $P^\star$, $s$ denotes the total Segre class, and the subscript denotes the term in degree 
specified by the virtual dimension (equal to 0 here).
The homology class $c_F(P^\star)$ is \emph{not} of pure degree. 
The expression \eqref{virtform} is a sum of different degree parts of the cohomology class $s\big((\E\udot)^\vee\big)$ capped with the different degree
parts of $c_F(P^\star)$
to give the virtual class.

Since the reduced scheme structure of $R$ and $P^\star$ is
the same, we may view $c_F(P^\star)$ and $[P^\star]^{\vir}$ 
as cycles on $R$.
The curvilinear lifting property of Proposition \ref{P=R} implies
a basic relation between the Fulton Chern classes of $R$ and $P^\star$ 
explained in Appendix \ref{ap22},
\begin{equation}
\label{credd} c_F(R) = c_F(P^\star)\  \in A_*(R)\,.
\end{equation}
Formulas \eqref{virtform} and \eqref{credd} allow us to
study $[P^\star]^{\vir}$ via the geometry of $R$.

Since $s(\O_D)=1-D$, the rightmost column of \eqref{term4} yields the identity
$$
s\big((\iota^*\E\udot)^\vee\big)\ =\ s\big((\F_{\red}\udot)^\vee\big)+
D.\,s\big((\iota^*\E\udot)^\vee\big).
$$
After substituting in \eqref{virtform}, we obtain
\begin{eqnarray}
\nonumber
\ \ { } \ { }\ \hspace{10pt}[P^\star]^{\vir} \hspace{-5pt}
&=&\hspace{-5pt}
\big[s\big((\F_{\red}\udot)^\vee\big) \cap c_F(P^\star)\big]_{0}
+
\big[s\big((\iota^*\E\udot|\_D)^\vee\big)
\cap c_F(P^\star)|\_D
\big]_0 \\
&=&\hspace{-5pt} \label{compari}
\big[s\big((\F_{\red}\udot)^\vee\big) \cap c_F(R)\big]_{0}
+
\big[s\big((\iota^*\E\udot|\_D)^\vee\big)
\cap c_F(R)|\_D
\big]_0\\
&=&\hspace{-5pt}[R]^{\red}+      \nonumber
\big[s\big((\iota^*\E\udot|\_D)^\vee\big)
\cap c_F(R)|\_D
\big]_0\, .
\end{eqnarray}
We may replace $D$ in the rightmost term
by any other Cartier divisor
in the same linear equivalence class, since the replacement
 leaves the $K$-theory class $[\O_D]$ unchanged. 
We will work on a cover of the rubber moduli space $R$ on which $D$ becomes linearly equivalent to a rather more tractable divisor.

\subsection{Rigidification}\label{rigi}
Let $R=R(n,\beta)$ be the moduli space of stable pairs on the rubber target.
Let
\begin{equation}\label{UtoR}
\pi:U(n,\beta)=\S\times _{\B_r} R \rightarrow R
\end{equation}
be the universal target over the rubber moduli space $R$.
Let 
\begin{equation}\label{vttpp2}
W(n,\beta)\subset U(n,\beta)
\end{equation} denote the
open set on which the morphism $\pi$ is smooth.

We  view $U$ as a moduli of pairs $(r,p)$ where $r\in R$ and $p$ is
a point in the rubber target associated to $r$. For pairs $(r,p)\in W$,
the point $p$ is not
permitted to lie on {\em any} creases. Hence, the restriction
$$\pi: W \rightarrow R$$ is a smooth morphism.
The rubber target admits a natural map, 
$$\rho: W \rightarrow S\,,$$
to the underlying $K3$ surface.

Viewing $[R]^{\red}$ and  $[P^{\star}]^{\vir}$ as cycle classes on $R$, 
we define classes
$[W]^{\red}$ and $[W]^{\vir}$ on $W$ by flat pull-back:
$$[W]^{\red}= \pi^*[R]^{\red}\,,\ \ \ \ [W]^{\vir}= \pi^*[P^\star]^{\vir}.
$$
By the definitions of the cones and the Fulton class,
\begin{equation}\label{cxxr4}
[W]^{\red} = \big[s\big((\pi^*\F\udot_{\red})^\vee\big)s\big(\mathcal{T}_\pi\big)
 \cap c_F(W)\big]_{3}\, ,
\end{equation}
where $\mathcal{T}_\pi$ on $W$ is the relative tangent bundle of $\pi$.
Similarly,
\begin{equation}\label{cxxr45}
[W]^{\vir} = \big[s\big((\pi^*\iota^*\E\udot)^\vee\big)s\big(\mathcal{T}_\pi\big)
 \cap c_F(W)\big]_{3}\, .
\end{equation}



Integrals over $R$ may be moved to $W$ by the following procedure.
To a class 
$\delta\in H^2(S,\mathbb{Q})$, we associate a primary insertion
$$
T_0(\delta)=\text{ch}_2(\mathbb{F}) \cup \rho^*(\delta)\in H^6\big(U,\mathbb{Q}
\big)\,,
$$
where $\mathbb{F}$ is the universal sheaf.
The key identity is the divisor formula obtained
by integrating down the fibers of \eqref{UtoR}: 
\beq{push}
\pi_*\big(T_0(\delta)\big)=\langle \delta ,\beta\rangle =\int_\beta \delta
\in H^0(R,\mathbb{Q})\,.
\eeq
The push-forward $\pi_*$ is well-defined since $\ch_2(\mathbb{F})$ and $T_0(\delta)$ are supported on $\text{Supp}(\mathbb{F})$ which is
projective over $R$ via $\pi$.

The derivation in Section \ref{izz} of \eqref{compari} can be pulled-back
via $\pi$ to $W$ to yield:
\beq{compari2}
[W]^{\vir}\ =\ [W]^{\red}+\big[s\big((\pi^*\iota^*\E\udot|\_D)^\vee\big)
s\big(\mathcal{T}_\pi\big)
\cap c_F(W)|\_D
\big]_3\,.
\eeq
As before, $D$ is any divisor  representing the first Chern class of the pull-back of $N^*$, the conormal bundle
of the divisor $\B_r\subset\B$. 
By integration against \eqref{compari2}, formula \eqref{push} 
yields
\begin{multline}\label{formula}
P^\star_{n,\beta}(X)-R^{\red}_{n,\beta}(S\times\mathcal R) = \\
\frac1{\langle \delta,\beta\rangle}
\int_{D} 
\Big(s\big((\pi^*\iota^*\E\udot|\_D)^\vee\big) 
s\big(\mathcal{T}_\pi\big)
\cap
c_F(W)|\_D \Big)
\cdot
T_0(\delta)\,.
\end{multline}
We describe next a geometric representative for $D$.

Attaching the infinity section $S_\infty$ of the rubber over $R(n_1,\beta_1)$
to the zero section $S_0$ of rubber over $W$ defines a divisor
\begin{equation}\label{nn1b1}
D_{n_1,\beta_1}=R(n_1,\beta_1)\times W(n_2,\beta_2)\ \subset\ W(n,\beta)\,,
\end{equation}
whenever $(n_1,\beta_1)+(n_2,\beta_2)=(n,\beta)$.
The following result is a form of {\em topological recursion}.

\begin{Lemma}\label{gbb3}
The line bundle $\pi^*N^*$ has a section with zeros given by the divisor
$$
D=\sum_{n_1,\beta_1}D_{n_1,\beta_1}\,.
$$
The sum is over the finitely many $(n_1,\beta_1)\in \Z\oplus H_2(S,\Z)$ for which $D_{n_1,\beta_1}$ is nonempty.
\end{Lemma}

\begin{proof}
The universal family $\mathcal X\to\B$ has smooth total space. Moving normal to $\B_r$ smooths the first crease in the expanded degeneration $\X$: the crease
 where $X$ joins the rubber $\S$ across $S\subset X$ and the 0-section $S_0$ of the rubber. Therefore the normal bundle $N$, pulled back to $\X$ and restricted to the first crease, is isomorphic to
$$
N_{S\subset X}\otimes N_{S_0\subset\S}\,.
$$
Fixing once and for all a trivialization of $N_{S\subset X}$, we find that $$
N\cong\psi_0^*
$$
is isomorphic to the tangent line to $\PP^1$ on the zero section.\footnote{More precisely, fix any point $s\in S$. Then the corresponding point of the zero section $S_0$ of $\S$ defines a section $s_0$ of $\S\to\B_r$. The first bubble is $S_0\times\PP^1$, and the conormal bundle to $S_0$ is the restriction of $T^*_{\PP^1}\subset T^*_{S_0\times\PP^1}$. Pulling back by $s_0$ gives the cotangent line $\psi_0$.}

Now pull-back to $W(n,\beta)$ via \eqref{UtoR}. By forgetting $S$ (but remembering the 
$\Z \oplus H_2(S,\Z)$ marking), $W(n,\beta)$ maps to the stack 
$\B^p_{n,\beta}$ of $(n,\beta)$-marked 
expanded degenerations
 of $\PP^1/\{0,\infty\}$,
\begin{equation} \label{jss2}
W(n,\beta)\rightarrow \B^p_{n,\beta}\,.
\end{equation}
The moduli space $W(n,\beta)$ parameterizes pairs $(r,p)$.
The map \eqref{jss2} is defined by using $p$ to select the rigid
component and to determine $1\in \PP^1$. 
The cotangent line at the relative divisor $0$ defines a line bundle on  
$\B^p_{n,\beta}$ which we also call $\psi_0$. On $W(n,\beta)$, $\psi_0$ pulls-back to  $$\pi^*\psi_0\cong\pi^*N^*$$ above.

Pick a nonzero element of $T^*_{\PP^1}|\_0$, pull-back to any expanded degeneration, and restrict to the new $0$-section. 
We have constructed  a section of $\psi_0$ over $\B_{n,\beta}^p$ which vanishes precisely on the divisor where $0$ has been bubbled. The latter
 divisor is the sum of the divisors $D_{n_1,\beta_1}$ as required.
\end{proof}

After combining Lemma \ref{gbb3}  with \eqref{formula}, we obtain the formula:
\begin{multline}
 P^\star_{n,\beta}(X)-R^{\red}_{n,\beta}(S\times\mathcal R) = \\
\label{formula2}
\frac1{\langle \delta,\beta\rangle}\sum_{n_1,\beta_1}
\int_{D_{n_1,\beta_1}}
\Big(s\big((\pi^*\iota^*\E\udot)^\vee\big) 
s\big(\mathcal{T}_\pi\big)|_{D_{n_1,\beta_1}}
\cap
c_F(W)|\_{D_{n_1,\beta_1}} \Big)
\cdot
T_0(\delta)\,.
\end{multline}
To proceed, we must to compute the restriction of $\pi^*\iota^*\E\udot$ to the divisor
$D_{n_1,\beta_1}$.

The \emph{relative} obstruction theory is given by the same formula \eqref{rpot} on  the moduli spaces $R$ and $P^\star$.
On $R$, the relative obstruction theory was denoted $F\udot$ \eqref{rrpot}. 
Since the relative obstruction theory is
 additive over the connected components of a stable pair,
\beq{roof0}
\pi^*\iota^*E\udot\big|\_{D_{n_1,\beta_1}}\,\cong\ 
\iota^*E_{R(n_1,\beta_1)}
\udot\oplus 
\pi^*F
\udot_{W(n_2,\beta_2)}\,,
\eeq
where we have split $D_{n_1,\beta_1}$ as $R(n_1,\beta_1)\times W(n_2,\beta_2)$ using \eqref{nn1b1}.

As in (\ref{5}, \ref{arpot}) the relationship of the relative to the
absolute obstruction theories $\E\udot_W$ and $\F\udot$ is through the usual exact triangles:
\beq{roof1}
\iota^*\E \udot\To \iota^*E \udot\To\iota^*\LL_{\B_{n,\beta}}[1]
\eeq
on the moduli space $R$ and 
\begin{multline} \label{klakla}
\iota^*\E_{R(n_1,\beta_1)} \udot\oplus \pi^*\F\udot_{W(n_2,\beta_2)}\To 
\iota^*E_{R(n_1,\beta_1)}\udot\oplus \pi^*F_{W(n_2,\beta_2)}\udot
\\ \To 
\iota^*\LL_{\B_{n_1,\beta_1}}[1]\oplus\pi^*\LL_{\B_{r}}[1]
\end{multline}
on $D_{n_1,\beta_1}=R(n_1,\beta_1)\times W(n_2,\beta_2)$. 
In $K$-theory, 
the isomorphism \eqref{roof0} and the
 exact sequences \eqref{roof1}-\eqref{klakla} yield
\begin{multline}\label{roof9}
\big[\pi^*\iota^*\E\udot\big|\_{D_{n_1,\beta_1}}\big]\,= 
[\iota^*\E\udot_{R(n_1,\beta_1)}]+
[\pi^*\F\udot_{W(n_2,\beta_2)}] \\ -[\iota^*\LL_{\B_{n_1,\beta_1}}]
-[\pi^*\LL_{\B_{r}}]+
[\iota^*\LL_{\B_{n,\beta}}]\,.
\end{multline}

Since $D_{n_1,\beta_1}$ is pulled-back{\footnote{More
precisely, $D_{n_1,\beta_1}\subset W$ is
an open and closed component of the pull-back.}}
 from $\B_{n,\beta}$, the
standard divisor conormal bundle sequence yields
$$[\LL_{\B_{n_1,\beta_1}}]+[\LL_{\B_{r}}]-
[\LL_{\B_{n,\beta}}] = -[\O_{D_{n_1,\beta_1}}(-D_{n_1,\beta_1})]\,. $$
Hence, we obtain
\begin{multline} \label{terr1}
s\Big(\big(\pi^*\iota^*\E\udot\big|\_{D_{\beta_1,n_1}}\big)^\vee\Big)
c\big(\O_{D_{n_1,\beta_1}}(D_{n_1,\beta_1})\big) \\
=s\big((\iota^*\E_{R(n_1,\beta_1)}\udot)^\vee\big) s\big((\pi^*\F\udot_{W(n_2,\beta_2)})^\vee\big)
\end{multline}

By the basic properties of cones and the embedding
$D_{n_1,\beta_1}\subset W(n,\beta)$
discussed in  Appendix \ref{ap33}, we have 
\begin{equation} \label{ggg}
c_F(W(n,\beta))\big|\_{D_{n_1,\beta_1}}\,
\ =\ c_F(D_{n_1,\beta_1})c(\O_{D_{n_1,\beta_1}}(D_{n_1,\beta_1}))\,.
\end{equation}

We replace $c_F(W(n,\beta))\big|\_{D_{n_1,\beta_1}}$ by \eqref{ggg}
 in \eqref{formula2}. After using  \eqref{terr1}
to cancel the $c(\O_{D_{n_1,\beta_1}}(D_{n_1,\beta_1}))$ term, we 
obtain{\footnote{The integration over $W(n_2,\beta_2)$ 
is well-defined since
the insertions $\tau_0(\delta)$ yields a complete cycle as
before.}}
\begin{multline*}
\frac1{\langle \delta,\beta\rangle}\sum_{n_1,\beta_1}
\int_{R(n_1,\beta_1)} s((\iota^*\E\udot)^\vee)\cap
c_F(R(n_1,\beta_1)) \\
\times\int_{W(n_2,\beta_2)}
\Big( s((\pi^*\F\udot)^\vee) 
s\big(\mathcal{T}_\pi\big)
\cap
c_F(W(n_2,\beta_2)) \Big)\cdot
T_0(\delta)\,.
\end{multline*}
Since the two obstruction theories $\F\udot$ and $\F\udot_{\red}$
on $W(n_2,\beta_2)$ differ only by the trivial line bundle,
 $$s((\F\udot)^\vee)=
s((\F\udot_{\red})^\vee)\,.$$ 
By another application of the rigidification formula \eqref{push}, we obtain
$$
P^\star_{n,\beta}(X)-R^{\red}_{n,\beta}(S\times\mathcal R)\ =
\ \frac1{\langle \delta,\beta\rangle} \sum_{n_1,\beta_1}
P^\star_{n_1,\beta_1}(X)\,  \cdot\,  \langle \delta,\beta_2\rangle\, 
R^{\red}_{n_2,\beta_2}(S\times\mathcal R)  \,.
$$

\subsection{Logarithm}
Let $\alpha\in \text{Pic}(S)$ is a  primitive class
which is positive with respect to a polarization, and
let $\beta\in \text{Eff}({m}\alpha)$. 
The only effective decompositions of $\beta$ are
$$\beta=  \beta_1 + \beta_2\,, \ \ \ \ \beta_i\in \text{Eff}({m}\alpha)\,.$$
We formulate the last equation of Section \ref{rigi} as the following result.
\begin{Theorem} \label{t66t}
For the family $\epsilon: X \rightarrow (\Delta,0)$ satisfying
conditions (i,ii,$\star$) of Section \ref{rtov} for ${m}\alpha$, we have
\begin{multline*}
P^\star_{n,\beta}(X)-R^{\red}_{n,\beta}(S\times\mathcal R)\ = \\
\ \frac1{\langle \delta,\beta\rangle} \sum_{n_1+n_2=n} \sum_{\beta_1+\beta_2 =\beta}
P^\star_{n_1,\beta_1}(X)\,  \cdot\,  \langle \delta,\beta_2\rangle\, 
R^{\red}_{n_2,\beta_2}(S\times\mathcal R)  \ 
\end{multline*}
for every $\beta\in \text{\em Eff}({m}\alpha)$ and 
$\delta\in H^2(S,\mathbb{Z})$ satisfying $\langle \delta,\beta\rangle\neq 0$.
\end{Theorem}

The basic relationship between $P^\star_{n,\beta}(X)$ and $R^{\red}_{n,\beta}(S\times\mathcal R)$
is an immediate Corollary of Theorem \ref{t66t}. Let
$$P^\star_{\beta}(X)= \sum_{n\in \mathbb{Z}} P^\star_{n,\beta}(X)\, q^n\,, \ \ \ \ 
R^{\red}_{\beta}(S\times\mathcal R)= \sum_{n\in \mathbb{Z}} R^{\red}_{n,\beta}(S\times\mathcal R)\, q^n\,. $$

\begin{Corollary} For $\beta \in \text{\em Eff}({m}\alpha)$, \label{g4545}
\begin{equation*}
P^\star_{\beta}(X) = \text{\em Coeff}_{v^{\beta}} \left[
\exp\left(  \sum_{\widehat{\beta}\in \text{\em Eff}({m}\alpha)} 
v^{\widehat{\beta}} R^{\red}_{\widehat{\beta}}(S\times\mathcal R)
\right)\right]\,.
\end{equation*}
\end{Corollary}

\begin{proof} Since $\delta \in H^2(S,\mathbb{Z})$ is arbitrary,
Theorem \ref{t66t} uniquely determines $P^\star_\beta(X)$ in terms of
$R^{\red}_{\beta_i}(S\times\mathcal R)$ for $\beta_i\in \text{Eff}({m}\alpha)$.
We see Corollary \ref{g4545} implies
exactly the recursion of Theorem \ref{t66t} by differentiating the
exponential.
To write the
differentiation explicitly, let 
$$v_1,\ldots, v_{22} \in H^2(S,\mathbb{Z})$$
be a basis. For $\beta = \sum_{i=1}^{22}{b_i v_i}$, we write
$$v^\beta = \prod_{i=1}^{22} v_i^{b_i} \,.$$
Let $\langle \delta,v_i \rangle = c_i$.
Then, differentiation of the equation of Corollary \ref{g4545} by 
$$\sum_{i=1}^{22} c_i v_i\frac{\partial}{\partial v_i}$$
yields the recursion of Theorem \ref{t66t}.
\end{proof}

\subsection{Definition of $\widetilde{R}_{n,\beta}$} \label{ped}
We  now return to the definition of stable pairs invariants
for $K3$ surface given in Section \ref{c123}.

Let $\alpha\in \text{Pic}(S)$ be a primitive class
which is positive with respect to a polarization.
Let $$\epsilon: X \rightarrow (\Delta,0)$$ satisfying the
conditions of Section \ref{rtov}: (i), (ii), and ($\star$)  for ${m}\alpha$.
We have
$$
\widetilde{R}_{m\alpha}(S)  = 
\text{Coeff}_{v^{m\alpha}} \left[
\log\left(1 +  \sum_{\beta\in \text{Eff}({m}\alpha)} 
v^{\beta} P^\star_\beta(X)\right)\right]\,.
$$
Let $\widetilde{R}_{n,m\alpha}(S)$ be
the associated $q$ coefficients:
$$ 
\widetilde{R}_{m\alpha}(S)= \sum_{n\in \mathbb{Z}} \widetilde{R}_{n,m\alpha}(S)\, q^n\,. $$
By Corollary \ref{g4545} of Theorem \ref{t66t}, we can take the logarithm.

\begin{Proposition}\label{kedd9}
We have
$\widetilde{R}_{m\alpha}(S)= R^{\red}_{m\alpha}(S\times\mathcal R)$.
\end{Proposition}

\begin{proof}
By Corollary \ref{g4545} and the above definition,
\begin{eqnarray*}
\widetilde{R}_{m\alpha}(S) & = & 
\text{Coeff}_{v^{m\alpha}} \left[
\log
\exp\left(  \sum_{\widehat{\beta}\in \text{Eff}({m}\alpha)} 
v^{\widehat{\beta}} R^{\red}_{\widehat{\beta}}(S\times\mathcal R)
\right)
\right] \\
& = &\text{Coeff}_{v^{m\alpha}} \left[
 \sum_{\widehat{\beta}\in \text{Eff}({m}\alpha)} 
v^{\widehat{\beta}} R^{\red}_{\widehat{\beta}}(S\times\mathcal R)\right] \,.
\end{eqnarray*}
Hence, if $m\alpha\in \text{Pic}(S)$ is effective,
$\widetilde{R}_{m\alpha}(S)=R^{\red}_{m\alpha}(S\times\mathcal R)$.
If $m\alpha\in \text{Pic}(S)$ is not effective, both 
$\widetilde{R}_{m\alpha}(S)$ and $R^{\red}_{m\alpha}(S\times\mathcal R)$
vanish.
\end{proof}

Proposition \ref{kedd9} is the main point of Section \ref{relt}.
The complete geometric interpretation of the logarithm 
as integration over a moduli of stable pairs is crucial for our
proof of the KKV conjecture.

\subsection{Dependence} \label{depped}
Let $\alpha\in \text{Pic}(S)$ be a primitive class
which is positive with respect to a polarization.
As a consequence of Proposition \ref{kedd9}, we obtain
a dependence result.

\begin{Proposition}  
$\widetilde{R}_{m\alpha}(S)$ 
depends {\em only}
upon $m$ and 
$$\langle \alpha, \alpha \rangle=2h-2$$
and {\em not} upon $S$ or the family $\epsilon$. 
\end{Proposition}

From the definition of
$\widetilde{R}_{m\alpha}(S)$, the dependence statement is not immediate
(since the argument of the logarithm depends upon invariants
of effective summands of $m\alpha$ which may or may not persist in
deformations of $S$ for which $\alpha$ stays algebraic).
However, $R^{\red}_{m\alpha}(S\times\mathcal R)$ depends {only} upon $m$ and the deformation
class of the pair $(S,\alpha)$ -- and the latter depends only upon
$\langle \alpha, \alpha\rangle$.

We finally have a well-defined
analogue in stable pairs of the Gromov-Witten invariants $R_{n,m\alpha}$.
We drop $S$ from the notation, and
to make the dependence clear, we define
\begin{equation}\label{gqgq}
\widetilde{R}_{m,m^2(h-1)+1}=\widetilde{R}_{m\alpha}
\end{equation}
where right side is obtained from {\em any} $K3$ surface and
family $\epsilon$ satisfying conditions (i), (ii), and ($\star$) for
$m\alpha$. 
Here, $m\alpha$ is of divisibility $m$ and has norm square 
$$\langle m\alpha,m\alpha \rangle=m^2(2h-2)=2(m^2(h-1)+1)-2\,.$$
In case $m=1$, we will use the abbreviation
\begin{equation}\label{psszz}
\widetilde{R}_h = \widetilde{R}_{1,h}\,.
\end{equation}

\section{Multiple covers}\label{mulcov}
\subsection{Overview}
By Proposition \ref{kedd9}, the stable pairs invariants
$\widetilde{R}_{m\alpha}(S)$ equal the reduced rubber invariants
$R^{\red}_{m\alpha}(S\times\mathcal R)$. Our goal here is to express the latter
in terms of the reduced residue invariants $\langle 1 \rangle^{\red}_{Y,m\alpha}$
of
$$Y= S \times \C$$
studied in Sections \ref{k3tcl} and \ref{van}.
The explicit calculations of Section \ref{lc22} then determine 
$\widetilde{R}_{m\alpha}(S)$ and can be interpreted in terms of
multiple cover formulas.

\subsection{Rigidification} \label{rigi2}
Let $S$ be a $K3$ surface equipped with an ample polarization $L$.  
Let $\alpha\in \text{Pic}(S)$ be a primitive, positive{\footnote{Positivity,
$\langle L, \alpha \rangle >0$,
is with respect to the polarization $L$.}} class 
with norm square
 $$\langle \alpha,\alpha \rangle =2h-2\,. $$
Let $m>0$ be a integer.
The invariants
$$R^{\red}_{m\alpha}(S\times\mathcal R)= \sum_{n\in \mathbb{Z}} R^{\red}_{n,m\alpha}(S\times\mathcal R) \, q^n$$
have been defined in Section \ref{defy} by integration over the
rubber moduli spaces $R(n,m\alpha)$.

Following the notation of Section \ref{rigi}, let
\begin{equation*}
\pi:U(n,m\alpha)=\S\times _{\B_r} R(n,m\alpha) \rightarrow R(n,m\alpha)
\end{equation*}
be the universal target with virtual class pulled-back from $R(n,m\alpha)$.
We also consider here
$$V(n,m\alpha/ 0,\infty)= P_n( S\times \proj^1/  S_0 \cup S_\infty, m\alpha)\,, $$
the moduli space of stable pairs on $S \times \proj^1$
relative to the fibers over $0,\infty \in \proj^1$.
There is a standard rigidification 
map{\footnote{We identify the point in the universal target $\S$ with the corresponding distinguished point of $S\times \proj^1/  S_0 \cup S_\infty$ lying over $1\in\PP^1$. The resulting identification of the two universal targets (for the rubber and for $S\times \proj^1/  S_0 \cup S_\infty$) is used to transport stable pairs from the former to the latter.}}
$$\rho:U(n,m\alpha) \rightarrow V(n,m\alpha/0,\infty) $$
with the point in universal target determining $1\in \proj^1$.

As in Section \ref{rigi}, let $T_0(L)$ be the primary insertion
in the rubber theory obtained from $L\in H^2(S,\mathbb{Z})$. 
By the divisor property \eqref{push},
$$m\langle L,\alpha \rangle\,
R_{m\alpha}^{\red}(S\times\mathcal R) = \sum_n q^n \int_{[U(n,m\alpha)]^{\red}}
T_0(L) \,.$$
By rigidification, we find
$$\sum_n q^n \int_{[U(n,m\alpha)]^{\red}}
T_0(L)  =
\sum_n q^n \int_{[V(n,m\alpha/0,\infty)]^{\red}}
T_0(\mathcal{L}) \,.$$
where $\mathcal{L} \in H^4(S \times \proj^1, \mathbb{Z})$
is dual to the cycle 
$$\ell \times \{1\} \subset S \times \proj^1 \ $$
and 
$\ell \subset S$ represents $L\cap [S] \in H_2(S,\mathbb{Z})$.
We define 
$$V_{m\alpha}^{\red}(S\times \proj^1/S_0\cup S_\infty)=\sum_n q^n \int_{[V(n,m\alpha/0,\infty)]^{\red}}
T_0(\mathcal{L})$$ and conclude
\begin{equation}\label{vff33}
m\langle L,\alpha \rangle\,
R_{m\alpha}^{\red}(S\times\mathcal R)=V^{\red}_{m\alpha}(S\times \proj^1/S_0\cup S_\infty) \,. 
\end{equation}

\subsection{Degeneration}

We now study $V^{\red}_{m\alpha}(S\times\proj^1/S_0\cup S_\infty)$
on the right side of \eqref{vff33} via the degeneration formula.

We consider the degeneration of the relative geometry
$S\times\proj^1/S_\infty$ to the normal cone of 
$$S_0 \subset S\times \proj^1\,.$$
The degeneration formula for the virtual class of the moduli of
stable pairs under
\begin{equation}\label{rdd44}
S\times \proj^1/S_\infty \ \rightsquigarrow\  S \times\proj^1/S_\infty
\ \bigcup\  S\times \proj^1/S_0  \cup S_\infty  \ 
\end{equation}
is easily seen to be compatible with the reduced class (since
all the geometries project to $S$). 
Let
\begin{eqnarray*}
V(n,m\alpha/ \infty)& = &P_n( S\times \proj^1/  S_\infty, m\alpha)\,,  \\
V_{m\alpha}^{\red}(S\times \proj^1/S_\infty)
&=&\sum_n q^n \int_{[V(n,m\alpha/\infty)]^{\red}}
T_0(\mathcal{L})\,.
\end{eqnarray*}
for the insertion $T_0(\mathcal{L})$ defined in Section \ref{rigi2}.
\begin{Proposition} We have
$$
V^{\red}_{m\alpha}(S\times \proj^1/S_\infty) =
V^{\red}_{m\alpha}(S\times \proj^1/S_0\cup S_\infty) \,. $$
\end{Proposition}

\begin{proof}
By the degeneration formula, the reduced virtual class
of the moduli space $V(n,m\alpha/\infty)$ distributes to the
products
\begin{equation}
\label{kxx77}
V(n_1,\beta_1/\infty) \times V(n_2,\beta_2/0,\infty)\, , \ \ \
n_1+n_2=n\, ,\  \beta_1+\beta_2=m\alpha\,,
\end{equation}
associated to the reducible target \eqref{rdd44}.
If $\beta_1$ and $\beta_2$ are {\em both} nonzero, 
then the product \eqref{kxx77} admits a double
reduction of the standard virtual class. Hence,
the (singly) reduced virtual class of \eqref{kxx77} vanishes
unless $\beta_1$ or $\beta_2$ is 0.

When the degeneration formula is applied to
$V^{\red}_{m\alpha}(S\times \proj^1/S_\infty)$,
the insertion $T_0(\mathcal{L})$
requires $\beta_2$ to be nonzero. Hence, 
only the $n_2=n$ and
$\beta_2=m\alpha$  term contributes to the
reduced degeneration formula.
\end{proof}

\subsection{Localization} \label{loc2}
The next step is to apply 
$\C^*$-equivariant localization to 
$V^{\red}_{m\alpha}(S\times \proj^1/S_\infty)$.

Let $\C^*$ act on $\proj^1$ with tangent weight $\mathfrak{t}$ at
$0$. We lift the  class $\mathcal{L}$ to $\C^*$-equivariant
cohomology by selecting the $\com^*$-fixed representative
$$L \times \{0\} \subset S \times \proj^1\,.$$
The virtual localization formula expresses 
$V^{\red}_{m\alpha}(S\times \proj^1/S_\infty)$
as a sum over products of residue contributions over 
$S_0$ and $S_\infty$ in $S\times \proj^1/S_\infty$.
The $\C^*$-fixed loci admit a double reduction of their
virtual class unless $m\alpha$ is distributed entirely to
$0$ or  $\infty$. Since the insertion $T_0(\mathcal{L})$
requires the distribution to be over $0$, we conclude
\begin{eqnarray}
V^{\red}_{m\alpha}(S\times \proj^1/S_\infty) & = &
\Big\langle T_0(\mathcal{L}) \Big\rangle_{Y,m\alpha}^{\red}(q) \label{gvv32}\\
& = & mt \langle L,\alpha \rangle\,  \nonumber
\Big \langle 1 \Big\rangle_{Y,m\alpha}^{\red}(q) \, ,
\end{eqnarray}
where we have followed the notation of Section \ref{lc22} except
for writing $\langle 1 \rangle_{Y,m\alpha}^{\red}$ as a series in
 $q$ instead of $y$.

\subsection{Multiple cover formula}
Sections \ref{rigi2}--\ref{loc2} together with Proposition \ref{mc11}
imply
\begin{eqnarray}\label{vxx99}
R^{\red}_{m\alpha}(S\times\mathcal R)& =&  t\Big \langle 1 \Big\rangle_{Y,m\alpha}^{\red}(q) \\
& = & \sum_{k|m} \frac{t}{k} I_{\frac{m^2}{k^2}(h-1)+1}(-(-q)^k) \ \nonumber
\end{eqnarray}
where $\langle \alpha,\alpha \rangle=2h-2$.
By the formulas of Section \ref{kye}, $tI(q)$ is
a rational function{\footnote{The variable $q$
here is the variable $y$ in Section \ref{kye}.}} of $q$.
Hence, $R^{\red}_{m\alpha}(S\times\mathcal R)$ is a rational function of $q$.

We define $R^{\red}_h$ to equal the generating series $R^{\red}_{\alpha}(S\times\mathcal R)$
associated to a primitive and positive class $\alpha$ with norm square $2h-2$.
By \eqref{vxx99} in the $m=1$ case,
$$R^{\red}_h= t I_h(q)\,,$$
so $R^{\red}_h=0$ for $h<0$ by Proposition \ref{zzzzz}.
Rewriting \eqref{vxx99}, we obtain the fundamental multiple
cover formula governing $R^{\red}_{m\alpha}$:
\begin{equation}\label{vzz22}
R^{\red}_{m\alpha}(q)=
\sum_{k|m} \frac{1}{k} R^{\red}_{\frac{m^2}{k^2}(h-1)+1}(-(-q)^k)\,.
\end{equation}
The terms on the right correspond to lower degree
primitive contributions to $m\alpha$.

Finally, we  write the multiple cover formula in terms of
the invariants $\widetilde{R}_{m\alpha}$. Following the
notation \eqref{psszz}, let
$\widetilde{R}_h$ equal the series $\widetilde{R}_{\alpha}(S)$
associated to a primitive and positive class $\alpha$ with norm square $2h-2$.
By Proposition \ref{kedd9}, we see
\begin{equation} \label{xjjx}
\widetilde{R}_{h}=0 \ \ \ \text{for} \ \ h<0\,.
\end{equation}
The multiple cover
formula \eqref{vzz22} implies the following result.

\begin{Theorem} \label{geget} The series $\widetilde{R}_{m\alpha}(q)$ is the
Laurent expansion of a rational function of $q$, and
\begin{equation*}
\widetilde{R}_{m\alpha}(q)=
\sum_{k|m} \frac{1}{k} \widetilde{R}_{\frac{m^2}{k^2}(h-1)+1}(-(-q)^k)\,.
\end{equation*}
\end{Theorem}

\subsection{Stable pairs BPS counts}\label{spbps}
The stable pairs potential $\widetilde{F}_{{\alpha}}(q,v)$ 
for classes proportional
to ${\alpha}$
is 
\begin{equation*} 
\widetilde{F}_{{\alpha}}=
\sum_{n\in \mathbb{Z}}\   \sum_{m>0}\   \widetilde{R}_{n,m\alpha} \ 
q^{n} 
v^{m{\alpha}}\,.
\end{equation*}
The stable pairs BPS counts $\widetilde{r}_{g,m\alpha}$
are uniquely defined \cite{PT1} by:
$$\widetilde{F}_{{\alpha}}= 
\ \ \ \sum_{g\in \mathbb{Z}}  \ \sum_{m>0} \
 \widetilde{r}_{g,m\alpha} \  \sum_{d>0}
\frac{(-1)^{g-1}}{d} 
 \left((-q)^d-2+(-q)^{-d} \right)^{g-1}
 v^{dm\alpha}\,.$$
Because $\widetilde{R}_{m\alpha}$ is a Laurent
series in $q$, we see
$$\widetilde{r}_{g,m\alpha}=0$$
for sufficiently high $g$ and fixed $m$.
In the primitive case, 
\begin{equation}\label{lsszz}
\widetilde{r}_{g,\alpha}=0 \ \ \ \text{for} \ g<0
\end{equation}
by equation (2.10) of \cite{PT2}.

Since we know $\widetilde{R}_{m\alpha}$ only depends upon $m$
and the norm square $\langle m\alpha,m\alpha \rangle$, 
the same is true for the associated BPS counts. 
Following the notation \eqref{gqgq},
we define
$$\widetilde{r}_{g,m,m^2(h-1)+1} = \widetilde{r}_{g,m\alpha}\,.$$

\begin{Proposition} \label{fcc19}
The stable pairs BPS counts do
not depend upon the divisibility:
$$\widetilde{r}_{g,m,m^2(h-1)+1}= \widetilde{r}_{g,1,m^2(h-1)+1}\,.$$
\end{Proposition}

\begin{proof} By Theorem \ref{geget}, we can write 
$\widetilde{F}_\alpha$ as
\begin{eqnarray*}
\widetilde{F}_{{\alpha}} &=&
   \sum_{m>0}\   \widetilde{R}_{m\alpha} \ 
v^{m{\alpha}}\ \\
&=& \sum_{m>0} 
\sum_{k|m} \frac{1}{k} \widetilde{R}_{\frac{m^2}{k^2}(h-1)+1}(-(-q)^k)
v^{m\alpha}\,.
\end{eqnarray*}
Next, using the definition of BPS counts for primitive
classes, we find
\begin{equation*}
\widetilde{F}_{{\alpha}}  = 
\sum_{m>0} 
\sum_{k|m} 
\sum_{g\geq 0} \frac{(-1)^{g-1}}{k}
\widetilde{r}_{g,1,\frac{m^2}{k^2}(h-1)+1}
((-q)^k-2+(-q)^{-k})^{g-1} v^{m\alpha} \,.
\end{equation*}
After a reindexing of the summation on the right, we obtain
$$
\ \ \ \sum_{g\geq 0}  \ \sum_{m>0} \
 \widetilde{r}_{g,1,m^2(h-1)+1} \  \sum_{d>0}
\frac{(-1)^{g-1}}{d} 
 \left((-q)^d-2+(-q)^{-d} \right)^{g-1}
 v^{dm\alpha}\,.$$
By the definition of the BPS counts and the uniqueness
statement, we conclude  the
 $\widetilde{r}_{g,m,m^2(h-1)+1}$ does not
depend upon the divisibility.
\end{proof}

As a Corollary of Proposition \ref{fcc19}, we
obtain basic properties of $\widetilde{r}_{g,m,h}$
required in Section \ref{t22}.

\begin{Corollary} \label{ww22} We have $\widetilde{r}_{g,m,h\leq 0}=0$
except for the case
$$\widetilde{r}_{0,1,0}=1\,.$$
\end{Corollary}

\begin{proof}  
By Proposition \ref{fcc19}, we need only
consider the $m=1$ case.
If $h<0$, the vanishing follows from \eqref{xjjx}.
If $h=0$, the result is the consequence of
the stable pair calculation of the conifold \cite{PT1}.
\end{proof}

\begin{Corollary}\label{ww3} We have $\widetilde{r}_{g<0,m,h}=0$.
\end{Corollary}

\begin{proof} After reducing to the $m=1$ case,
the result is \eqref{lsszz}.
\end{proof}



\section{P/NL correspondence}
\label{pnlc}

\subsection{Overview} Our goal here is to prove the Pairs/Noether-Lefschetz
correspondence of Theorem \ref{ffc2} in Section \ref{spnl}. The main tool needed is Theorem  \ref{t66t}.

Following the definitions of Sections \ref{fams} and  \ref{fccv}, let $$(\pi_3:\XX\rightarrow \proj^1,L_1,L_2,E)$$ 
be the 1-parameter family of $\Lambda$-polarized, 
$$\Lambda =  \left( \begin{array}{ccc}
2 & 3 & 0  \\
3 & 0 & 0\\
0 & 0 & -2  \end{array} \right),$$
$K3$ surfaces obtained
from a very general{\footnote{Very general here is the
complement of a countable set.}}
anticanonical Calabi-Yau hypersurface,
$$\XX\subset \widetilde{\proj^2 \times \proj^1} \times \proj^1\,.$$
For a
very general
fiber of the base $\xi \in \proj^1$, 
\begin{equation}\label{mmqq22}
\text{Pic}(\XX_\xi) \cong \Lambda\,.
\end{equation}
We also assume, for each nodal fiber{\footnote{The nodal fibers
have exactly 1 node.}} $\XX_\xi$, the $K3$ resolution $\widetilde{\XX}_\xi$ satisfies
\begin{equation}\label{mmqq33}
\text{Pic}(\widetilde{\XX}_\xi) \cong \Lambda \oplus \mathbb{Z}[\widetilde{E}]
\end{equation}
where $\widetilde{E}\subset \widetilde{\XX}_\xi$ is the exceptional $-2$ curve.
Both \eqref{mmqq22} and \eqref{mmqq33} can be satisfied
since $\Lambda$ is the Picard lattice of a very general{\footnote{We
leave these standard facts about $K3$ surfaces of type $\Lambda$
to the reader.
The upcoming text by Huybrechts is an excellent source for the study of
$K3$ surfaces.}} point of
$\mathcal{M}_\Lambda$ and $\Lambda \oplus \mathbb{Z}[\widetilde{E}]$ is the 
Picard lattice of a very general point of the nodal locus
in $\mathcal{M}_\Lambda$. 

The stable pairs potential $\widetilde{F}^{\XX}$ for nonzero vertical classes
is the series
\begin{eqnarray*}
\widetilde{F}^{{\XX}}& =& \log \left(1 + 
\sum_{0\neq \gamma\in H_2({\XX},\mathbb{Z})^{\pi_3}}  
 \ZZ_{\mathsf{P}}\Big(\XX;q\Big)_{\!\gamma}\ v^\gamma\right) \\
& = & \sum_{n\in \mathbb{Z}}\ \sum_{0 \neq \gamma\in H_2(\XX,\mathbb{Z})^{\pi_3}}
 \widetilde{N}_{n,\gamma}^\XX \ q^n v^\gamma\,.
\end{eqnarray*}
Here, $v$ is the curve class variable, and the second equality 
defines
the connected stable pairs invariants $\widetilde{N}_{n,\gamma}^\XX$.
The stable pairs
BPS counts $\widetilde{n}_{g,\gamma}^{{\XX}}$
are then defined 
by 
\begin{equation*}
\widetilde{F}^{{\XX}}  =    \sum_{g\in \mathbb{Z}} \ \sum_{0\neq \gamma\in
H_2({\XX},\mathbb{Z})^{\pi_3}} 
\widetilde{n}_{g,\gamma}^{{\XX}} \ \sum_{d>0}
\frac{(-1)^{g-1}}{d}  \left((-q)^d-2+(-q)^{-d} \right)^{g-1}
v^{d\gamma} \,. 
\end{equation*}

Let $\widetilde{n}_{g,(d_1,d_2,d_3)}^\XX$ denote the stable pairs BPS invariant of $\XX$
in genus $g$ for $\pi_3$-vertical curve classes of degrees $d_1,d_2,d_3$
with respect to the line bundles $L_1,L_2, E\in \Lambda$ respectively. 
Let $\widetilde{r}_{g,m,h}$ be the stable pairs BPS counts associated to $K3$
surfaces in Section \ref{k3intp}.
The  Pairs/Noether-Lefschetz correspondence is the following result.

\vspace{8pt}
\noindent{\bf Theorem 4.} 
{\em For degrees $(d_1,d_2,d_3)$ positive with respect to the
quasi-polarization,}
$$\widetilde{n}_{g,(d_1,d_2,d_3)}^\XX= \sum_{h=0}^\infty \sum_{m=1}^{\infty}
\widetilde{r}_{g,m,h}\cdot  NL_{m,h,(d_1,d_2,d_3)}^{\pi_3}\,.$$

\subsection{Strategy of proof} \label{stopr}
Since the formulas relating the BPS counts to stable pairs invariants
are the same for $\XX$ and the $K3$ surface,  Theorem \ref{ffc2} is 
equivalent to the analogous 
stable pairs statement:
\begin{equation}\label{cq29}
\widetilde{N}_{n,(d_1,d_2,d_3)}^\XX= \sum_{h} \sum_{m=1}^{\infty}
\widetilde{R}_{n,m,h}\cdot  NL_{m,h,(d_1,d_2, d_3)}^{\pi_3}
\end{equation}
for degrees $(d_1,d_2,d_3)$ positive with respect
to the quasi-polarization of $\XX$.

We denote the base of the fibration $\pi_3$ by $C\cong \proj^1$.
Let $C^\circ \subset C$ be the locus over which $$\pi_3:\XX \rightarrow C$$ 
is smooth.
By condition (i) of Section \ref{fams} for a
{1-parameter family of $\Lambda$-polarized
$K3$ surfaces}, the complement of $C^\circ$ consists of finitely
many points over which each fiber of $\pi_3$ has a single ordinary
node.
For each $\xi\in C^\circ$, let 
$$\mathcal{V}_\xi= H^2(\XX_\xi,\mathbb{Z})\,.$$

As $\xi\in C^\circ$ varies, the fibers $\mathcal{V}_\xi$ 
determine a local system 
$$\mathcal{V}^\circ \rightarrow C^\circ\,.$$ 
We denote the effective divisor classes on $\XX_\xi$ by
$$\text{Eff}_\xi = \{ \beta \in 
\mathcal{V}_\xi
 \ | \ \beta\in \text{Pic}(\XX_\xi) \ \text{and $\beta$ effective}
\}.$$

For 
$\xi \in C\setminus C^\circ$, we denote the
$K3$ resolution of singularities of the node of $\XX_\xi$ by
$$\rho:\widetilde{\XX}_\xi \rightarrow \XX_\xi\,,$$
and we define
 $$\mathcal{V}_\xi= H^2(\widetilde{\XX}_\xi,\mathbb{Z})\,.$$
As before, let
$$\text{Eff}_\xi= \{ \beta \in 
\mathcal{V}_\xi
 \ | \ \beta\in \text{Pic}(\widetilde{\XX}_\xi) \ \text{and $\beta$ effective}
\}\,.$$
The push-forward of a divisor class on $\widetilde{\XX}_\xi$
to $\XX_\xi$ can be considered in $H_2(\XX_\xi,\mathbb{Z})$
and, by Poincar\'e duality (for the quotient singularity), in
$$\frac{1}{2} H^2(\XX_\xi,\mathbb{Z}) \subset 
H^2(\XX_\xi,\mathbb{Q})\,.$$ 
We view the push-forward by $\rho$ of the effective divisors classes
as:
\begin{equation}
\label{rhop}
\rho_*: \text{Eff}_\xi \rightarrow \frac{1}{2}H^2(\XX_\xi,\mathbb{Z})\,.
\end{equation}

We will study the contributions of the 
classes $\text{Eff}_\xi$ to both sides of \eqref{cq29}.
Certainly, only effective curves contribute to the left side of
\eqref{cq29}. Suppose $\xi\in C$ lies on the Noether-Lefschetz divisor{\footnote{We follow the notation of Section \ref{nld}.}}
$$D_{m,h,(d_1,d_2,d_3)}\subset \mathcal{M}_{\Lambda}\,.$$
 Then there exists 
$\beta\in \text{Pic}(\XX_\xi)$ if $\xi\in C^\circ$ (or 
 $\beta\in \text{Pic}(\widetilde{\XX}_\xi)$
if $\xi \in C\setminus C^\circ$)
of divisibility $m$, 
$$2h-2=\langle \beta, \beta \rangle\,,$$
and degree $(d_1,d_2,d_3)$ positive with respect to the
quasi-polarization.
Let $\alpha= \frac{1}{m}\beta$ be the corresponding primitive class. If 
$\beta$ is not effective on $\XX_\xi$ (or
$\widetilde{\XX}_\xi$ if $\xi\in C\setminus C^\circ$), then $\alpha$ is 
also not
effective. Since $\alpha$ is positive with respect to the
quasi-polarization, Riemann-Roch implies
$h_\alpha<0$,  where
$$2h_\alpha-2=\langle \alpha,\alpha \rangle\,.$$
By Theorem \ref{geget} and the vanishing \eqref{xjjx},
$$\widetilde{R}_{n,m,h}=0\,.$$
Ineffective classes
$\beta$
 therefore
do not contribute to the right side of \eqref{cq29}.

\subsection{Isolated contributions} \label{iso1}
We consider first the simplest contributions.
Let $\xi\in C^\circ$.
A nonzero effective class $\beta\in \text{Pic}(\XX_\xi)$
is {\em completely isolated} on $\XX$
if the following property holds:
\begin{enumerate}
\item[$(\star\star)$] for 
every
 {\em effective} decomposition
$$\beta=\sum_{i=1}^l \gamma_i
\in \text{Pic}(\XX_\xi)\,,$$
the  local 
Noether-Lefschetz locus $\text{NL}(\gamma_i)\subset C$ corresponding to
each  class $\gamma_i \in \text{Pic}(\XX_\xi)$ contains
$\xi$ as an isolated{\footnote{Nonreduced structure is allowed at
$\xi$.}} point.
\end{enumerate}

Let $\gamma\in \text{Pic}(\XX_\xi)$ be an effective summand of
$\beta$ which occurs in condition $(\star\star)$.
The stable pairs with set-theoretic
support on $\XX_\xi$ form an open and closed
component,
$$P_n^\xi(\XX,\gamma)\subset P_n(\XX,\gamma)$$
 of the moduli space of stable pairs{\footnote{As usual,
we denote the push-forward of $\gamma$ to $H_2(\XX,\mathbb{Z})$
also by $\gamma$.}}  
for every $n$.

We consider now the contribution of $\XX_\xi$ for $\xi\in C^\circ$
to the
stable pairs series $\widetilde{F}^\XX$ of a nonzero effective
and completely isolated class $\beta\in H^2(\XX_\xi,\mathbb{Z})$
satisfying
$$\text{div}(\beta)=m\,, \ \ \ \langle \beta,\beta \rangle= 2h-2\ $$
and of degree $\dd$ with respect to $\Lambda$.
More precisely, let  
$\text{Cont}(\XX_\xi,\beta,\widetilde{N}^\XX_{n,\dd} )$ be
the contribution corresponding to $\beta$ of
all the moduli of stable pairs with set theoretic
support on $\XX_\xi$:
\begin{multline} \label{gttl2}
\text{Cont}(\XX_\xi,\beta,\widetilde{N}^\XX_{n,\dd} ) =
 \\
\text{Coeff}_{q^nv^{\beta}} \left[
\log\left(1 + \sum_{n\in \mathbb{Z}} \sum_{\gamma\in \text{Eff}({\beta})} 
q^n v^{\gamma}
\int_{[P_n^\xi(\XX,\gamma)]^{\vir}} 1\right)\right]\,,
\end{multline}
where 
$\text{Eff}_\xi(\beta) \subset \text{Pic}(\XX_\xi)$
is the subset of effective summands
of 
$\beta$. 
By condition $(\star\star)$, the contribution is {\em well-defined}.

Let $\ell_\beta$ be the length of the local Noether-Lefschetz
locus $\text{NL}(\beta) \subset C$ at $\xi\in C$.
We define 
$$\text{Cont}(\XX_\xi,\beta, NL^{\pi_3}_{m,h,\dd} ) =\ell_\beta\,,$$ 
the local intersection contribution to the Noether-Lefschetz number. 

\begin{Proposition} \label{crt} For a
completely isolated effective class $\beta\in \Pic(\XX_\xi)$ with
$\xi\in C^\circ$, 
we have
$$\text{\em Cont}(\XX_\xi,\beta,\widetilde{N}^\XX_{n,\dd} ) =
\widetilde{R}_{n,m,h} \cdot
\text{\em Cont}(\XX_\xi,\beta, NL^{\pi_3}_{m,h,\dd} )\,.$$
\end{Proposition}


\begin{proof}
We perturb the family $C$ locally near $\xi$ to be transverse
to {\em all} the local Noether-Lefschetz loci corresponding
to effective summands of $\beta$  on $\XX_\xi$.
In order to perturb in algebraic geometry,
we first approximate $C$ by $C'$ near $[\XX_\xi]$ to sufficiently 
high order{\footnote{The
order should be high enough to obstruct all
the deformations away from $\XX_\xi$ of the
stable pairs occurring on the right side of \eqref{gttl2}.}}
by a moving family of curves in the moduli space of $\Lambda$-polarized
$K3$ surfaces. Then, we perturb the resulting moving curve $C'$
to $C''$ to achieve the desired transversality. Since
transversality is a generic condition, we may take the
perturbation $C''$ to be as small as necessary.
Let
$$\pi: \XX'' \rightarrow C''$$
be the family of $\Lambda$-polarized $K3$ surfaces determined
by $C''$.

Near
$[\XX_\xi]$ in the moduli of $\Lambda$-polarized $K3$ surfaces, 
the local system of second cohomologies is trivial.
In a contractible neighborhood $U$ of $[\XX_\xi]$, we have a
canonical isomorphism 
\begin{equation}\label{gzza}
H^2(S,\mathbb{Z}) \cong H^2(\XX_\xi,\mathbb{Z})
\end{equation}
for all $[S]\in U$. We will use the identification
\eqref{gzza} when discussing
$H^2(\XX_\xi,\mathbb{Z})$.

By deformation invariance of the stable pairs theory,
the contribution \eqref{gttl2} can be calculated after perturbation.
We assume all the intersections of $C''$ with the above local Noether-Lefschetz
loci 
which have limits tending to $[\XX_\xi]$ (as $C''$ tends to $C$)
lie in $U$.

For every $\gamma\in \text{Eff}_\xi(\beta)$, the curve $C''$ intersects
the local Noether-Lefschetz loci associated to $\gamma$
transversely at a finite set of reduced points $$I_\gamma\subset C''$$
near $[\XX_\xi]$.
The local
intersection number at $\xi\in C$ of $C$ with the local
Noether-Lefschetz locus corresponding to $\gamma$ is $|I_\gamma|$.
Let $\xi_\gamma\in I_\gamma$
be one such  point of intersection. 
Let $m_\gamma$ be the divisibility of $\gamma$.

On the $K3$ surface $\XX''_{\xi_\gamma}$, the class
$\frac{1}{m_\gamma}\gamma$ is primitive and positive since $\gamma$ is positive on $\XX_\xi$. Moreover,
the family $C''$ is 
$m_\gamma$-rigid for $\frac{1}{m_\gamma}\gamma$ on $\XX''_{\xi_\gamma}$ since
the effective summands of $\gamma$ on
$\XX''_{\xi_\gamma}$ all lie in $\text{Eff}_\xi(\beta)$.
By upper semicontinuity, no more effective summands
of $\beta$ can appear near $[\XX_\xi]$ in the moduli
of $K3$ surfaces.

By Corollary \ref{g4545} to Theorem \ref{t66t}, we have a formula for the 
stable pairs contribution of $\gamma$ at $\XX''_{\xi_\gamma}$, 
\begin{equation*}
P^{\star,\xi_\gamma}_{\gamma}(\XX'') = \text{Coeff}_{v^{\gamma}} \left[
\exp\left(  \sum_{\widehat{\gamma} \in \text{Eff}_{\xi_\gamma}(\gamma)} 
v^{\widehat{\gamma}} R^{\red}_{\widehat{\gamma}}(\XX''_{\xi_\gamma}\times\mathcal R)
\right)\right]\,.
\end{equation*}
Here, $\text{Eff}_{\xi_\gamma}(\gamma)\subset \text{Pic}(\XX''_{\xi_\gamma})$ 
is the subset of effective summands of $\gamma$ (which is empty if
$\gamma$ is not effective).

Finally, consider the original contribution 
$\text{Cont}(\XX_\xi,\beta,\widetilde{N}^\XX_{n,\dd} )$.
Let $$I\subset C''$$ be the union of all the $I_\gamma$
for $\gamma\in \text{Eff}_\xi(\beta)$. For each $\widehat{\xi}\in I$,
let 
$$\text{Alg}_{\widehat{\xi},\xi}(\beta) \subset \text{Eff}_\xi(\beta)$$
be the subset of classes of $\text{Eff}_\xi(\beta)$
which are algebraic on $\XX''_{\widehat{\xi}}$. 
All elements of $\text{Alg}_{\widehat{\xi},\xi}(\beta)$ are positive.
By semicontinuity
$$\text{Eff}_{\widehat{\xi}}(\gamma) \subset 
\text{Alg}_{\widehat{\xi},\xi}(\beta)$$
for all $\gamma\in \text{Eff}_\xi(\beta)$.
After perturbation, our formula for \eqref{gttl2} is
$$
\text{Coeff}_{q^nv^{\beta}} \left[
\log\left(\prod_{\widehat{\xi}\in I} 
\exp\left(  \sum_{{\gamma} \in \text{Alg}_{\widehat{\xi},\xi}(\beta)} 
v^{{\gamma}} R^{\red}_{{\gamma}}(\XX''_{\widehat{\xi}}\times\mathcal R)
\right) \right)\right]\,.
$$
After taking the logarithm of the exponential, we have
$$\text{Cont}(\XX_\xi,\beta,\widetilde{N}^\XX_{n,\dd} ) =
\text{Coeff}_{q^nv^{\beta}} \left[
\sum_{\widehat{\xi}\in I} 
 \sum_{{\gamma} \in \text{Alg}_{\widehat{\xi},\xi}(\beta)} 
v^{{\gamma}} R^{\red}_{{\gamma}}(\XX''_{\widehat{\xi}}\times\mathcal R)
\right]\,.
$$
Hence, we see only the $\xi\in I$ for which $\beta \in 
\text{Alg}_{\widehat{\xi},\xi}(\beta)$
contribute. We conclude
\begin{eqnarray*}
\text{Cont}(\XX_\xi,\beta,\widetilde{N}^\XX_{n,\dd} ) & = &
\sum_{\widehat{\xi}\in I_\beta} 
\text{Coeff}_{q^n} \left[
R^{\red}_{\beta}(\XX''_{\widehat{\xi}}\times\mathcal R)\right]\\
& = &
\widetilde{R}_{n,m,h} \cdot |I_\beta|\,,
\end{eqnarray*} 
where the second equality uses Proposition \ref{kedd9}.

Since
 $\text{Cont}(\XX_\xi,\beta, NL^\pi_{m,h,\dd} )=|I_\beta|$
is the local contribution of $\beta$ to the Noether-Lefschetz
number, the Proposition is established.
\end{proof}

We have proven Proposition \ref{crt} for our original family $\pi_3$
of $\Lambda$-polarized $K3$ surfaces. 
Definition $(\star\star)$
of a {\em completely isolated} class is valid for 
any family 
\begin{equation}\label{anyf}
\pi:X \rightarrow C
\end{equation}
of
lattice polarized $K3$ surfaces.
By the proof given, Proposition \ref{crt} is valid for the
contributions of every
 completely isolated
class 
$$\beta\in H^2(X_\xi,\mathbb{Z})\,,\ \ \ \ \xi\in C^\circ$$
 for any family
\eqref{anyf}. For different lattices $\widehat{\Lambda}$,
the 
degree index $\dd$ in Proposition \ref{crt} is replaced by the degree with
respect to a basis of $\widehat{\Lambda}$.

\subsection{Sublattice $\widehat{\Lambda} \subset \Lambda$}
\label{lsel}
For $\ell>0$, we define
$$\text{Eff}_\XX(\lambda^\pi,\ell)\in H_2(\XX,\mathbb{Z})^\pi$$ to be the set of
classes of degree at most $\ell$ 
(with respect to the quasi-polarization $\lambda^\pi$)
which are
 represented by algebraic curves
on $\XX$. 
By the boundedness of the Chow variety of curves of $\XX$ of degree at most $\ell$,
$\text{Eff}_\XX(\lambda^\pi,\ell)$ is 
a {\em finite} set. 

For any quasi-polarization $\delta$ given by an
ample class  of $\XX$ in $\Lambda$, let the set
of effective classes of degree at most $\ell$ (with respect to $\delta$)
be
$$\text{Eff}_\XX(\delta,\ell)\in H_2(\XX,\mathbb{Z})^\pi\,.$$
Similarly, let{\footnote{We follow the notation of Section \ref{stopr}.}}
$$\text{Eff}_\xi(\delta,\ell)\subset \text{Eff}_\xi$$
be the effective curve classes of $\delta$-degree at most $\ell$
over $\xi\in C$.

We select a primitive sublattice  
$\widehat{\Lambda}=\mathbb{Z}\delta_1
\oplus \mathbb{Z}\delta_2 
\subset \Lambda$ satisfying
the following properties:
\begin{enumerate}
\item[(i)] $\delta_1$ is a quasi-polarization with 
$$L = \text{Max}\{\ \langle \delta_1,\gamma \rangle  \ |\ \gamma
\in \text{Eff}_{\XX}(\lambda^\pi,\ell)\ \}\,,$$
\item[(ii)] 
$\widehat{\Lambda}\ \cap\ \text{Eff}_{\xi}(\delta_1,L)=\emptyset$
for all $\xi\in C^\circ$,
\item[(iii)]
$\frac{1}{2}\widehat{\Lambda}\ \cap\ \rho_*\text{Eff}_{\xi}(\delta_1,L)=\emptyset$
for all $\xi\in C\setminus C^\circ$
\item[(iv)] the intersection maps
\begin{eqnarray*}
\text{Eff}_X(\delta_1,L) \rightarrow \text{Hom}(\Lambda,\mathbb{Z})\,, & &  
\gamma \mapsto \langle \star,\gamma\rangle\,, \\
\text{Eff}_X(\delta_1,L)\rightarrow \text{Hom}(\widehat{\Lambda},
\mathbb{Z})\,, & &
\gamma \mapsto \langle \star,\gamma\rangle
\end{eqnarray*}
have images of the same cardinality.
\end{enumerate}
We will find such $\delta_1,\delta_2 \in \Lambda$
by a method explained below: we will first select $\delta_1$
and then $\delta_2$.

Consider the quasi-polarization $\widehat{\delta}=k\lambda^{\pi_3}$ for
$k>2\ell$. Then $\frac{1}{2}\widehat{\delta}$ has $\widehat{\delta}$-degree
$\frac{k^2}{2}\langle \lambda^{\pi_3},\lambda^{\pi_3}\rangle$ on the $K3$
fibers
while 
$$\text{Max}\{\ \langle \widehat{\delta},\gamma \rangle  \ |\ \gamma
\in \text{Eff}_{X}(\lambda^\pi,\ell)\ \}\leq k\ell< \frac{k^2}{2}
\leq \frac{k^2}{2}\langle \lambda^{\pi_3},\lambda^{\pi_3}\rangle\,.$$
We choose $\delta_1$ to be a small shift of $k\lambda^{\pi_3}$ in the
lattice $\Lambda$ to ensure primitivity. 
Hence, we have met the conditions
\begin{equation}\label{fced2}
\mathbb{Z}\delta_1 \ \cap \ \text{Eff}_\xi(\delta_1,L)=\emptyset\,, \ \ \ \
\frac{1}{2}\mathbb{Z}\delta_1 \ \cap \ \rho_*\text{Eff}_\xi(\delta_1,L)=\emptyset
\end{equation}
for $\xi \in C^\circ$ and $\xi \in C\setminus C^\circ$
respectively.
Conditions \eqref{fced2} are required for (ii) and (iii).

Again by the boundedness
 of the  Chow variety of curves in $\XX$, the following subset of 
$\frac{1}{2}\Lambda$ is a finite set:
$$ \bigcup_{\xi\in C^\circ} 
\Lambda \cap \text{Eff}_\xi(\delta_1,L) \
\cup\  \bigcup_{\xi\in C\setminus C^\circ} 
\frac{1}{2}\Lambda\cap
\rho_*\text{Eff}_\xi(\delta_1,L)\  \ \subset\ \frac{1}{2}
\Lambda\,.$$
Since the rank of $\Lambda$ is 3, we can easily find $\delta_2\in \Lambda$
such that $$\mathbb{Z} \delta_1\oplus
\mathbb{Z}\delta_2 \subset \Lambda$$ is primitive and 
conditions (ii) and (iii) are satisfied.

Finally, since $\text{Eff}_X(\delta_1,L)$ is a finite set,
condition (iv) is easily satisfied when $\delta_2$ is
selected as above.
Since 
$$\text{Eff}_X(\delta_1,L) \rightarrow \text{Hom}(\Lambda,\mathbb{Z})
\rightarrow \text{Hom}(\widehat\Lambda,\mathbb{Z})\,,$$
condition (iv) states there is no loss of information
for a class in $\text{Eff}_X(\delta_1,L)$ in taking
degrees in $\widehat{\Lambda}$ instead of $\Lambda$.

\subsection{Non-isolated contributions: smoothing}\label{ssmo}
Let $\widehat{\Lambda}\subset \Lambda$ be the rank 2
lattice selected in Section \ref{lsel}.
We consider here the moduli space of 
$\widehat{\Lambda}$-polarized $K3$ surfaces which is
1 dimension larger than the moduli of $\Lambda$-polarized
$K3$ surfaces,
$$\mathcal{M}_\Lambda \subset \mathcal{M}_{\widehat{\Lambda}}\,.$$

To the curve $C \subset \mathcal{M}_\Lambda$, we can attach
a nonsingular complete curve $C' \subset \mathcal{M}_{\widehat{\Lambda}}$
satisfying the following properties:
\begin{enumerate}
\item[(1)] $C \cup C'$ is a connected nodal curve with nodes
occurring at very general points of $C$,
\item[(2)] $C'$ does not lie in any of the finitely many
Noether-Lefschetz
divisors of $\mathcal{M}_{\widehat{\Lambda}}$ 
determined by
 $$\Lambda \cap \text{Eff}_\xi(\delta_1,L) \ \ \text{for} \ \ \xi\in C\,,$$
and $C'$ is transverse to 
the Noether-Lefschetz divisor of nodal $K3$ surfaces, 
\item[(3)] $C\cup C'$ smooths in $\mathcal{M}_{\widehat{\Lambda}}$
to a nonsingular curve 
$$C''\subset \mathcal{M}_{\widehat{\Lambda}}$$
which also does not lie in any of the finitely many
Noether-Lefschetz divisors listed in (2) and is also
transverse to the nodal divisor.
\end{enumerate}

Let $\xi \in C$ and let $\beta\in \Lambda \cap \text{Eff}_\xi(\delta_1,L)$. By the construction of $\widehat{\Lambda}$ in Section \ref{lsel},
$\beta\notin \widehat{\Lambda}$.
Let 
$$e_1=\langle \delta_1\,,\beta\rangle\,,\ \  e_2=\langle \delta_2,\beta\rangle , 
\ \ 
2h-2=\langle \beta,\beta\rangle \,.$$
By the Hodge index theorem, the intersection
form on the lattice generated by $\widehat{\Lambda}$ and $\beta$ is of
signature
$(1,2)$. In particular, the discriminant  
$$\Delta(\widehat{\Lambda}\oplus \mathbb{Z}\beta) = (-1)^{2} \det
\begin{pmatrix}
\langle \delta_{1},\delta_{1}\rangle & \langle \delta_1,\delta_2\rangle
 & e_1 \\
\langle \delta_{2},\delta_{1}\rangle & \langle \delta_2,\delta_2\rangle
 & e_2 \\
 e_1 & e_2 & 
2h-2 
\end{pmatrix}$$
is positive. Hence, by the construction of Section \ref{nld}, 
the Noether-Lefschetz divisor
$$D_{m,h,(e_1,e_2)} \subset \mathcal{M}_{\widehat{\Lambda}}$$
is of pure codimension 1.
Conditions (2) and (3) require $C'$ and $C''$ respectively
not to lie in the Noether-Lefschetz divisors obtained from
 $$\Lambda \cap \text{Eff}_\xi(\delta_1,L)\,, \ \ \ \text{for} \ \xi\in C \ $$
and the nodal Noether-Lefschetz locus $D_{1,0,(0,0)}$.
Together, these are finitely many proper divisors in 
$\mathcal{M}_{\widehat{\Lambda}}$.

There is no difficulty in finding $C'$ and the smoothing $C''$
since the moduli space $\mathcal{M}_{\widehat{\Lambda}}$
has a Satake compactification as a projective variety of dimension
18 with boundary of dimension 1. We first find a nonsingular
projective surface 
\begin{equation}\label{rt2222}
Y\subset \mathcal{M}_{\widehat{\Lambda}}
\end{equation}
which contains
$C$ and does {\em not} lie in any of the Noether-Lefschetz
divisors listed in (2). Then $C'$ is chosen to be a very ample
section of $Y$ whose union with $C$ smooths to a very ample
section $C''$ of $Y$. As divisor classes
$$[C]+[C']=[C''] \in \text{Pic}(Y)\,, $$
and the smoothing occurs as a pencil of divisors in the
linear series on $Y$.

\subsection{Non-isolated contributions: degeneration} 
\label{nonano}
Let $X$, $X'$, and $X''$ denote the families of $\widehat{\Lambda}$-polarized
$K3$ surfaces
$$\pi: X \rightarrow C\,, \ \ \pi': X' \rightarrow C'\,, \ \ 
\pi'': X'' \rightarrow C''$$
obtained from the curves $C,C',C'' \subset \mathcal{M}_{\widehat{\Lambda}}$.
By the transversality conditions in (2) and (3) of Section
\ref{ssmo} with respect to the
nodal Noether-Lefschetz divisor, the families $X'$ and $X''$ are
nonsingular
and have only finitely many nodal fibers.
Of course, $\pi$ is just 
$$\pi_3:\XX\rightarrow C$$ viewed with a different lattice
polarization.
As $C''$ degenerates to the curve
$C\cup C'$, we obtain a degeneration of
3-folds
$$ X'' \rightsquigarrow X\cup X'$$
with nonsingular total space.

Consider the degeneration formula for stable pairs invariants
of $X''$ in fiber classes.
The degeneration formula expresses such stable pairs invariants
of $X''$ in terms of the relative stable pairs
invariants of
$$X\ /\ \pi^{-1}(C\cap C') \  \ \ \ \text{and} \ \ \
X'\ /\ (\pi')^{-1}(C\cap C')\ $$
in fiber classes.
Degeneration to the normal cone of $\pi^{-1}(C\cap C') \subset X$
(together with usual $K3$ vanishing via the reduced theory) 
shows the relative invariants of
$X/ \pi^{-1}(C\cap C')$ equal the standard stable pairs invariants
of $X$ in fiber classes. Similarly, the relative invariants
of $X'/(\pi')^{-1}(C\cap C')$ equal the standard stable pairs invariants
of $X'$.

We index the fiber class invariants of $X$, $X'$, and $X''$ by the
degrees measured against $\delta_1$ and $\delta_2$.
The stable pairs partition functions are:
\begin{eqnarray*}
\ZZ_{\mathsf{P}}(X;q)& =& 1 + \sum_{(e_1,e_2)\neq (0,0)}  
 \ZZ_{\mathsf{P}}\Big(X;q\Big)_{(e_1,e_2)} v_1^{e_1}v_2^{e_2}\,,\\
\ZZ_{\mathsf{P}}(X';q)& =& 1 + \sum_{(e_1,e_2)\neq (0,0)}  
 \ZZ_{\mathsf{P}}\Big(X';q\Big)_{(e_1,e_2)} v_1^{e_1}v_2^{e_2}\,,\\
\ZZ_{\mathsf{P}}(X'';q)& =& 1 + \sum_{(e_1,e_2)\neq (0,0)}  
 \ZZ_{\mathsf{P}}\Big(X'';q\Big)_{(e_1,e_2)} v_1^{e_1}v_2^{e_2} \,.
\end{eqnarray*}
Since $\delta_1$ is ample, only terms with $e_1>0$
can occur in the above sums.
The degeneration formula yields the following result.

\begin{Proposition} \label{party} We have
$$\ZZ_{\mathsf{P}}(X'';q) = \ZZ_{\mathsf{P}}(X;q)\cdot \ZZ_{\mathsf{P}}(X';q)\,.$$
\end{Proposition}

Let $\widetilde{F}^X$, $\widetilde{F}^{X'}$, and $\widetilde{F}^{X''}$
denote the logarithms of the partition functions of 
$X$, $X'$, and $X''$ respectively. Proposition \ref{party} yields
the relation:
\begin{equation}\label{xfft}
\widetilde{F}^{X''} = \widetilde{F}^{X} +\widetilde{F}^{X'\,}.
\end{equation}

We now restrict ourselves to fiber classes of $\delta_1$-degree
bounded by $L$ (as specified in the construction of 
$\widehat{\Lambda}$ in Section \ref{lsel}). For such classes
on $X'$, we will divide $\widetilde{F}^{X'}$ into two summands.

\begin{Lemma}\label{hkkh}
There are no curves of $X'$ in class
 $\beta \in H_2(X',\mathbb{Z})^{\pi'}$
of $\delta_1$-degree bounded by $L$ which move in
families
dominating $X'$.
\end{Lemma} 

\begin{proof}
If a such a family of curves were to dominate $X'$, then 
every fiber $X'_{\xi'}$ as $\xi'$ varies in  $C'$ would contain an 
effective curve class
of the family. In particular, the fibers over $\xi'\in C \cap C'$
would contain such effective curves.
By construction,
$\xi'\in C \cap C'$ is a very general point of $C$. Therefore,
$$\text{Pic}(X'_{\xi'})=\text{Pic}(X_{\xi'}) \cong \Lambda\,.$$
Since $C'$ was chosen {\em not} to lie in any Noether-Lefschetz
divisors associated to effective curves on $X_{\xi'}$ in $\Lambda$ of 
$\delta_1$-degree bounded by $L$, we have a contradiction.
\end{proof}

By Lemma \ref{hkkh}, we can separate the contributions of
the components of $P_n(X',\beta)$ by the points $\xi'\in C'$
over which they lie:
\begin{equation}\label{xfft2}
\widetilde{F}^{X',L} = \sum_{\xi'\in C\cap C'} \widetilde{F}^{X',L}_{\xi'}
\ +\  \widetilde{F}^{X',L}_{C'\setminus (C\cap C')}\,.
\end{equation}
Here, $\widetilde{F}^{X',L}$ is the $\delta_1$-degree $L$
truncation of $\widetilde{F}^{X'}$.
For each ${\xi'\in C\cap C'}$, 
$\widetilde{F}^{X',L}_{\xi'}$ is the $\delta_1$-degree $L$
truncation of
$\log(\ZZ_{P,\xi'}^L(X';q))$, the logarithm of the
truncated stable pairs partition functions of moduli component
contributions over
$\xi'$. Finally, $\widetilde{F}^{X',L}_{C'\setminus (C\cap C')}$
is the $\delta_1$-degree $L$
truncation of
$\log(\ZZ_{P,C'\setminus(C\cap C')}^L(X';q))$, the logarithm of the
truncated stable pairs partition functions of 
contributions over
$C'\setminus (C\cap\ C')$.

\begin{Lemma} For $\xi'\in C\cap C'$, we have \label{w990}
$$\widetilde{F}^{X',L}_{\xi'} =
\sum_{\beta\in \text{\em Eff}_{\xi'}(\delta_1,L)} 
\sum_{n\in \mathbb{Z}} 
q^n v_1^{e^\beta_1}v_2^{e^\beta_2}\ 
\widetilde{R}_{n, m_\beta,h_\beta}
\cdot \text{\em Cont}\big(X'_{\xi'},\beta, NL^{\pi'}_{m_\beta,h_\beta, (e^\beta_1,e^\beta_2)}\big)\,.$$
\end{Lemma}

\noindent Here, $m_\beta$ denotes the  divisibility of $\beta$, and
$$2h_\beta-2 =\langle \beta,\beta \rangle\,,\ \  e_1^\beta =
\langle \delta_1,\beta \rangle\,, \ \ 
e_2^\beta=\langle \delta_2,\beta \rangle\,.$$

\begin{proof}
By Lemma \ref{hkkh}, every class ${\beta\in \text{Eff}_{\xi'}(\delta_1,L)}$
is completely isolated with respect to the family $\pi'$.
Hence, we may apply Proposition \ref{crt} to the contributions of
$\beta$ to the $\widehat{\Lambda}$-polarized family $\pi'$.
\end{proof}

\subsection{Nonisolated contributions: analysis of $C''$}
\label{nonan}
By the construction of $C''$ in Section \ref{ssmo},
$$C,C',C'' \subset Y \subset \mathcal{M}_{\widehat{\Lambda}}$$
where $Y$ is the nonsingular projective surface \eqref{rt2222} not contained
in any of the Noether-Lefschetz divisors listed in 
condition (2). Both $C'$ and $C''$ were constructed as
 very ample divisors on $Y$
which also do not lie in the Noether-Lefschetz loci listed in (2).
Since $$[C]+[C']=[C'']\in\text{Pic}(Y)\,,$$
 we have
$$\langle [C], [C] \rangle_Y + \langle [C], [C']\rangle_Y = 
\langle [C],[C'']\rangle_Y >0\,.$$
Let $r=\langle [C],[C'']\rangle_Y$, and let $\zeta_1,\ldots,\zeta_r$ be the
$r$ distinct intersection points of $C$ with $C''$.
We can choose $C''$ so the $\zeta_i$ 
are very general points of $C$. In particular
\begin{equation}\label{zqq7}
\zeta_i \notin C\cap C'\,.
\end{equation}
The degeneration of $C''$ to $C\cup C'$ occurs in the pencil on $Y$
spanned by $C''$ and $C\cup C'$.

Let $D_{m,h,(e_1,e_2)}$ occur in the finite list of Noether-Lefschetz divisors
given in condition (2) of Section \ref{ssmo}.
Since $Y$ does {\em not} lie in $D_{m,h,(e_1,e_2)}$, the intersection
$$D_{m,h,(e_1,e_2)} \cap Y \subset Y$$
is a proper divisor. We write
$$D_{m,h,(e_1,e_2)}\cap Y = w[C] + \sum_{j} w_j [T_j]\,, \ \ \  w,w_j\geq 0$$
where $w$ is the multiplicity of $C$ and the $T_j\subset Y$
are curves not containing $C$.
By the genericity hypotheses,
\begin{equation}\label{rrff}
C \cap C' \cap T_j = C \cap C'' \cap T_j = \emptyset\,.
\end{equation}

\begin{Lemma}\label{hkkh2}
There are no curves of $X''$ in class
 $\beta \in H_2(X'',\mathbb{Z})^{\pi''}$
of $\delta_1$-degree bounded by $L$ which move in
families
dominating $C''$.
\end{Lemma} 

\begin{proof}
If a such a family of curves were to dominate $C''$, then 
every fiber $X''_{\xi''}$ as $\xi''$ varies in  $C''$ would contain an 
effective curve class
of the family. In particular, the fibers over $\zeta\in C \cap C''$
would contain such effective curves.
By construction,
$\zeta\in C \cap C''$ is a very general point of $C$. Therefore,
$$\text{Pic}(X''_{\zeta})=\text{Pic}(X_{\zeta}) \cong \Lambda\,.$$
Since $C''$ was chosen {\em not} to lie in any Noether-Lefschetz
divisors associated to effective curves on $X_{\zeta}$ in $\Lambda$ of 
$\delta_1$-degree bounded by $L$, we have a contradiction.
\end{proof}

We now consider the family of nonsingular curves $C''_t$ in the pencil as $C''$
degenerates to $C\cup C'$. Here $t$ varies in $\Delta$, the base of the
pencil. The  total space of the pencil is
$$\mathcal{C}'' \rightarrow \Delta$$
with special fiber 
$$C_0''=C\cup C'\ \ \text{over}\ \ \ 0\in \Delta\,.$$
For fixed $t\neq 0$,
the stable pairs theory in $\delta_1$-degree bounded by $L$ 
for each $C''_{t}$ can be separated, by Lemma \ref{hkkh2},
into contributions over isolated points of $C''_t$.
The union of these support points for the stable pairs theory
of $C''_t$ defines an algebraic curve 
$$\text{Supp}\subset \mathcal{C}''\setminus C''_0 \rightarrow \Delta_{t\neq 0}\ $$
in the total space of
the pencil. Precisely, $\text{Supp}$ equals the set
$$\{ (t,p) \ | \ t\neq 0,\  p\in C''_t,\ 
X''_{(t,p)} {\text{carries an effective curve of degree}} \leq L \} \,.$$

\begin{Lemma}\label{g999}
The closure $\overline{\text{\em Supp}}\subset \mathcal{C}''$  contains no
components which 
intersect the special fiber $C''_0=C\cup C'$ in $C\cap C'$.
\end{Lemma}

\begin{proof}
If a component $Z\subset \overline{\text{Supp}}$ 
meets $\xi'\in C\cap C'$, then there
must be an effective curve class $\beta\in \text{Pic}(X_{\xi'})$
which remains effective (and algebraic) on all of $Z$.
By construction,
$$\text{Pic}(X_{\xi'}) =\Lambda\,.$$
Hence, $Z$ must be contained in the Noether-Lefschetz
divisor at $\xi'$ corresponding to $\beta$.
The latter Noether-Lefschetz divisor is on the list 
specified in condition (2) of Section \ref{ssmo} and hence
takes the form
$$D_{m,h,(e_1,e_2)}=w[C] + \sum_{j} w_j [T_j]\,.$$
The intersection of $C''_{t\neq 0}$ with $C$ is
always $\{\zeta_1,\ldots, \zeta_r\}$ since $C''_t$ is
a pencil. By condition \eqref{zqq7}, $\xi'$ is not a limit.
The intersection of $C''_{t\neq 0}$ with $T_j$ can not 
have limit $\xi'$ by \eqref{rrff}.
\end{proof}

The subvariety $\overline{\text{Supp}} \subset \mathcal{C}''$
is proper (by Lemma \ref{g999}) and 
therefore consists of finitely many curves and points.
Let $$\text{Supp}_1 \subset \overline{\text{Supp}}$$
denote the union of the 1-dimensional components of 
$\overline{\text{Supp}}$.
By Lemmas \ref{hkkh2} and \ref{g999}
no such component lies in the fibers of the pencil
$$\mathcal{C}'' \rightarrow \Delta\,.$$
Hence, after a base change $\epsilon:\widetilde{\Delta} \rightarrow \Delta$
possibly ramified over $0\in \Delta$, the pull-back of 
$\widetilde{\text{Supp}}_1$ to the
pulled-back family
$$\widetilde{\mathcal C}'' \rightarrow \widetilde{\Delta}$$
is a {\em union of sections}.
We divide the sections into two types
$$\widetilde{\text{Supp}}_1 = \bigcup_i A_i \ \cup \ \bigcup_j B_j $$
by the values of the sections over $0\in \widetilde{\Delta}$:
$$A_i(0)\in C \subset C''_0\,,  \ \ \ B_j(0)\in C' \subset C''_0 \,.$$
By Lemma \ref{hkkh2}, no section meets
$C \cap C'$ over $0\in \widetilde{\Delta}$.

The stable pairs partition functions of every 3-fold
$$X''_t \rightarrow \widetilde{C}''_t\,, \ \ \ t\neq 0$$
factors into into contributions over the sections $A_i(t)$
and $B_j(t)$,
\begin{equation}\label{gree}
\mathsf{Z}^L_{P}(X''_t,q)=
\left[\prod_{i} \mathsf{Z}^L_{P,A_i(t)}(X''_t,q)\ \cdot \
\prod_{j} \mathsf{Z}^L_{P,B_j(t)}(X''_t,q)\right]_{\leq L}\,.
\end{equation}
As before, the partition functions are all $\delta_1$-degree $L$
truncations.
The finitely many points
of ${\text{Supp}}\setminus \text{Supp}_1$
are easily seen not to contribute to \eqref{gree} by deformation
invariance of the virtual class.
We can take the logarithm,
$$\widetilde{F}^{X''_t,L} =
\sum_{i} \widetilde{F}_{A_i(t)}^{X''_t,L} \ + \
\sum_{j} \widetilde{F}_{B_j(t)}^{X''_t,L}\ .$$
Since the sections $B_j$ all meet $C'\setminus C\cap C'$
over $0\in \widetilde{\Delta}$, the moduli space
of stable pairs supported over
$$\widetilde{\mathcal C}'' \setminus \left(
C\  \cup \ \bigcup_i A_i\right)$$
is proper. Again, by deformation invariance of the
virtual class,
$$\sum_{j} \widetilde{F}_{B_j(t)}^{X''_t,L} = 
\widetilde{F}^{X',L}_{C'\setminus C\cap C'}\,.$$ 
Then relations \eqref{xfft} and \eqref{xfft2} imply
\begin{equation}\label{lzzs}
\widetilde{F}^{X,L}=
-\sum_{\xi'\in C\cap C'} \widetilde{F}^{X',L}_{\xi'}\ + \
\sum_{i} \widetilde{F}_{A_i(t)}^{X''_t,L}\ 
\end{equation}
for every $0\neq t\in \widetilde{\Delta}$.

\begin{Lemma}\label{nnn17}
For general $t\in \widetilde{\Delta}$, the section
$A_i(t)$ does {\em not} 
lie in 
the nodal Noether-Lefschetz locus in $\widetilde{C}_t''$. 
\end{Lemma}

\begin{proof}
Suppose $A_i(t)$ is always contained in the nodal
Noether-Lefschetz locus in $\widetilde{C}_t''$.
We can assume by construction that $Y$ meets the
nodal Noether-Lefschetz divisor $D_{1,0,(0,0)}$
at very general points of the latter divisor in 
$\mathcal{M}_{\widehat{\Lambda}}$. Hence, for very general $t$,
$X''_{t,A_i(t)}$ is a nodal $K3$ surface with $K3$ resolution
$$\rho:\widetilde{X}''_{t,A_i(t)} \To X''_{t,A_i(t)}$$
with Picard lattice{\footnote{The lattice
 $\widehat{\Lambda} \oplus \mathbb{Z}[\widetilde{E}]$
is primitive in $\text{Pic}(\widetilde{X}''_{t,A_i(t)})$
since $\widehat{\Lambda}\subset \Lambda$ is primitive and
\eqref{mmqq33} holds. The isomorphism \eqref{lqlq} is then immediate
at a very general point of the nodal Noether-Lefschetz divisor of
$\mathcal{M}_{\widehat{\Lambda}}$.}}
\begin{equation}
\label{lqlq}
\text{Pic}(\widetilde{X}''_{t,A_i(t)})\cong 
\widehat{\Lambda}\oplus \mathbb{Z}[\widetilde{E}]
\end{equation}
where $E$ is the exceptional $-2$-curve.

By the definition of $A_i(t)$, there must exist
an effective curve $$Q\subset {X}''_{t,A_i(t)}$$
of $\delta_1$-degree bounded by $L$ on $X''_{t,A_i(t)}$.
Since $X''_{t,A_i(t)}$ is nodal, $Q$ may not be a Cartier divisor.
However, $2Q$ is Cartier. By pulling-back via $\rho$, using the
identification of the Picard lattice \eqref{lqlq},
 and pushing-forward by $\rho$,
\begin{equation}\label{clld}
2Q\in \widehat{\Lambda}\,.
\end{equation}

The effective curve $Q$ moves with the $K3$ surfaces
$X''_{t,A_i(t)}$ as $t$ goes to $0$. The
condition \eqref{clld} holds for very general $t$ and hence 
for {\em all} $t$.
In particular the condition \eqref{clld} hold at $t=0$.
The $K3$ surface 
$X''_{0,A_i(0)}$ is a nodal fiber of $$X \rightarrow C\,.$$
The existence of an effective curve $Q\in\frac{1}{2}\widehat{\Lambda}$
directly contradicts condition (iii) of Section \ref{lsel} 
of $\widehat{\Lambda}$.
\end{proof}

Let $\widetilde{C}''_t$ be a general curve of the pencil $\widetilde{\Delta}$.
By Lemma \ref{nnn17}, the sections $A_i(t)\in \widetilde{C}_t''$ are
disjoint from the nodal Noether-Lefschetz divisor on $\widetilde{C}_t''$.
We would like to apply the contribution relation of 
Proposition \ref{crt} to the fiber of 
$$X''_t \rightarrow \widetilde{C}_t''\ $$
over $A_i(t)$.
We therefore need the following result.

\begin{Lemma}
For general $t\in \widetilde{\Delta}$, the curve
$\widetilde{C}_t''$
does {\em not} 
lie in 
the Noether-Lefschetz divisor associated
to any effective curve
$$\beta \in \text{\em Pic}(X''_{t,A_i(t)})$$
of $\delta_1$-degree bounded by $L$.
\end{Lemma}

\begin{proof}
If the assertion of the Lemma were false, there
would exist a moving family of effective curves
$$\beta \in \text{Pic}(X''_{t,A_i(t)})$$
such that $\widetilde{C}''_t$ always lies in
the Noether-Lefschetz divisor corresponding to $\beta$.
Then all of $Y$ would lie in the Noether-Lefschetz divisor
corresponding to $\beta$. The limit of $\beta$ is 
an effective class in the $K3$ fiber over 
$A_i(0)\in C$ of $\delta_1$-degree bounded by $L$.
Hence, the Noether-Lefschetz divisor corresponding to $\beta$
is listed in condition (2) of Section \ref{ssmo}. But
$C''\subset Y$ was constructed to not lie in any of the Noether-Lefschetz
divisors of the list (2), a contradiction.
\end{proof}

We may now apply
Proposition \ref{crt} to the fiber of 
$$X''_t \rightarrow \widetilde{C}_t''\ $$
over $A_i(t)$ for a
general curve $\widetilde{C}''_t$ of the pencil $\widetilde{\Delta}$.
Just as in Lemma \ref{w990}, we obtain 
\begin{multline} \label{w991}
\widetilde{F}^{X''_t,L}_{A_i(t)} = \\
\sum_{\beta\in \text{Eff}_{A_i(t)}(\delta_1,L)} 
\sum_{n\in \mathbb{Z}} 
q^n v_1^{e^\beta_1}v_2^{e^\beta_2}\ 
\widetilde{R}_{n, m_\beta,h_\beta}
\cdot \text{Cont}\Big(X''_{A_i(t)},\beta, 
NL^{\pi''}_{m_\beta,h_\beta, (e^\beta_1,e^\beta_2)}\Big).
\end{multline}

\subsection{Proof of Theorem \ref{ffc2}}

We now complete the proof of Theorem \ref{ffc2} by proving the relation
\begin{equation}\label{cq299}
\widetilde{N}_{n,(d_1,d_2,d_3)}^\XX= \sum_{h} \sum_{m=1}^{\infty}
\widetilde{R}_{n,m,h}\cdot  NL_{m,h,(d_1,d_2, d_3)}^{\pi_3}
\end{equation}
for degrees $(d_1,d_2,d_3)$ positive with respect
to the quasi-polarization in $\Lambda$.
The Noether-Lefschetz divisors lie in the moduli space
$\mathcal{M}_{\Lambda}$.

The degrees $(d_1,d_2,d_3)$
of a class in $H_2(\XX,\mathbb{Z})^{\pi_3}$ 
with respect to $\Lambda$ determine the degrees $(e_1,e_2)$
with respect to $\widehat{\Lambda}$.
We first show relation \eqref{cq299} is equivalent 
for classes of $\delta_1$-degree bounded by $L$
to the relation
\begin{equation}\label{cq2999}
\widetilde{N}_{n,(e_1,e_2)}^\XX= \sum_{h} \sum_{m=1}^{\infty}
\widetilde{R}_{n,m,h}\cdot  NL_{m,h,(e_1,e_2)}^{\pi_3} \,.
\end{equation}
The Noether-Lefschetz theory in \eqref{cq2999} occurs
 in the moduli space
$\mathcal{M}_{\widehat{\Lambda}}$.

The equivalence of \eqref{cq299} and \eqref{cq2999} for
classes of $\delta_1$-degree bounded by $L$ is a consequence of
 condition (4) in Section \ref{lsel} for $\widehat{\Lambda}\subset
\Lambda$.
For effective
classes of $\delta_1$-degree bounded by $L$, condition
(4) says the $(e_1,e_2)$ degrees
{\em determine} the $(d_1,d_2,d_3)$ degrees.
The left sides of  \eqref{cq299} and \eqref{cq2999}
then match since the stable pairs invariants
only involve effective classes. As shown 
in Section \ref{stopr}, only effective classes
contribute to the right sides as well. So the right sides of
\eqref{cq299} and \eqref{cq2999} also match.

We prove \eqref{cq2999} by the result obtains in Sections \ref{nonano}
and \ref{nonan}.
By equations \eqref{lzzs} and \eqref{w991} and Lemma \ref{w990},
we have
\begin{multline}\label{masman}
\widetilde{F}^{X,L}=\\
-\sum_{\xi'\in C\cap C'}
\sum_{\beta\in \text{Eff}_{\xi'}(\delta_1,L)} 
\sum_{n\in \mathbb{Z}} 
q^n v_1^{e^\beta_1}v_2^{e^\beta_2}\ 
\widetilde{R}_{n, m_\beta,h_\beta}
\cdot \text{Cont}\big(X'_{\xi'},\beta, 
NL^{\pi'}_{m_\beta,h_\beta, (e^\beta_1,e^\beta_2)}\big) \\
+\sum_i
\sum_{\beta\in \text{Eff}_{\!A_i(t)\!}(\delta_1,L)} 
\sum_{n\in \mathbb{Z}} 
q^n v_1^{e^\beta_1}v_2^{e^\beta_2}\ 
\widetilde{R}_{n, m_\beta,h_\beta}
\cdot \text{Cont}\big(X''_{A_i(t)},\beta, 
NL^{\pi''}_{m_\beta,h_\beta, (e^\beta_1,e^\beta_2)}\big)\ 
\end{multline}
for a general $t\in \widetilde{\Delta}$.

By definition, the $q^nv_1^{e_1}v_2^{e_2}$ coefficient of the left 
side of \eqref{masman}
is $\widetilde{N}_{n,(e_1,e_2)}^\XX$. 
The $q^nv_1^{e_1}v_2^{e_2}$
coefficients of the right side of \eqref{masman} correspond to 
intersections with the Noether-Lefschetz divisors
$D_{m,h,(e_1,e_2)}$. As in Section \ref{nonan}, we write
\begin{equation} \label{yy23}
D_{m,h,(e_1,e_2)}\cap Y = w[C] + \sum_{j} w_j [T_j]\,, \ \ \  w,w_j\geq 0\,,
\end{equation}
where $w$ is the multiplicity of $C$ and the $T_j\subset Y$
are curves not containing $C$.

The first sum on the right side of \eqref{masman} concerns
$C\cap C'$. The 
contribution of $D_{m,h,(e_1,e_2)}$ to the 
$q^nv_1^{e_1}v_2^{e_2}$ coefficient
of the first sum is 
\begin{equation}\label{pppzzz}
-w \ \langle [C],[C'] \rangle_Y \ \widetilde{R}_{n, m,h}
\end{equation}
if all the instances of $\beta\in \text{Pic}(X'_{\xi'})$
associated to the Noether-Lefschetz divisor $D_{m,h,(e_1,e_2)}$
are effective. As we have seen in Section \ref{stopr}, if {\em any} 
such instance of 
 $\beta$ is not effective, then $\widetilde{R}_{n, m,h}=0$
and the entire Noether-Lefschetz divisor $D_{m,h,(e_1,e_2)}$
contributes 0 to
the right sides of both \eqref{cq2999} and \eqref{masman}.

Next, we study the intersection of $D_{m,h,(e_1,e_2)}$ with
$\widetilde{C}''_t$. The intersection with $C$,
$$C\cap \widetilde{C}''_t=\{\zeta_1,\ldots,\zeta_r\}\,, \ \ \
r=\langle C, C'' \rangle_Y\,, $$
is independent of $t$ by the construction of the
pencil. The contribution to 
the $q^nv_1^{e_1}v_2^{e_2}$ coefficient
of the full intersection $C\cap \widetilde{C}''_t$ is
\begin{equation}\label{pqqzz}
w \ \langle [C],[C''] \rangle_Y \ \widetilde{R}_{n, m,h}
\end{equation}
if all the instances of 
$$\beta\in \text{Pic}(X''_{t,\zeta_k})=\text{Pic}(X_{\zeta_k})$$
associated to the Noether-Lefschetz divisor $D_{m,h,(e_1,e_2)}$
are effective. Such an effective $\beta$ implies
the point $\zeta_k\in \widetilde{C}''_s$ is always in
$\widetilde{\text{Supp}}$ and hence corresponds to a
section $A_i$. In the effective case, the contributions
\eqref{pqqzz} all occur in the second sum of \eqref{masman}.

On the other hand, if {\em any} 
class $\beta\in \text{Pic}(X''_{t,\zeta_k})$, for any 
$\zeta_k\in C\cap C'$,
associated to the Noether-Lefschetz divisor $D_{m,h,(e_1,e_2)}$
is not effective, then $\widetilde{R}_{n, m,h}=0$
and the entire Noether-Lefschetz divisor $D_{m,h,(e_1,e_2)}$
contributes 0 to
the right sides of \eqref{cq2999} and \eqref{masman}.

Finally, we consider the
intersection of $T_j$ with $\widetilde{C}''_t$.
By construction,
 $T_j \cap \widetilde{C}''_t \subset \widetilde{C}''_t$
is a finite collection of points. 
We divide the intersection
\begin{equation}\label{tt99}
T_j \cap \widetilde{C}''_t = I_{j,t} \cup I'_{j,t} 
\end{equation}
into disjoint subset with the following properties:
\begin{enumerate}
\item[$\bullet$] as $t\rightarrow 0$, the points of $I_{j,t}$
                 have limit in $C$,
\item[$\bullet$] as $t\rightarrow 0$, the points of $I'_{j,t}$
                    have limit in $C'$.
\end{enumerate}
Since $T_j$ does not intersect $C\cap C'$, the disjoint 
union is well-defined and unique \eqref{tt99} 
for $t$ sufficiently near 0. Moreover, the
sum of the local intersection numbers of $T_j\cap \widetilde{C}''_t$
over $I_{j,t}$ is $\langle C, T_j \rangle_Y$.
The points of $I'_{j,t}$, related to the sections $B(t)$ in the
analysis of 
Section \ref{nonan}, do not play a role{\footnote{The contributions
of the sections $B(t)$ is cancelled in \eqref{lzzs}.}} 
in 
the analysis of
\eqref{masman}. 

If a single instance of a class
$$\beta\in \text{Pic}(X''_{t,\xi''}) \ \text{for} \ \xi''\in I_{j,t}$$
associated to the Noether-Lefschetz divisor $D_{m,h,(e_1,e_2)}$
is ineffective for $t$ sufficiently near 0, then $\widetilde{R}_{n, m,h}=0$
and the entire Noether-Lefschetz divisor $D_{m,h,(e_1,e_2)}$
contributes 0 to
the right sides of \eqref{cq2999} and \eqref{masman}.
Otherwise, every instance of such a $\beta$ is effective
for all $t$ sufficiently near 0. Then
$$I_{j,t} \subset \bigcup_i A_i(t)\,,$$
and  the contribution to 
the $q^nv_1^{e_1}v_2^{e_2}$ coefficient of the right side of \eqref{masman}
of the full intersection $I_{j,t}$ is
\begin{equation}\label{pqqzzz}
w_j \ \langle [C],T_j \rangle_Y \ \widetilde{R}_{n, m,h} \,.
\end{equation}

Summing all the contributions \eqref{pppzzz}, \eqref{pqqzz}, and \eqref{pqqzzz}
 to the right side of \eqref{masman}
associated to $D_{m,h,(e_1,e_2)}$ yields
\begin{equation} \label{ff44} 
\Big(-w  \langle [C],[C'] \rangle_Y  +w \ \langle [C],[C''] \rangle_Y  
 +\sum_j w_j \ \langle [C],T_j \rangle_Y\Big) \cdot \widetilde{R}_{n, m,h}\  \ 
\end{equation}
Using the relation $-[C']+[C'']=[C]$ and \eqref{yy23}, 
the sum \eqref{ff44} exactly matches contribution 
$$ \widetilde{R}_{n,m,h} \cdot NL_{m,h,(e_1,e_2)}^{\pi_3} =
\widetilde{R}_{n,m,h} \cdot
\int_C [D_{m,h,(e_1,e_2)}]\ $$
of $D_{m,h,(e_1,e_2)}$ to 
the right side of \eqref{cq2999}.
The proofs of \eqref{cq2999} and of 
Theorem \ref{ffc2} are complete.

\section{Katz-Klemm-Vafa conjecture}\label{kkvcon}
The proof of the P/NL correspondence of Theorem \ref{ffc2} was
the last step in the proof of Proposition \ref{gg12}:
\begin{equation*} 
r_{g,m,h} = \widetilde{r}_{g,m,h}\  \ \text{\em for all} \ \
 g\in \mathbb{Z}\,, \ \ m> 0\,,
\ \ h\in \mathbb{Z}\,.
\end{equation*}
The proof of Theorem \ref{vvee} is also now complete.

In Section \ref{spbps}, several properties of
the stable pairs invariants $\widetilde{r}_{g,m,h}$
were established (and in fact were used in the proofs of Theorem \ref{ffc2}
and Proposition \ref{gg12}).
The most important property of $\widetilde{r}_{g,m,h}$ is 
independence of divisibility established in
 Proposition \ref{fcc19},
$$\widetilde{r}_{g,\beta} \ {\text{\em depends only upon $g$
and $\langle\beta,\beta \rangle$}}\,.$$
Also proven in Section \ref{spbps} were the basic vanishing results
$$\widetilde{r}_{g<0,m,h}=0\,, \ \ \widetilde{r}_{g,m,h<0}=0\,.$$ 
The independence of $\widetilde{r}_{g,\beta}$ upon the
divisibility of $\beta$ reduces the Katz-Klemm-Vafa
conjecture to the primitive case.

The stable pairs BPS counts in the primitive case
are determined by Proposition \ref{kedd9}, relation \eqref{vxx99},
and the interpretation of the Kawai-Yoshioka results presented
in Section \ref{kye}. Taken together, we prove the Katz-Klemm-Vafa
conjecture in the primitive case and hence in all cases.
The proof of Theorem 1 is complete. 

Following the notation of Section \ref{gwnlcc}, let
$$\pi: X \rightarrow C$$
be a 1-parameter family of $\Lambda$-polarized $K3$
surfaces with respect to a rank $r$ lattice $\Lambda$.
Using the independence of divisibility, Theorem \ref{ffc} in Gromov-Witten 
theory 
takes a much simpler form.
\begin{Theorem} \label{ffc91}
 For degrees $(d_1,\dots,d_r )$ positive with respect to the
quasi-polariza\-tion $\lambda^\pi$,
$$n_{g,(d_1,\dots,d_r)}^X= \sum_{h=0}^\infty 
r_{g,h}\cdot  NL_{h,(d_1,\dots,d_r)\,}^\pi.$$
\end{Theorem}
Theorems 1 and \ref{ffc91} together give 
closed form solutions for
the BPS states in fiber classes in term
of the Noether-Lefschetz numbers (which are
expressed in terms of modular forms by Borcherds'
results). A classical example
is given in the next Section.

\section{Quartic $K3$ surfaces}\label{qk3}
We provide a complete calculation of
the Noether-Lefschetz numbers 
and BPS counts in fiber classes for
the family of
$K3$ surfaces  determined by a Lefschetz pencil of
quartics in $\proj^3$:
$$\pi: X \rightarrow \proj^1,\ \  \ \ \ X \subset \proj^3 \times \proj^1
 \ \ \text{of type $(4,1)$} . $$

Let $A$ and $B$ be modular forms of weight $1/2$ and level 8,
$$A= \sum_{n\in \mathbb{Z}} q^{\frac{n^2}{8}}\,, \ \ \
B= \sum_{n\in \mathbb{Z}} (-1)^n  q^{\frac{n^2}{8}}\,.$$
Let $\Theta$ be the modular form of weight $21/2$ and level 8 
defined by
\begin{eqnarray*}
2^{22} \Theta &=&\ \ 3A^{21}-81 A^{19}B^2 -627 A^{18}B^3 -14436 A^{17}B^4 \\
& & -20007 A^{16}B^5  -169092 A^{15}B^6 -120636 A^{14}B^7\\
& &   -621558 A^{13}B^8
-292796 A^{12}B^9 -1038366 A^{11}B^{10}\\
& &  -346122 A^{10}B^{11}
-878388 A^{9} B^{12} -207186 A^8 B^{13}\\
& &  -361908 A^7 B^{14} -56364 A^6 B^{15} -60021 A^5 B^{16}\\
& &  -4812 A^4 B^{17}
-1881 A^3 B^{18} -27 A^2 B^{19}+ B^{21}\,.
\end{eqnarray*}
We can expand $\Theta$ as a series in $q^{\frac{1}{8}}$,
$$ \Theta = -1 +108 q +320 q^{\frac{9}{8}}+50016 q^{\frac{3}{2}}+ 76950q^2 \ldots .$$
Let $\Theta[m]$ denote the coefficient of $q^m$ in $\Theta$.

The modular form $\Theta$  first appeared in calculations of \cite{germans}.
The following result was proven in \cite{gwnl}:
{\em the Noether-Lefschetz numbers of the quartic pencil
$\pi$ are coefficients of $\Theta$,
$$NL^{\pi}_{h,d} = \Theta\left[ \frac{\bigtriangleup_4(h,d)}{8}\right]\,,$$
where the discriminant is defined by
$$\bigtriangleup_4(h,d)= -\det
 \left( \begin{array}{cc}
4 & d  \\
d & 2h-2  \end{array} \right) =d^2- 8h+8\,.$$
}
By Theorem \ref{ffc91}, we obtain
$$n_{g,d}^X = \sum_{h=0}^\infty r_{g,h}\cdot \Theta\left[ \frac{\bigtriangleup_4(h,d)}{8}\right]\,, $$
as predicted in \cite{germans}.
Similar closed form solutions can be found for all the
classical families of $K3$-fibrations, see \cite{gwnl}.

\appendix
\section{Invariants}
We include here a short table of the various invariants
associated to a $K3$ surface $S$ and a class $\beta \in \text{Pic}(S)$.

\vspace{18pt}
{\footnotesize
\begin{table}[ht]
\begin{tabular}{|c|c|c|}
        \hline & & \\
$R_{g,\beta}(S)$ &   {\text Reduced GW invariants of $S$} & Section 0.1  \\
& & \\
\hline & & \\
$r_{g,\beta}$ & {\text BPS counts for $K3$ surfaces in GW theory} 
& Section 0.3 \\ & & \\
\hline & & \\
$\widetilde{R}_{n,\beta}(S)$ &   {\text Stable pair invariants of $S$
parallel to $R_{g,\beta}$} & Sections 0.6, 6.2 \\ & & \\ \hline & & \\
$\widetilde{r}_{n,\beta}$ & {\text BPS counts for $K3$ surfaces via stable
pairs} 
& Sections 3.4, 7.6 \\ & & \\ 
\hline
 & & \\
$R_{n,\beta}^{\red}(S\times\mathcal R)$ & {\text Reduced stable pair invariants 
of the rubber} 
& Section 6.6 \\ & & \\ \hline
 & & \\

$\Big\langle 1 \Big\rangle_{Y,\beta}^{\red}$ & {\text Reduced stable
pair residues of $Y=S\times \C$} 
& Section 5.6 \\ & & \\ \hline
 & & \\

$I_h$ & {\text {} $I_h=\Big\langle 1 \Big\rangle_{Y,\alpha}^{\red}$
for $\alpha$ primitive with 
$\langle \alpha,\alpha \rangle=2h-2$} 
& Section 5.6 \\ & & \\ \hline
\end{tabular}
\end{table}
}
\vspace{18pt}

Associated to $K3$ fibrations over a curve
$$X \To C\,,$$
there are several more invariants. Here, $\beta\in H_2(X,\mathbb{Z})$
is a fiber class.

{\footnotesize
\begin{table}[ht]
\begin{tabular}{|c|c|c|}
        \hline & & \\
$N_{g,\beta}^X$ & {\text Connected GW invariants of the
$K3$-fibration $X$} 
& Section 2.2 \\ & & \\ \hline
 & & \\
${n}_{g,\beta}^X$ & {\text BPS counts for $X$ in  GW theory} 
& Section 2.2 \\ & & \\ \hline
& & \\
$\widetilde{N}_{n,\beta}^X$ & {\text Connected stable pairs invariants of the
$K3$-fibration $X$} 
& Section 8.1 \\ & & \\ \hline
 & & \\
$\widetilde{n}_{g,\beta}^X$ & {\text BPS counts for $X$ via  stable pairs} 
& Sections 3.5, 8.1 \\ & & \\ \hline

\end{tabular}
\end{table}}
\vspace{18pt}

When $S$ is a nonsingular $K3$ fiber of $X\rightarrow C$ and
$\beta \in \text{Pic}(S)$ is a class for which no effective
summand on $S$ deforms over $C$, we have two invariants.

\vspace{18pt}
\begin{table}[ht]
\begin{tabular}{|c|c|c|}
        \hline & & \\
$P^\star_{n,\beta}(X)$ & {\text Contribution of stable pairs
supported on $S$} 
& Section 6.2 \\ & {\text to the stable pairs invariant of $X$} 
& \\ \hline
 & & \\
$P^\star_{n,\beta}(X/S)$ & {
Contribution of stable pairs
over $S$} 
& Section 6.5 \\ & 
{\text to the relative stable pairs invariant of $X/S$}
& \\ \hline

\end{tabular}
\end{table}

\section{Degenerations}

Let $\widetilde{\proj^2 \times \proj^1}$ be the blow-up of
$\proj^2\times \proj^1$ at a point. Consider the toric 4-fold
 $$\mathsf{Y}=\widetilde{\proj^2 \times \proj^1} \times \proj^1 $$
of Picard group of rank 4,
$$\text{Pic}(\mathsf{Y})\cong
\mathbb{Z} L_1 \oplus \mathbb{Z} L_2 \oplus \mathbb{Z}E
\oplus \mathbb{Z}L_3
\,.$$
Here, $L_1,L_2,E$ are the pull-backs of divisors{\footnote{We follow the
notation of Section \ref{tttt}.}} from 
 $\widetilde{\proj^2 \times \proj^1}$
 and $L_3$ is the pull-back of $\O(1)$ from the
last $\proj^1$.
The divisors $L_1$, $L_2$, and $L_3$ are certainly base point
free on $\mathsf{Y}$. Since
$L_1+L_2-E$ arises on $\widetilde{\proj^2 \times \proj^1}$
via the projection from a point of the $(1,1)$-Segre
embedding
$$\proj^2 \times \proj^1 \hookrightarrow \proj^5\,,$$
the divisor $L_1+L_2-E$ determines
a map to the quadric $Q\subset \proj^4$,
$$\widetilde{\proj^2 \times \proj^1} \rightarrow Q \subset \proj^4\,.$$
Hence, $L_1+L_2-E$ is base point free on both 
 $\widetilde{\proj^2 \times \proj^1}$ and $\mathsf{Y}$.

The anticanonical series $3L_1+2L_2-2E+2L_3$
is base point free on $\mathsf{Y}$ since $L_1$, $L_2$,
$L_3$, and $L_1+L_2-E$
are all base point free.
Let
$$\XX \subset \mathsf{Y}$$
be a general anticanonical divisor (nonsingular by Bertini).
In \cite{PP}, the Gromov-Witten/Pairs correspondence is
proven for Calabi-Yau 3-fold which admit appropriate degenerations.
To find such degenerations for $\XX$, we simply factor equations.

Let $\XX_{a,b,c,d}\subset \mathsf{Y}$ denote a general divisor
of class $aL_1+bL_2+cE+dL_3$. We first degenerate
$\XX=\XX_{3,2,-2,2}$ via the product
$$\XX_{2,1,-1,1} \cdot \XX_{1,1,-1,1}\,.$$
For such a degeneration to be used in the scheme of \cite{PP},
all of the following varieties must be nonsingular:
$$\XX_{3,2,-2,2}\,, \ \ \XX_{2,1,-1,1}\,, \ \ \XX_{1,1,-1,1}\,,$$
$$\XX_{2,1,-1,1}\cap \XX_{1,1,-1,1}\,, \ \ 
\XX_{3,2,-2,2}\cap
\XX_{2,1,-1,1}\cap \XX_{1,1,-1,1}\,.$$
Since all three divisor classes 
$\XX_{3,2,-2,2}$, $\XX_{2,1,-1,1}$, $\XX_{1,1,-1,1}$
are base point free, the required nonsingularity follows from Bertini.
Next, we degenerate $\XX_{2,1,-1,1}$ via the product
$$\XX_{1,1,-1,1} \cdot \XX_{1,0,0,0}\,.$$
The nonsingularity of the various intersections is again
immediate by Bertini. Since $\XX_{1,0,0,0}$ is a toric 3-fold,
no further action must be taken for $\XX_{1,0,0,0}$.

We are left with the divisor 
$\XX_{1,1,-1,1}$ which we degenerate via 
the product
$$\XX_{1,0,0,1} \cdot \XX_{0,1,-1,0}\,.$$
While the divisor classes of $X_{1,1,-1,1}$ and $\XX_{1,0,0,1}$ 
are base point free, the class $\XX_{0,1,-1,0}$ is not.
There is unique effective divisor
$$\XX_{0,1,-1,0}\subset \mathsf{Y}\,.$$
Fortunately $\XX_{0,1,-1,0}$ is a nonsingular toric 3-fold
isomorphic to $\widetilde{\proj}^2 \times \proj^1$ where
$\widetilde{\proj}^2$ is the blow-up of $\proj^2$ in a point.
The nonsingularity of $\XX_{0,1,-1,0}$ is sufficient
to guarantee the nonsingularity of 
$$
\XX_{1,0,0,1}\cap \XX_{0,1,-1,0}\,, \ \ 
\XX_{1,1,-1,1}\cap
\XX_{1,0,0,1}\cap \XX_{0,1,-1,0}\ $$
since $\XX_{1,1,-1,1}$ and $\XX_{1,0,0,1}$ 
are both base point free.

The result of \cite{PP} reduces the GW/P correspondence
for $\XX$ to the toric cases
$$\XX_{1,0,0,0}\,,\ \XX_{0,1,-1,0}\,, \ \XX_{0,0,0,1}$$
and the geometries of the various 
$K3$ and rational surfaces and higher genus curves
which occur as intersections in the degenerations.
The GW/P correspondences for all these end states have been
established in
\cite{PP} and earlier work.
Hence, the GW/P correspondence holds for $\XX$.

\section{Cones and virtual classes}

\subsection{Fulton Chern class}\label{ap11}
Let $X$ be a scheme{\footnote{The constructions are
also valid for a Deligne-Mumford stack which admits embeddings
into nonsingular Deligne-Mumford stacks.}} of dimension $d$. Let
$$X \subset M$$
be a closed embedding in a nonsingular ambient $M$ of dimension $m\geq d$.
Of course, we also have an embedding
$$X \subset M \times \C=\widetilde{M}$$
where $X$ lies over $0\in \C$.
The normal cones $C_{X}M$ and
$C_X\widetilde{M}$ of $X$ in $M$ and $\widetilde{M}$
are of pure dimensions $m$ and $m+1$ respectively. Moreover,
$$C_{X}\widetilde{M} = C_{X}M \oplus 1\,,$$
following the notation of \cite{Fu}.
Let $q$ be the structure morphism of the projective cone
$$q:\PP(C_{X}\widetilde{M}) \rightarrow X\,,$$
and let $[\PP(C_{X}\widetilde{M})]$ be the fundamental class of pure dimension $m$.
The Segre class $s(X,M)$ is defined by
\begin{eqnarray*}
s(X,M) &=& q_*\left( \sum_{i=0}^\infty c_1(\mathcal{O}(1))^i \cap 
[\PP(C_{X}M \oplus 
1)]\right) \\
& =& q_*\left( \sum_{i=0}^\infty c_1(\mathcal{O}(1))^i \cap 
[\PP(C_{X}\widetilde{M})]\right).
\end{eqnarray*}
The Fulton total Chern class, 
\begin{eqnarray}
c_F(X) &=& c(T_{M|X}) \cap\ s(X,M) \label{ghat}\\
& = & c(T_{\widetilde{M}|X}) \nonumber
\cap\ q_*\left( \sum_{i=0}^\infty c_1(\mathcal{O}(1))^i \cap 
[\PP(C_{X}\widetilde{M})]\right),
\end{eqnarray}
is independent of the embedding $M$, see  \cite[4.2.6]{Fu}.

Let $E\udot=[E^{-1}\rightarrow E^0]$, together with a morphism to 
the cotangent complex $L\udot$,
be a perfect obstruction theory
on $X$. The virtual class associated to $E\udot$ can be expressed
in terms of Chern classes of $E\udot$ and the Fulton total Chern class
of $X$:
\begin{eqnarray*}
[X]^{\vir} &=&\left[s\big((E\udot)^\vee\big) \cap c_F(X)\right]_{\text{virdim}} \\
& = & \left[ \frac{c(E_1)}{c(E_0)} \cap c_F(X)\right]_{\text{virdim}}\,,
\end{eqnarray*}
where $E_1=(E^{-1})^*$ and $E_0=(E^0)^*$.
The above formula occurs in \cite{Si} and earlier in the excess
intersection theory of \cite{Fu}.
As a consequence, the virtual class depends only upon the
$K$-theory class of $E\udot$.

\subsection{The curvilinear condition}\label{ap22}
Let $Y\subset X$ be a subscheme satisfying the curvilinear lifting
property:
$${\text{\em every map $\,\,  \Spec\C[x]/(x^k)\to X$ factors through $Y$.}}$$
By the $k=1$ case, the curvilinear lifting property implies 
$Y\subset X$ is a bijection on closed points.

We view the embedding $X\subset M\times \C=\widetilde{M}$ also as an embedding
of
$$Y\subset X \subset \widetilde{M} .$$
Let $I_X\subset I_Y$ be the ideal sheaves of $X$ and $Y$ in 
$\widetilde{M}$.
There is a canonical rational map over $\widetilde{M}$,
$$f: \Proj\left(\oplus_{i=0}^\infty I_Y^i\right) \  - - \rightarrow\
\Proj\left(\oplus_{i=0}^\infty I_X\right)$$
associated to the morphism of graded algebras
$$\oplus_{i=0}^\infty I_X^i \rightarrow \oplus_{i=0}^\infty I_Y^i\,, \ \ \
I_X^0=I_Y^0=\mathcal{O}_{\widetilde{M}}\,.$$
By definition, the source and target of $f$ are the
blow-ups of $\widetilde{M}$ along $Y$ and $X$ respectively,
$$\Bl_Y(\widetilde{M})= \Proj\left(\oplus_{i=0}^\infty I_Y^i\right)
\stackrel{\pi_Y}{\longrightarrow} 
\widetilde{M}\,, $$
$$
\Bl_X(\widetilde{M}) =\Proj\left(\oplus_{i=0}^\infty I_X^i\right)
\stackrel{\pi_Y}{\longrightarrow} \widetilde{M}\,,$$
with exceptional divisors
$$\pi_Y^{-1}(Y)= \PP(C_Y \widetilde{M})\,, \ \ \
\pi_X^{-1}(X)= \PP(C_X \widetilde{M})\,. $$
\begin{Proposition} \label{bababa}
The rational map has empty base locus and thus
yields a projective morphism
$$f: \Bl_Y(\widetilde{M}) \To \Bl_X(\widetilde{M})\,.$$
Moreover, as Cartier divisors on $\Bl_Y(\widetilde{M})$ ,
$$f^*( \PP(C_X \widetilde{M})) = \PP(C_Y \widetilde{M})\,.$$
\end{Proposition}

\begin{proof}
Away from the exception divisor $\pi_Y^{-1}(Y)$, $f$ is certainly a morphism. 
We need only study the base locus on the exceptional divisor
$$ \PP(C_Y\widetilde{M})\subset \Bl_Y \widetilde{M} \,.$$
We can reach any closed
point $q\in \PP(C_Y\widetilde{M})$
                          by the strict transform to $\Bl_Y\widetilde{M}$ of map of a nonsingular
quasi-projective curve
     $$g:(\Delta,p) \rightarrow \widetilde{M}$$
with $g^{-1}(Y)$ supported at p.
In other words,  the strict transform  
$$g_Y: (\Delta,p) \rightarrow  \Bl_Y\widetilde{M}$$
satisfies $g_Y(p)=q$. 
Since $Y\subset X$ is a bijection on closed points, $g^{-1}(X)$ is also supported at $p$.

We work locally on an open affine $U=\text{Spec}(A) \subset \widetilde{M}$ containing
$g(p)\in \widetilde{M}$.
Let $$a_1,...,a_r\in I_X\,, \ \ \ a_1,...,a_r,b_1,...,b_s\in I_Y$$
be generators of the ideals $I_X\subset I_Y\subset A$. 
By definition of the blow-up,
$$\pi_Y^{-1}(U) \subset U\times \proj^{r+s-1}, \ \ \ \pi_X^{-1}(U)\subset U \times \proj^{r-1}\,.$$
Let $t$ be the local parameter of $\Delta$ at $p$ with $t(p)=0$.
From the map $g$, we obtain functions
$$a_i(t)=a_i(g(t))\,,\ \ \ b_j(t)=b_j(g(t))$$
in the local parameter $t$ which are
regular at $0$. 
Since $g(p)\in Y \subset X$, we have $a_i(0)=0$ and $b_j(0)=0$ for all $i$ and $j$.
Let $\ell_X$ be  the {\em lowest} valuation  of  $t$ among
   all the functions $a_i(t)$. Since the $a_i(t)$ can not all vanish
identically,  $\ell_x>0$.
The limit
$$v=\lim_{t\rightarrow 0} \left(\frac{a_1(t)}{ t^{\ell_X}} , \ldots , \frac{a_r(t)}{t^{\ell_X}}\right)$$
is a well-defined {\em nonzero} vector $v$. By definition of the blow-up,
$$g_X(p)= (g(p),[v]) \in U \times \proj^{r-1}\,.$$

Similarly, let $\ell_Y$ be the {lowest} valuation  of  $t$ among
   all the functions $a_i(t)$ {\em and} $b_j(t)$.
Then, the limit
$$w=\lim_{t\rightarrow 0} \left(\frac{a_1(t)}{ t^{\ell_Y}} , \ldots , \frac{a_r(t)}{t^{\ell_Y}},
\frac{b_1(t)}{ t^{\ell_Y}} , \ldots , \frac{b_s(t)}{t^{\ell_Y}}\right)
$$
is a well-defined nonzero vector $w$, and
$$g_Y(p)= (g(p),[w]) \in U \times  \proj^{r+s-1}\,.$$

Certainly, $\ell_Y \leq \ell_X$ since $\ell_Y$ is a minimum over a larger set.
If $\ell_Y < \ell_X$, then there is
a $b_j(t)$ with lower valuation than all the $a_i(t)$.
Such a situation directly contradicts the curvilinear lifting property for
the map
$$\text{Spec}(\mathcal{O}_D/t^{\ell_X})\subset D \stackrel{g}{\longrightarrow} X \,.$$
Hence, $\ell_Y=\ell_X$.

The equality of $\ell_Y$ and $\ell_X$ has the following consequence:
{\em the first $r$ coordinates of $w$ are not all 0}.
As a result, the rational map 
$$f:\Bl_Y\widetilde{M} \xymatrix{\ar@{-->}[r]&} \Bl_X\widetilde{M}$$  
defined on $U\times \proj^{r+s-1}$ by projection
$$f(q) = f((g(p),w))=(g(p),(w_1,\ldots,w_r))= (g(p),v)$$
has no base locus at $q$.
Since $q$ was arbitrary, $f$ has no base locus on $\Bl_Y(\widetilde{M})$.

The exceptional divisors $\PP(C_Y\widetilde{M})$ and 
$\PP(C_X\widetilde{M})$
are $\mathcal{O}(-1)$ on $\Bl_Y(\widetilde{M})$
and $\Bl_X(\widetilde{M})$ respectively. Since the morphism
$f$ respects $\mathcal{O}(-1)$, the relation
$$f^*( \PP(C_X \widetilde{M})) = \PP(C_Y \widetilde{M})\ $$
holds as Cartier divisors.               
\end{proof}

By Proposition \ref{bababa} and the push-pull formula for
the degree 1 morphism $f$, we find
\begin{eqnarray}
\label{gggg4444}
f_* [\PP(C_Y\widetilde{M})] & = & 
f_* [f^*\PP(C_X\widetilde{M})]\\ &=& \text{deg}(f) \cdot\nonumber 
[\PP(C_X\widetilde{M})] \\ \nonumber
& = & [\PP(C_X\widetilde{M})]\,,
\end{eqnarray}
where $[D]$ denotes the fundamental cycle of a Cartier divisor $D$.
The restriction of $f$ to the exceptional divisors yields
a morphism
$$f: \PP(C_Y\widetilde{M}) \rightarrow \PP(C_X\widetilde{M})$$
which covers $\iota:Y \rightarrow X$.
Since the morphism on projective cones respects $\mathcal{O}(1)$,
relation \eqref{gggg4444}  and definition \eqref{ghat}
together 
 imply
$$\iota_* c_F(Y)= c_F(X) \,. $$

In other words, the Fulton total Chern class is the same
for embeddings satisfying the curvilinear lifting property.

\subsection{The divisor $D_{n_1,\beta_1}$}\label{ap33}
Following the notation of Section \ref{rigi}, we have
$$D_{n_1,\beta_1} \subset W(n,\beta)\,, $$
and we would like to compare the Fulton total Chern classes
of these two moduli spaces.
The subspace $D_{n_1,\beta_1}$ is the pull-back to $W(n,\beta)$
of a nonsingular divisor in the Artin stack $\B^p_{n,\beta}$.
Hence, $D_{n_1,\beta_1}$ is locally defined by a single equation.

There is no obstruction to smoothing the crease for 
any stable pair parameterized by $D_{n_1,\beta_1}$.
An elementary argument via vector fields moving points{\footnote{Consider
the versal deformation of the degeneration of
$\proj^1$ to the chain $\proj^1 \cup \proj^1$. Given any finite
collection of nonsingular points of the special
fiber $\proj^1\cup \proj^1$, an open set $U$ of the special fiber can be found
containing the points
together with a vector field which translates $U$ over the
base of the deformation. In the case the degeneration is a longer
chain, such a vector field can be found for each node.  
}}  
in $\proj^1$ shows
$W(n,\beta)$ to be \'etale locally a trivial product of
$D_{n_1,\beta_1}$ with the smoothing parameter in $\com$.
Hence, the equation of $D_{n_1,\beta_1}$ is nowhere
a zero divisor. Given an embedding $W(n,\beta)\subset M$
in a nonsingular ambient space, we consider
$$W(n,\beta) \subset M\times \C = \widetilde{M}$$
with $W(n,\beta)$ lying over $0\in \C$.
There is an exact sequence of cones on $D_{n_1,\beta_1}$,
$$0 \To N \To C_{D_{n_1,\beta_1}}\widetilde{M} 
\To C_{W(n,\beta)}\widetilde{M}|_{D_{n_1,\beta_1}}
\To 0\,,  $$
where $N= \mathcal{O}_{D_{n_1,\beta_1}}(D_{n_1,\beta_1})$.
As a consequence,
$$s(W(n,\beta),M)|_{D_{n_1,\beta_1}} = s(D_{n_1,\beta_1},M)
c(\mathcal{O}_{D_{n_1,\beta_1}}(D_{n_1,\beta_1}))\,.$$
By the definition of the Fulton total Chern class
\begin{eqnarray*}
c_F(W(n,\beta))|_{D_{n_1,\beta_1}} & = & 
c(T_M |_{D_{n_1,\beta_1}}) \cap \
s(W(n,\beta),M)|_{D_{n_1,\beta_1}} \\
& = & c(T_M |_{D_{n_1,\beta_1}}) \cap \
s(D_{n_1,\beta_1},M)
c(\mathcal{O}_{D_{n_1,\beta_1}}(D_{n_1,\beta_1}))
\\
& = & c_F(D_{n_1,\beta_1})
c(\mathcal{O}_{D_{n_1,\beta_1}}(D_{n_1,\beta_1}))\,,
\end{eqnarray*}
which is \eqref{ggg} of Section \ref{rigi}.

\vspace{+14 pt}

\begin{minipage}[position]{65mm}
\noindent
Departement Mathematik\\
ETH Z\"urich\\
{\tt rahul@math.ethz.ch}
\end{minipage}
\begin{minipage}[position]{10cm}
\noindent Department of Mathematics\\
Imperial College\\
{\tt rpwt@imperial.ac.uk}
\end{minipage}


\begin{thebibliography}{99}

\bibitem{AIK} A.~Altman, A.~Iarrobino and S.~Kleiman, \textit{Irreducibility of the compactified Jacobian}, Real and complex singularities (Proc. Ninth Nordic Summer School/NAVF Sympos.~Math., Oslo, 1976) (1977), 1--12.



\bibitem{Beh} K.~Behrend, {\em Gromov-Witten invariants in algebraic geometry}, Invent. Math.
{\bf 127} (1997), 601--617.

\bibitem{BehFan} K.~Behrend and B.~Fantechi,
{\em The intrinsic normal cone}, Invent.\ Math.\
{\bf 128} (1997), 45--88.

\bibitem{beu} A. Beauville, {\em Counting rational curves on $K3$ surfaces},
Duke Math. J. {\bf 97} (1999), 99--108.





\bibitem{borch} R. Borcherds, {\em The Gross-Kohnen-Zagier theorem in higher dimensions},
Duke J. Math. {\bf 97} (1999), 219--233.

\bibitem{Br} T. Bridgeland, {\em  Hall algebras and curve-counting invariants}, J. AMS \textbf{24} (2011), 969--998.


\bibitem{bkl} J. Bryan, S. Katz, and C. Leung, {\em
Multiple covers and the integrality conjecture for rational
curves in Calabi-Yau threefolds}, J. Alg. Geom. {\bf 10} (2001), 
549--568.

\bibitem{brl} J. Bryan and C. Leung, {\em The enumerative geometry of
K3 surfaces and modular forms}, J. AMS {\bf 13} (2000), 371--410.

\bibitem{BuF}
R.-O.\ Buchweitz and H.\ Flenner.
\emph{A semiregularity map for modules and applications to deformations},
Compositio Math. {\bf 137} (2003), 135--210 .

\bibitem{cogp} P. Candelas, X.  de
la Ossa, P.  Green, and L.  Parkes, {\em A pair of Calabi-Yau manifolds
as an exactly soluble superconformal field theory}, Nuclear Physics
{\bf B359} (1991), 21--74.

\bibitem{xc} X. Chen, {\em Rational curves on $K3$ surfaces}, J. Alg.
Geom. {\bf 8} (1999), 245--278.  



\bibitem{dolga} I. Dolgachev and S. Kondo, {\em Moduli of $K3$
surfaces and complex ball quotients}, Lectures in Istambul, math.AG/0511051.

\bibitem{FP}  C. Faber and R. Pandharipande,
\emph{Hodge integrals and Gromov-Witten theory}, Invent.\ Math.\
{\bf 139} (2000), 173--199.

\bibitem{fgd} B. Fantechi, L. G\"ottsche, and Duco van Straten,
{\em Euler number of the compactified Jacobian and multiplicity of
rational curves}, J. Alg. Geom. {\bf 8} (1999), 115--133.

\bibitem{Fu} W.~Fulton, \textit{Intersection theory}, Springer-Verlag (1998).





\bibitem{GV1}  R. Gopakumar and C. Vafa, {\em M-theory and
topological strings I},  hep-th/9809187.

\bibitem{GV2}  R. Gopakumar and C. Vafa, {\em M-theory and
topological strings II}, hep-th/9812127.



\bibitem{GP} T. Graber and R. Pandharipande, {\em Localization
of virtual classes}, Invent. Math. {\bf 135} (1999), 487--518.




\bibitem{HT}
D.~Huybrechts and R.~P.~Thomas.
\newblock {\em Deformation-obstruction
theory for complexes via Atiyah and Kodaira--Spencer classes}, Math. Ann. \textbf{346} (2010), 545--569 .

\bibitem{Ill} L.~Illusie, \textit{Complexe cotangent et d\'eformations I}, Lecture Notes in Math.~\textbf{239} Springer-Verlag (1971).

\bibitem{JS} D.~Joyce and Y.~Song, \textit{A theory of generalized {D}onaldson-{T}homas invariants}, Memoirs of the AMS \textbf{217} (2012), no 1020.

\bibitem{KKP} S.~Katz, A.~Klemm, and R.~Pandharipande, {\em On the motivic
stable pairs invariants of $K3$ surfaces},
with an Appendix by R. Thomas, in {\em K3 surfaces and their
moduli}, C. Faber, G. Farkas, and G. van der Geer, eds.,
Birkhauser Prog. in Math. {\bf 315} (2016),
111--146.

\bibitem{kkv} S. Katz, A. Klemm, C Vafa, {\em M-theory, topological
strings, and spinning black holes}, Adv. Theor. Math. Phys. {\bf 3} (1999),
1445--1537.

\bibitem{ky} T. Kawai and K Yoshioka, {\em String partition functions
and infinite products}, Adv. Theor. Math. Phys. {\bf 4} (2000), 397--485.



\bibitem{KiemLi} Y. H. Kiem and J. Li, {\em Localized virtual cycles by cosections}, J. AMS \textbf{26} (2013), 1025--1050.

\bibitem{germans} A. Klemm, M. Kreuzer, E. Riegler, and E. Scheidegger,
{\em Topological string amplitudes, complete intersections Calabi-Yau spaces,
and threshold corrections}, J. HEP \textbf{05} (2005) 023. 

\bibitem{KMPS} A. Klemm, D. Maulik, R. Pandharipande, and E. Scheidegger,
{\em Noether-Lefschetz theory and the Yau-Zaslow conjecture}, J. AMS \textbf{23} (2010), 1013--1040.

\bibitem{KT1} M.~Kool and R.~P.~Thomas, \textit{Reduced classes and curve counting on surfaces I: theory}, Algebraic Geometry \textbf{1} (2014), 334--383.

\bibitem{KT2} M.~Kool and R.~P.~Thomas, \textit{Reduced classes and curve counting on surfaces II: calculations}, Algebraic Geometry \textbf{1} (2014), 384--399.

\bibitem{kudmil} S. Kudla and J. Millson, {\em Intersection numbers
of cycles on locally symmetric spaces and Fourier coefficients
of holomorphic modular forms in several complex variables},
Pub. IHES {\bf 71} (1990), 121--172.


\bibitem{ll} J. Lee and C. Leung, {\em Yau-Zaslow formula
for non-primitive classes in K3 surfaces}, 
Geom. Topol. \textbf{9} (2005), 1977--2012.



\bibitem{LiRelative}
J.~Li.
\newblock {\em Stable morphisms to singular schemes and relative stable
  morphisms}, J. Differential Geom. {\bf 57} (2001), 509--578.



\bibitem{junli}
J.~Li, 
\emph{
A degeneration formula of Gromov-Witten invariants}, J. Diff. Geom. {\bf 60} (2002), 199--293.



\bibitem{LiWu} J.~Li and B.~Wu, \textit{Good degeneration of Quot-schemes and coherent systems}, Comm. Anal. Geom. \textbf{23} (2015), 841--921. 







\bibitem{gwnl} D. Maulik and R. Pandharipande, {\em 
Gromov-Witten theory and Noether-Lefschetz theory}, in {\em A 
celebration of algebraic geometry}, Clay Mathematics Proceedings {\bf 18},
469--507,
AMS (2010).


\bibitem{MPT} D. Maulik, R. Pandharipande, and R. Thomas, 
{\em Curves on $K3$ surfaces and modular forms}, J. Topol {\bf 3}
(2010), 937--996.


\bibitem{ObPan} G.~Oberdieck and R.~Pandharipande, {\em Curve
counting on $K3 \times E$, the Igusa cusp form $\chi_{10}$, and
descendent integration}, in {\em K3 surfaces and their
moduli}, C. Faber, G. Farkas, and G. van der Geer, eds.,
Birkhauser Prog. in Math. {\bf 315} (2016),
245--278.



\bibitem{PP} R. Pandharipande and A. Pixton, {\em Gromov-Witten/Pairs
correspondence for the quintic 3-fold}, J. AMS (2016), to appear.

\bibitem{PT1}
R.~Pandharipande and R.~P. Thomas,
\newblock {\em Curve counting via stable pairs in the derived category},
Invent. Math. \textbf{178} (2009), 407--447.



\bibitem{PT2} 
R.~Pandharipande and R.~P. Thomas,
\newblock{\em Stable pairs and BPS invariants},
J. AMS \textbf{23} (2010), 267--297.


\bibitem{Pr} J.~Pridham, \textit{Semiregularity as a consequence of Goodwillie's theorem}, arXiv:1208.3111.





\bibitem{ared} T. Sch\"urg, B. To\"en, and G. Vezzosi, {\em Derived algebraic geometry, determinants
of perfect complexes, and applications to obstruction theories for maps and complexes},
arXiv:1102.1150.

\bibitem{Si} B.~ Siebert, \textit{An update on (small) quantum cohomology}, Mirror symmetry, III (Montreal, PQ, 1995), 279-312. AMS/IP Stud. Adv. Math., \textbf{10}, Amer. Math. Soc., Providence, RI, 1999.


\bibitem{Toda} Y. Toda, {\em Curve counting theories via stable objects I. DT/PT correspondence}, J. AMS \textbf{23} (2010), 1119--1157.

\bibitem{toda} Y. Toda, {\em Stable pairs on local $K3$ surfaces}, J. Diff.
Geom. {\bf{92}} (2012), 285--370.


\bibitem{CTC} C. T. C. Wall, {\em On the orthogonal groups of unimodular quadratic forms},
 Math. Ann. {\bf 147} (1962), 328--338.


\bibitem{yauz} S.-T. Yau and E. Zaslow, {\em BPS states, string duality, and
nodal curves on $K3$}, Nucl. Phys. {\bf B457} (1995), 484--512.



 



\end{thebibliography}
\end{document}